\documentclass[review,10pt]{elsarticle}
\usepackage[a4paper,margin=0.75in,right=0.75in]{geometry}
\usepackage{amsthm,amssymb,amsmath,mathalpha,amsfonts,dsfont,mathrsfs}
\allowdisplaybreaks  
\biboptions{sort&compress}

\usepackage{graphicx}   
\usepackage{lipsum}
\usepackage{ragged2e}
\usepackage{hyperref}
\usepackage{enumitem}
\usepackage{float}

\usepackage{xcolor}
\usepackage{bm}
\usepackage{url}
\usepackage{booktabs}         
\usepackage{tabularx}
\usepackage{multirow}
\usepackage{ragged2e}

\usepackage{lineno}
\numberwithin{equation}{section}

\newtheorem{theorem}{Theorem}[section]
\newtheorem{lemma}{Lemma}[section]
\newtheorem{assumption}{Assumption}[section]
\newtheorem{proposition}{Proposition}[section]

\newtheorem{example}{Example}[section]
\newtheorem{corollary}{Corollary}[section]
\newcommand{\dd}{\text{d}}
\newcommand{\mylabel}[2]{#2\def\@currentlabel{#2}\label{#1}}
\makeatother

\begin{document}
\begin{frontmatter}

\title{Strong error analysis and first-order convergence of Milstein-type schemes for McKean-Vlasov SDEs with superlinear coefficients}

\author{Jingtao Zhu}
\ead{bingbll.zhu@gmail.com}

\author{Yuying Zhao\corref{cor1}}
\ead{zhaoyuying78@gmail.com}

\author{Siqing Gan\corref{cor1}}
\ead{sqgan@csu.edu.cn}

\cortext[cor1]{Corresponding authors}
\address{School of Mathematics and Statistics, HNP-LAMA, Central South University, Changsha 410083, China}

\begin{abstract}
In the study of McKean-Vlasov stochastic differential equations (MV-SDEs), numerical approximation plays a crucial role in understanding the behavior of interacting particle systems (IPS). Classical Milstein schemes provide strong convergence of order one under globally Lipschitz coefficients. Nevertheless, many MV-SDEs arising from applications possess super-linearly growing drift and diffusion terms, where classical methods may diverge and particle corruption can occur. In the present work, we aim to fill this gap by developing a unified class of Milstein-type discretizations handling both super-linear drift and diffusion coefficients.
The proposed framework includes the tamed-, tanh-, and sine-Milstein methods as special cases and establishes order-one strong convergence for the associated interacting particle system under mild regularity assumptions, requiring only once differentiable coefficients.
In particular, our results complement Chen et al. (Electron. J. Probab., 2025), where a taming-based Euler scheme was only tested numerically without theoretical guarantees, by providing a rigorous convergence theory within a broader Milstein-type framework. The analysis relies on discrete-time arguments and binomial-type expansions, avoiding the continuous-time \text{It\^o} approach that is standard in the literature. Numerical experiments are presented to illustrate the convergence behavior
and support the theoretical findings.
\end{abstract}




\end{frontmatter}

\section{Introduction}
As a class of important mathematical models,
McKean-Vlasov stochastic differential equations (MV-SDEs), also known as distribution-dependent SDEs, play a central role in modeling interacting particle systems (IPS) and have been extensively studied in fields such as statistical physics \cite{sznitman1991topics, braun1977vlasov}, mathematical finance and mean-field games \cite{lasry2007mean, carmona2018probabilistic}, control theory \cite{bensoussan2013mean}, and neuroscience \cite{baladron2012mean}. In general, closed-form solutions to MV-SDEs are not available and the development and analysis of numerical methods for their simulation are of significant practical interest.
In this paper, we concern the design and convergence analysis of numerical approximations for the following MV-SDEs:
\begin{equation}
\label{eq:MV_SDE}
\dd X_t = f(X_t, \mathscr{L}_{X_t})\, \dd t + g(X_t, \mathscr{L}_{X_t})\, \dd W_t, \quad t \in [0, \mathcal{T}],
\end{equation}
where $\mathscr{L}_{X_{t}}$ denotes the law of $X_{t}$, $f \colon \mathbb{M} \to \mathbb{R}^d$ and $g \colon \mathbb{M} \to \mathbb{R}^{d \times m}$, with $\mathbb{M} := \mathbb{R}^d \times \mathscr{P}_2(\mathbb{R}^d)$ as defined in Section \ref{sec:notations}. The process $\{W_t\}_{0 \leq t \leq \mathcal{T}}$ is an $m$-dimensional standard Brownian motion  on a complete probability space $(\Omega, \mathcal{F}, \{\mathcal{F}_t\}_{0 \leq t \leq \mathcal{T}}, \mathbb{P})$ with a filtration $\{\mathcal{ F}_t\}_{t \in [0,\mathcal{T}]}$.

As the law of $X_t$ is known, the underlying MV-SDEs reduce to the usual SDEs, numerical methods for which have been extensively studied for the past decades (consult monographs \cite{Milstein2004stochastic,hutzenthaler2015numerical,wang2020mean,li2019explicit} and references therein), in the context of both numerical convergence and numerical stability. When the law of $X_t$ is unknown, a widely used approach is to approximate the solution via an IPS and then simulate the particle system using time discretization schemes. Under suitable assumptions, the principle of propagation of chaos (PoC) guarantees that the empirical distribution of the particle system converges to the law of the solution as the number of particles tends to infinity \cite{sznitman1991topics}, the time discretization of the IPS becomes a powerful tool to understand the behavior of the underlying problems. To analyze time discretization approximations of the IPS, a global Lipschitz condition is often imposed on the coefficient functions of MV-SDEs \cite{li2022strong, bao2022approximations}.
Nevertheless, MV-SDEs arising from applications rarely obey such a traditional but restrictive condition.
Notable examples of MV-SDEs with non-globally Lipschitz continuous coefficients include numerous models, such as the double-well dynamics (see \eqref{exam:num-double-well model}), the Cucker-Smale flocking model (see \eqref{eq:num-exam-Cuck-Sma-model} or \cite{erban2016cucker,chen2022flexible}), and the FitzHugh-Nagumo model (see Example \ref{exam:FHN-model-numerical} or \cite{baladron2012mean}).
Evidently, the coefficients of these three models violate the globally Lipschitz condition.
It is well established that even for SDEs \cite{hutzenthaler2011strong, mattingly2002ergodicity} with superlinearly growing coefficients, classical schemes such as Euler-Maruyama and Milstein methods may diverge. In the case of MV-SDEs, analogous divergence phenomena commonly referred to as ``particle corruption" have also been observed and analyzed \cite{goncalo2021simulation}. Specifically, a single particle may become excessively influential over all others, thereby destroying the weakly interacting particle structure underlying the MV-SDE system. 

To address this, a growing body of work has focused on the construction and analysis of convergent Euler-type schemes in non-globally Lipschitz settings (see \cite{goncalo2021simulation,liu2023tamed,huagui2023tamed,chen2022flexible,chen2023euler,kumar2022well,yuanping2025explicit,guo2024convergence,chen2025wellposedness,jian2025modified,christoph2022adaptive,neelima2020well,khue2025infinite,biswas2025milstein,zhu2025euler} and references therein). 
Nevertheless, Euler-type schemes for MV-SDEs with multiplicative noise inherently achieve at most a strong convergence rate of order $\frac{1}{2}$, which often results in high computational costs when higher accuracy is required. 
In contrast, Milstein-type schemes offer order-one strong convergence and can be efficiently combined with multilevel Monte Carlo techniques to reduce computational complexity further \cite{giles2006improved,giles2008multilevel}. Recently, tamed Milstein methods have been successfully applied to MV-SDEs with superlinear drift and Lipschitz diffusion coefficients \cite{bao2021first, kumar2021explicit, bao2023milstein}.
A natural yet unresolved question thus arises as to whether such methods can be extended to settings where both the drift and diffusion coefficients exhibit superlinear growth. To fill this gap, we develop a general framework for a family of Milstein-type schemes designed to handle superlinear drift and diffusion coefficients and establish their strong convergence theory. Compared with \cite{bao2021first}, our analysis only requires once differentiability of the coefficients in the state and measure components.
Specifically, our scheme modifies the drift and diffusion coefficients through bounded operators that scale with a negative power of the time step size $\Delta t$, leading to the following general form:
\begin{align}\label{intro-eq:modified_Milstein_method}
   Y_{t_{k+1}}^{i, N} =&~ Y_{t_{k}}^{i,N} + \Gamma_{1}\!\left(f \!\left(Y_{t_{k}}^{i, N}, \rho_{t_{k}}^{Y,N} \right),\Delta t \right) \Delta t  
   + \sum_{j_{1}=1}^{m} \Gamma_{2}\!\left(g_{j_{1}} \!\left(Y_{t_{k}}^{i, N}, \rho_{t_{k}}^{Y, N} \right), \Delta t\right) \Delta W^{i,j_{1}}_{t_{k}}  \notag \\
  &~+\sum_{j_{1},j_{2}=1}^{m} \Gamma_{3}\!\left(\mathcal{L}_{y}^{j_{2}}~g_{j_{1}}\!\left(Y_{t_{k}}^{i,N}, \rho_{t_{k}}^{Y, N} \right),\Delta t\right) \!\int_{t_{k}}^{t_{k+1}} \!\!\int_{t_{k}}^{s} \mathrm{d} W_{r}^{i,j_{2}}  \mathrm{d} W^{i,j_{1}}_{s}  \\
  &~+\frac{1}{N}\sum_{j_{1},j_{2}=1}^{m} \sum_{k_{1}=1}^{N} \Gamma_{4}\!\left( \mathcal{L}_{\rho}^{j_{2}}~g_{j_{1}} \!\left(Y_{t_{k}}^{i,N},\rho_{t_{k}}^{Y,N},Y_{t_{k}}^{k_{1},N}\right) ,\Delta t\right)
  \!\int_{t_{k}}^{t_{k+1}}  \!\int_{t_{k}}^{s} \mathrm{d} W_{r}^{k_{1},j_{2}}\mathrm{d} W^{i,j_{1}}_{s}. \notag
\end{align}
Here, the operators $\Gamma_i(\cdot, \Delta t)$, $i=1,\dots,4$, are uniformly bounded and designed to control the superlinear growth of the coefficients through their dependence on $\Delta t$. Moreover, this framework encompasses several schemes as special cases, including the tamed, sine-, and tanh-type Milstein methods (see Examples \eqref{ex-eq:tanh-misltein}–\eqref{ex:tamed-euler-scheme}) and also recovers the taming-in Euler scheme in \cite{chen2025wellposedness} as a degenerate instance, where the scheme was
only tested numerically without theoretical guarantees. In particular, the tamed Milstein scheme considered in \cite{kumar2022well} for MV-SDEs with common noise under similar growth conditions is recovered as a specific example within our framework, corresponding to special choices of the operators $\Gamma_i$.

Indeed, when treating MV-SDEs with superlinearly growing coefficients by taming strategy, it becomes a standard way in the literature to work with continuous-time extensions of the schemes and carry out the analysis based on the \text{It\^o} formula (see, e.g. \cite{neelima2020well,kumar2022well,biswas2025milstein}). In this work, non-standard arguments are developed to overcome these difficulties in the analysis. For example, to obtain high-order moments of the schemes, we work directly with the discrete scheme \eqref{intro-eq:modified_Milstein_method} and rely on discrete arguments based on binomial expansion (see the proof of Lemma \ref{lem:moment-bounds-of-numerical-solution}). 

In summary, by generalizing the taming strategy through operator modifications, we develop a unified class of Milstein-type schemes handling both superlinear drift and diffusion coefficients and establish their order-one strong convergence.  
The main contributions are summarized as follows:
\begin{itemize}
\item 
A unified framework of Milstein-type discretizations is developed for the IPS associated with nonlinear MV-SDEs. The framework covers the tamed, sine-, and tanh-type Milstein methods (\eqref{ex-eq:tanh-misltein}–\eqref{ex:tamed-euler-scheme}) as special cases and also recovers the taming-in Euler scheme as a degenerate instance, thereby complementing \cite{chen2025wellposedness}, where the taming algorithm was only numerically implemented without theoretical guarantees.
\item 
Building on the proposed framework, we establish moment bounds and prove order-one strong $L^p$-convergence of the proposed Milstein-type schemes for the IPS associated with nonlinear MV-SDEs.
These results hold under mild conditions, requiring only that the modified operators $\Gamma$ are uniformly bounded and scale with a negative power of the time step $\Delta t$.
The analysis relies on discrete-time arguments and a mean value-type technique rather than the continuous-time \text{It\^o} approach, assuming only once differentiability of the coefficients in both the state and measure components.
\end{itemize}

To apply the theoretical results to a specific numerical scheme, it is sufficient to verify two main conditions. \textit{First}, given Assumptions \ref{ass:Gamma-control-conditions} and \ref{ass:Coefficient-comparison-conditions-of-f}, moments of the numerical solution are bounded. \textit{Second}, under Assumption \ref{ass:Coefficient-comparison-conditions-of-Gamma1-Gamma4} on the errors between the operators and the original functions, we obtain from Theorem \ref{thm:convergence-result-particle-scheme} that the proposed Milstein-type schemes converge with strong order one or smaller, depending on the choice of operators, particularly, concrete example verifications provided in Section \ref{subsec:examples_of_scheme}. For clarity, Table \ref{tab:parameter-overview} in \ref{appendix:summary-of-parameters} summarizes the key parameters involved and their roles in applying the theoretical framework.  

As already mentioned above, the analysis of both the boundedness and the strong convergence rate is highly non-trivial and essential difficulties are caused by the non-globally Lipschitz setting, where the diffusion coefficients are allowed to grow polynomially.

The remainder of this paper is structured as follows. We introduce key notations and assumptions in Section \ref{sec:ass-and-chao-propa}, where we also present the relevant PoC results.
In Section \ref{sec:Mil-sche-and-moment-bound}, we present the proposed Milstein-type schemes and derive uniform moment bounds under relaxed conditions.
Section \ref{sec:numeri-experi} provides numerical experiments that illustrate the theoretical results.
All auxiliary estimates used in the proof of Theorem \ref{thm:convergence-result-particle-scheme} are collected in Section \ref{sec:conver-analy}, and additional details are given in the Appendix. A brief conclusion closes the paper.

\section{Preliminaries and N-particle descretization} 
\label{sec:ass-and-chao-propa}
\subsection{Notation and assumptions}
\label{sec:notations}
We denote by $\langle \cdot, \cdot \rangle$ the standard inner product in $\mathbb{R}^d$ and by $|\cdot|$ the associated Euclidean norm. For a matrix $A \in \mathbb{R}^{d \times m}$, we define the (Frobenius) norm as
$|A| := \sqrt{\operatorname{trace}(A^T A)}$, where $A^T$ denotes the transpose of $A$. Throughout, we use $\mathbb{E}$ to denote expectation on a filtered probability space $(\Omega, \mathcal{F}, \{\mathcal{F}_t\}_{0 \leq t \leq \mathcal{T}}, \mathbb{P})$. For $q \ge 1$, let $\mathscr{P}_q(\mathbb{R}^d)$ denote the space of Borel probability measures on $\mathbb{R}^d$ with finite $q$-th moment, i.e.
$$
\mathscr{P}_q(\mathbb{R}^d):=\{\rho \in \mathscr{P}\left(\mathbb{R}^d\right)\colon \rho\left(|\cdot|^q\right)=\int_{\mathbb{R}^d}|x|^q \rho(\mathrm{d} x)<\infty\}.
$$
Based on this, we set $\mathbb{M} := \mathbb{R}^d \times \mathscr{P}_2(\mathbb{R}^d)$, $\bar{\mathbb{M}} := \mathbb{R}^d \times \mathscr{P}_2(\mathbb{R}^d) \times \mathbb{R}^d$. 

For $\rho, \bar{\rho} \in \mathscr{P}_q(\mathbb{R}^d)$, the $L^q$-Wasserstein distance is defined by
$$
\mathbb W_q(\rho, \bar{\rho}) = \inf_{\pi \in \Pi(\rho, \bar{\rho})} \left( \int_{\mathbb R^d \times \mathbb R^d}  |x - y|^q d\pi(x, y) \right)^{\frac{1}{q}},
$$
where $\Pi(\rho, \bar{\rho})$ is the collection of all probability measures on $\mathbb R^d \times \mathbb R^d$ with its marginals agreeing with $\rho$ and $\bar{\rho}$.
  
For convenience, we write $\mathcal{I}_N :=\{1, 2, \dots, N\}$. The notation $\delta_y$ denotes the Dirac measure at point $y \in \mathbb{R}^d$. We also use the notation $a_1 \vee b_1 := \max\{a_1, b_1\}$ and $a_1 \wedge b_1 = \min\{a_1, b_1\}$ for any $a_1, b_1 \in \mathbb{R}$. 

We next introduce some assumptions on the coefficients for MV-SDEs \eqref{eq:MV_SDE}.
\begin{assumption}
\label{ass:assumptions_for_MV_coefficients}
 We assume the following hold uniformly for all $y, \bar{y} \in \mathbb{R}^d$ and $\rho, \bar{\rho} \in \mathscr{P}_2(\mathbb{R}^d)$:
\begin{itemize}
[label=\textnormal{(\arabic*)}]
    \item[(A1)] $\mathbb E\left[\left|X_0\right|^{\bar{p}}\right] < \infty$, where $\bar{p} > 2$ is fixed.
    \item [(A2)] The functions $f$ and $g$ meet the following coercivity condition: $$2 \langle y, f(y, \rho) \rangle + (\bar{p} - 1) |g(y, \rho)|^2 \leq C \left( 1 + |y|^2 + \mathbb{W}_2^2 (\rho, \delta_0) \right),$$
    for some constant $C>0$.
    \item[(A3)] The functions $f$ and $g$ satisfy coupled monotonicity condition:
    $$2\langle y - \bar{y}, f(y, \rho) - f(\bar{y}, \bar{\rho}) \rangle + p^{\ast}|g(y, \rho) - g(\bar{y}, \bar{\rho}) |^2 \leq C \left(|y-\bar{y}|^2 + \mathbb{W}_2^2 (\rho, \bar{\rho}) \right),$$
    for some constants $p^{\ast}>1$ and $C>0$.
    \item[(A4)] The function $f(y, \rho)$ is uniformly continuous in $y$.
\end{itemize}
\end{assumption}

Under Assumption \ref{ass:assumptions_for_MV_coefficients}, which ensures the existence, uniqueness and moment boundedness of the solution to \eqref{eq:MV_SDE} (cf. \cite[Theorem 2.1]{kumar2022well} with common noise), we deduce from \textit{(A2)} that
$$
\sup_{0 \leq t \leq \mathcal{T}} \mathbb{E} \left[ \left| X_{t} \right|^{p} \right] \le C\left(1+\mathbb{E}\left[|X_{0}|^{\bar{p}}\right]\right), \quad p \in [1, \bar{p}].
$$

\subsection{N-particle discretization}
One of the main difficulties in simulating the MV-SDEs \eqref{eq:MV_SDE} is that the exact distribution $\mathscr{L}_{X_t}$ of the solution is generally not accessible for all $t \geq 0$ in practice. Instead, for all $t \in [0, \mathcal{T}]$ and $i\in \mathcal{I}_{N}$, we approximate $\mathscr{L}_{X_t}$ using an interacting particle system (IPS):
\begin{equation}
\label{eq:interact_particle_system}
\dd X_t^{i, N} = f (X_t^{i, N}, \rho_t^{X, N} ) \, \dd t + \sum_{j=1}^{m} g_{j} ( X_t^{i, N}, \rho_t^{X, N}) \, \dd W_{t}^{i,j}, \quad \rho_t^{X, N} (\cdot) := \frac{1}{N} \sum_{i = 1}^{N} \delta_{X_t^{i, N}} (\cdot).
\end{equation}
Here, $g_j$ stands for its $j$-th column vector and $( W^{i}_{t}, X_0^{i})$ are independent copies of the pair $(W_t, X_0)$.

Under Assumption \ref{ass:assumptions_for_MV_coefficients}, the IPS \eqref{eq:interact_particle_system}, interpreted as an $\mathbb{R}^{d \times N}$-valued SDE, admits a unique strong solution with uniform moment bounds up to order $\bar{p}$ \footnote{For classical SDEs, see \cite{Gyongy01011980}; for MV-SDEs and their particle approximations, see \cite{kumar2022well}.}.
That is, there exists $C > 0$ such that for all $p \in [1, \bar{p}]$,
\begin{equation}
\label{remak:well_posedness_interact_particle}
\sup_{ t \in[0,\mathcal{T}]}\sup_{i\in\mathcal{I}_{N}} \mathbb{E} \left[ \left| X^{i,N}_{t} \right|^{p} \right] \le C\left(1+\mathbb{E}\left[|X_{0}^{i}|^{\bar{p}}\right]\right).
\end{equation}
Here $\bar{p}$ is from \textit{(A2)} of Assumption \ref{ass:assumptions_for_MV_coefficients}.

The IPS \eqref{eq:interact_particle_system} provides an approximation to the MV-SDE \eqref{eq:MV_SDE} as $N \to \infty$, a fact justified by the propagation of chaos (PoC) principle. This principle characterizes the asymptotic independence of the particles and the convergence of their empirical distribution toward the law of the solution to MV-SDE \eqref{eq:MV_SDE}.

To formalize this result, consider the corresponding system of non-interacting particles (NIPS) given by
\begin{equation}
\label{eq:non_interacting_particles_system}
\dd X^{i}_t =  f \left(X^{i}_t,  \mathscr{L}_{X^{i}_t} \right) \, \dd t + \sum_{j=1}^{m} g_{j} \left( X^{i}_t,  \mathscr{L}_{X^{i}_t}  \right) \, \dd W^{i,j}_{t},
\end{equation}
for all $t \in [0, \mathcal{T}]$ and $i \in \mathcal{I}_{N}$. 
Here ${X_t^i}$ are i.i.d. copies of $X_t$\footnote{That is, $\mathscr{L}_{X^{i}_t} = \mathscr{L}_{X_t}$ for all $t \in [0, \mathcal{T}]$, assuming the MV-SDE has a unique solution.}.

The following result, adapted from \cite[Proposition 1]{kumar2022well}, quantifies the convergence of the IPS \eqref{eq:interact_particle_system} to the MV-SDEs \eqref{eq:MV_SDE} in terms of the number of particles $N$.

\begin{proposition}
\label{prop:propagation_of_chaos}
(PoC, \cite[Proposition 1]{kumar2022well})  
Let Assumption \ref{ass:assumptions_for_MV_coefficients} hold with $\bar{p} > 4$.  
Define
\[
\eta_d(N) :=
\begin{cases} 
N^{-1/2}, & d < 4, \\[2pt]
N^{-1/2} \log N, & d = 4, \\[2pt]
N^{-2/d}, & d > 4 .
\end{cases}
\]
Then
\[
\sup_{t\in[0,\mathcal{T}] } \sup_{i \in \mathcal{I}_{N}} 
\mathbb{E} \left[ \left\vert X_{t}^{i} - X_{t}^{i, N} \right\vert^{2} \right] 
\leq C\, \eta_d(N),
\]
where $C > 0$ is independent of $N$.
\end{proposition}
The PoC is crucial to establishing a strong convergence analytical framework for the MV-SDE. 
\section{Strongly convergent Milstein-type schemes for IPS} \label{sec:Mil-sche-and-moment-bound}
We now present the considered time discretization for IPS \eqref{eq:interact_particle_system}. In the sequel, we adopt a uniform time discretization over $[0, \mathcal{T}]$ using a step size $\Delta t = \frac{\mathcal{T}}{n}, n\in\mathbb{N}_{+}$. 
The discrete time points are then given by $t_k = k\Delta t$ for $k = 0, 1, \dots, n$.

\subsection{%
  \texorpdfstring{Milstein-type schemes and the maps $\Gamma_l$}{Milstein-type schemes and the maps Gamma l}%
}
Recall the Milstein method for SDEs, when applied to each particle in IPS \eqref{eq:interact_particle_system} associated with the MV-SDE \eqref{eq:MV_SDE}, leads to the following scheme:
\begin{align} \label{eq:classi_Milstein_method}
   \bar{Y}_{t_{k+1}}^{i, N} =&~ \bar{Y}_{t_{k}}^{i,N} + f \left(\bar{Y}_{t_{k}}^{i, N}, \rho_{t_{k}}^{\bar{Y},N} \right) \Delta t  + \sum_{j_{1}=1}^{m} g_{j_{1}} \left(\bar{Y}_{t_{k}}^{i, N}, \rho_{t_{k}}^{\bar{Y}, N} \right) \Delta W^{i,j_{1}}_{t_{k}}  \notag \\
  &~+\sum_{j_{1},j_{2}=1}^{m} \mathcal{L}_{y}^{j_{2}}~g_{j_{1}}\left(\bar{Y}_{t_{k}}^{i,N}, \rho_{t_{k}}^{\bar{Y}, N} \right) \int_{t_{k}}^{t_{k+1}} \int_{t_{k}}^{s} \mathrm{d} W_{r}^{i,j_{2}}  \mathrm{d} W^{i,j_{1}}_{s}  \notag \\
  &~+\frac{1}{N}\sum_{j_{1},j_{2}=1}^{m} \sum_{k_{1}=1}^{N}  \mathcal{L}_{\rho}^{j_{2}}~g_{j_{1}} \left(\bar{Y}_{t_{k}}^{i,N},\rho_{t_{k}}^{\bar{Y},N},\bar{Y}_{t_{k}}^{k_{1},N}\right) \int_{t_{k}}^{t_{k+1}}  \int_{t_{k}}^{s} \mathrm{d} W_{r}^{k_{1},j_{2}}\mathrm{d} W^{i,j_{1}}_{s},
\end{align}

where
\begin{align*}
\Delta W^{i,j_1}_{t_k} := &~ W^{i,j_1}_{t_{k+1}} - W^{i,j_1}_{t_k}, \notag \\
\mathcal{L}_{y}^{j_{2}}g_{j_{1}}\left(\bar{Y}_{t_{k}}^{i,N},\rho_{t_{k}}^{\bar{Y},N}\right) := &~\partial_{y}g_{j_{1}}\left(\bar{Y}_{t_{k}}^{i,N},\rho_{t_{k}}^{\bar{Y},N}\right)g_{j_{2}}\left(\bar{Y}_{t_{k}}^{i,N},\rho_{t_{k}}^{\bar{Y},N}\right), \notag \\
\mathcal{L}_{\rho}^{j_{2}}g_{j_{1}}\left(\bar{Y}_{t_{k}}^{i,N},\rho_{t_{k}}^{\bar{Y},N},\bar{Y}_{t_{k}}^{k_{1},N}\right) := &~\partial_{\rho}g_{j_{1}}\left(\bar{Y}_{t_{k}}^{i,N},\rho_{t_{k}}^{\bar{Y},N},\bar{Y}_{t_{k}}^{k_{1},N}\right) g_{j_{2}} \left(\bar{Y}_{t_{k}}^{k_{1},N},\rho_{t_{k}}^{\bar{Y},N}\right). \footnotemark
\end{align*}
\footnotetext{Here, $\partial_y g_j$ is classical gradient with respect to the state variable $y$ and $\partial_{\rho}g(x,\rho,z)$ denotes the Lions derivative of $g$ with respect to $\rho$, as introduced in \ref{appen:differen_measure}.}

Under standard assumptions such as Lipschitz continuity of the coefficients $f$ and $g$, the scheme achieves strong convergence of order one \cite{Milstein2004stochastic}. However, in the presence of superlinear (e.g., polynomial) growth and a lack of Lipschitz regularity, the method, like the Euler scheme, can blow up in finite time \cite{hutzenthaler2011strong}.
Inspired by the Milstein method, we propose the following Milstein-type schemes approximating the IPS \eqref{eq:interact_particle_system}:
\begin{align} \label{eq:modified_Milstein_method}
   Y_{t_{k+1}}^{i, N} =& Y_{t_{k}}^{i,N} + \Gamma_{1}\left(f \left(Y_{t_{k}}^{i, N}, \rho_{t_{k}}^{Y,N} \right),\Delta t \right) \Delta t  + \sum_{j_{1}=1}^{m} \Gamma_{2}\left(g_{j_{1}} \left(Y_{t_{k}}^{i, N}, \rho_{t_{k}}^{Y, N} \right), \Delta t\right) \Delta W^{i,j_{1}}_{t_{k}}  \notag \\
  &~+\sum_{j_{1},j_{2}=1}^{m} \Gamma_{3}\left(\mathcal{L}_{y}^{j_{2}}~g_{j_{1}}\left(Y_{t_{k}}^{i,N}, \rho_{t_{k}}^{Y, N} \right),\Delta t\right) \int_{t_{k}}^{t_{k+1}} \int_{t_{k}}^{s} \mathrm{d} W_{r}^{i,j_{2}}  \mathrm{d} W^{i,j_{1}}_{s}  \notag \\
  &~+\frac{1}{N}\sum_{j_{1},j_{2}=1}^{m} \sum_{k_{1}=1}^{N} \Gamma_{4}\left( \mathcal{L}_{\rho}^{j_{2}}~g_{j_{1}} \left(Y_{t_{k}}^{i,N},\rho_{t_{k}}^{Y,N},Y_{t_{k}}^{k_{1},N}\right) ,\Delta t\right)\int_{t_{k}}^{t_{k+1}}  \int_{t_{k}}^{s} \mathrm{d} W_{r}^{k_{1},j_{2}}\mathrm{d} W^{i,j_{1}}_{s},
\end{align}
where $\{\Gamma_l(\cdot)\}_{l=1}^4$ are taming operators approximating $\cdot$, and the initial values satisfy $Y_0^{i,N} = Y_0^i$ for $\{Y_0^i\}_{i=1}^N \stackrel{\text{i.i.d.}}{\sim} Y_0$ and $Y_0 \sim \mathscr{L}_{X_0}$.

\subsection{%
  \texorpdfstring{Moment bounds of the Milstein-type schemes \eqref{eq:modified_Milstein_method}}%
  {Moment bounds of the Milstein-type schemes}%
}
\label{sec:moment_bounds_Milstein}
To establish moment bounds for the Milstein-type schemes \eqref{eq:modified_Milstein_method}, we first present the necessary structural assumptions on the coefficients and maps $\{\Gamma_l(\cdot)\}_{l=1}^4$. 

\begin{assumption} 
\label{ass:poly-f-initial}
\noindent \mylabel{ass:Initial-value}{(H1)}
    There exist constant $C>0$ such that 
    $$
    |f(0,\delta_{0})|\vee |g(0,\delta_{0})| \vee |\partial_{y}f(0,\delta_{0})|\vee |\partial_{\rho}f(0,\delta_{0},0)| \vee |\partial_{y}g_{j}(0,\delta_{0})|\vee|\partial_{\rho}g_{j}(0,\delta_{0},0)| \le C.
    $$
\mylabel{ass:polynomial-growth-of-f}{(H2)}
There exist $\gamma\ge 2$ and $C>0$ such that for all $y,\bar{y}\in \mathbb{R}^{d}$, $\rho,\bar{\rho}\in \mathscr{P}_{2}(\mathbb{R}^{d})$,
    \begin{align}
    |f(y, \rho) - f(\bar{y}, \bar{\rho}) | \leq C \left( \left( 1 + |y|^{\gamma} + |\bar{y}|^{\gamma} \right) |y - \bar{y}| + \mathbb{W}_2 (\rho, \bar{\rho}) \right).
    \end{align}
\end{assumption}
The conditions \textit{(H1)} and \textit{(H2)} imply that
\begin{align}   
\label{ineq:growth-condition-of-f}
    \left|f(y,\rho)\right| 
    &~\le \left|f(y,\rho)-f(0,\delta_{0})\right| + \left|f(0,\delta_{0})\right|  
    \le C\left(1+|y|^{\gamma+1}+ \mathbb{W}_{2}\left(\rho,\delta_{0}\right)\right).
\end{align}
Based on \textit{(A3)} of Assumption \ref{ass:assumptions_for_MV_coefficients}, together with the Cauchy-Schwarz inequality, we obtain
\begin{equation}
\label{ineq:polynomial-growth-of-g}   
    \left|g(y,\rho)-g(\bar{y},\bar{\rho})\right|^{2} 
    \le C \left(\left(1+|y|^{\gamma}+|\bar{y}|^{\gamma }\right)\left|y-\bar{y}\right|^{2}+\mathbb{W}_{2}^{2}(\rho,\bar{\rho})\right),
\end{equation}  
which implies the polynomial growth 
\begin{equation}
\label{ineq:growth-condition-of-g}
\left|g(y,\rho)\right| \leq C\big(1+|y|^{\frac{\gamma}{2}+1}+ \mathbb{W}_{2}\left(\rho,\delta_{0}\right)\big). 
\end{equation}

To handle the higher-order terms in the Milstein-type schemes, we further assume polynomial growth conditions on the differential operators
 $\partial_{y}f$, $\partial_{y}g_{j}$, $\partial_{\rho}f$ and $\partial_{\rho}g_{j}$ (for $j=1,2,\cdots,m$).
 
\begin{assumption} \label{ass:Derivative-of-f-and-g-with-respect-to-y-mu}

For all $y, \bar{y}, z, \bar{z} \in \mathbb{R}^d$ and $\rho, \bar{\rho} \in \mathscr{P}_{2}(\mathbb{R}^d)$, the derivatives satisfy:
\begin{align*}
    \left|\partial_{y}f(y,\rho)-\partial_{y}f(\bar{y},\bar{\rho})\right| &~\leq C \left( \left( 1 + |y|^{\gamma-1} + |\bar{y}|^{\gamma-1} \right) |y - \bar{y}| + \mathbb{W}_2 (\rho, \bar{\rho}) \right), \\
    \left|\partial_{y}\,g_{j}(y,\rho)-\partial_{y}\,g_{j}(\bar{y},\bar{\rho})\right| &~\leq C \Big( \big( 1 + |y|^{\frac{\gamma}{2}-1} + |\bar{y}|^{\frac{\gamma}{2}-1} \big) |y - \bar{y}| + \mathbb{W}_2 (\rho, \bar{\rho}) \Big), \\
    \left|\partial_{\rho} f(y,\rho,z)-\partial_{\rho}f(\bar{y},\bar{\rho},\bar{z})\right| &~\leq C \left( \left( 1 + |y|^{\gamma} + |\bar{y}|^{\gamma} \right) |y - \bar{y}| +|z-\bar{z}| + \mathbb{W}_2 (\rho, \bar{\rho}) \right), \\
    \left|\partial_{\rho}\,g_{j}(y,\rho,z)-\partial_{\rho}\,g_{j}(\bar{y},\bar{\rho},\bar{z})\right| &~\leq C \Big( \big( 1 + |y|^{\frac{\gamma}{2}} + |\bar{y}|^{\frac{\gamma}{2} } \big) |y - \bar{y}| +|z-\bar{z}|+ \mathbb{W}_2 (\rho, \bar{\rho}) \Big),
\end{align*}
with some constants $C >0, \gamma\ge 2$.
\end{assumption} 
Similar arguments in \eqref{ineq:growth-condition-of-f} and \eqref{ineq:growth-condition-of-g}, these conditions yield the growth bounds:
\begin{align}
    |\partial_{y}f(y,\rho)| \le &~ C \left(1+|y|^{\gamma}+ \mathbb{W}_{2}\left(\rho,\delta_{0}\right)\right), \label{ineq:growth-condition-of-f-y}\\ 
    |\partial_{y}\,g_{j}(y,\rho)| \le &~ C \left(1+|y|^{\frac{\gamma}{2}}+ \mathbb{W}_{2}\left(\rho,\delta_{0}\right)\right), \label{ineq:growth-condition-of-g-y}\\
    |\partial_{\rho} f(y,\rho,z)|  \le  &~ C\left(1+|y|^{\gamma+1}+|z|+ \mathbb{W}_{2}\left(\rho,\delta_{0}\right)\right), \label{ineq:growth-condition-of-f-mu}\\
    |\partial_{\rho}\,g_{j}(y,\rho,z)| \le &~C\left(1+|y|^{\frac{\gamma}{2}+1}+|z|+ \mathbb{W}_{2}\left(\rho,\delta_{0}\right)\right).\label{ineq:growth-condition-of-g-mu}
\end{align}

In addition to the conditions on the coefficients, we introduce assumptions on the operators $\{\Gamma_l(\cdot)\}_{l=1}^4$, which are applied to the mappings $\mathds{F}_l$ corresponding to $f$, $g_j$, $\mathcal{L}_{y}^{j_2} g_{j_1}$, and $\mathcal{L}_{\rho}^{j_2} g_{j_1}$.
\begin{assumption} \label{ass:Gamma-control-conditions}
There exist $C, \zeta_{l}>0 ~(l=1,2,3,4)$ such that
\begin{align*}
\big|\Gamma_{l}\big(\mathds{F}_l(\cdot),\Delta t\big)\big| &\le C\Delta t^{-\zeta_{l}} \wedge |\mathds{F}_l(\cdot)|.
\end{align*} 
\end{assumption}

\begin{assumption} \label{ass:Coefficient-comparison-conditions-of-f}
There exist $ C, r_{1}, r_{2}>0 $ such that
\begin{align*}
\left|\Gamma_{1}(f,\Delta t)-f\right| \le & C \Delta t ^{r_{1}} \left|f\right|^{r_{2}}.
\end{align*}
\end{assumption}
The Assumption \ref{ass:Coefficient-comparison-conditions-of-f} can be viewed as a consistency requirement for the operator $\Gamma_{1}$, ensuring that it approximates $f$ with an error that vanishes as $\Delta t \to 0$ for any fixed $f \in \mathbb{R}^d$. In other words, $\Gamma_{1}(f, \Delta t) \to f$ as $\Delta t \to 0$, and the bound in Assumption \ref{ass:Coefficient-comparison-conditions-of-f} quantifies both the rate of convergence (via $r_{1}$) and its possible growth with respect to $\lvert f \rvert$ (via $r_{2}$). 

To establish the strong convergence order of the Milstein-type schemes \eqref{eq:modified_Milstein_method}, it is necessary to derive moment bounds for the numerical solutions.
\begin{lemma} \label{lem:moment-bounds-of-numerical-solution}
Suppose Assumptions \ref{ass:assumptions_for_MV_coefficients} and \ref{ass:Initial-value}-\ref{ass:Coefficient-comparison-conditions-of-f} hold. Then there exist $C,\beta>0$, such that 
\begin{align} \label{esti:moment_boundness}
\sup_{i\in\mathcal{I}_{N}} \mathbb{E}\left[|Y_{t_{k}}^{i,N}|^{p}\right] \leq C\bigg(1+\Big(\mathbb{E}\Big[\big|Y_{0}\big|^{\bar{p}}\Big]\Big)^{\beta}\bigg), \quad p\in \left[2, \frac{\bar{p}-\Theta}{1+ \bar{\zeta}\Theta}\right], \quad  k=0,1,\cdots,n,
\end{align}
where $\Theta:=\Theta(\gamma,r_{1},r_{2})=\frac{r_{2}(\gamma+1)-1}{r_{1}}\vee 3\gamma $ and $\bar{\zeta} = \zeta_{1} \vee \left(\zeta_{2}+\frac{1}{2}\right) \vee \zeta_{3} \vee \zeta_{4}$. Here $\gamma, \{\zeta_{l}\}_{l=1}^{4}, r_{1}, r_{2}$ come from \textit{(H2)}, Assumptions  \ref{ass:Gamma-control-conditions} and \ref{ass:Coefficient-comparison-conditions-of-f}. Additionally, $\bar{p}\ge 2+(2\bar{\zeta}+1)\Theta$.
\end{lemma}
The proof of this lemma follows by extending the techniques developed in \cite{ZHANG20171,zhao2024weakerroranalysisstrong,jian2025modified};
a detailed argument is provided in  \ref{appen:proof-moment-bound-modi-method}.

\subsection{%
  \texorpdfstring{Strong convergence of order $1$ for Milstein-type schemes \eqref{eq:modified_Milstein_method}}%
  {Strong convergence of order 1 for Milstein-type schemes}%
}
\label{sec:conver_rate_Milstein}

In order to obtain the strong convergence order of the Milstein-type scheme \eqref{eq:modified_Milstein_method} for the IPS defined by \eqref{eq:interact_particle_system}, we consider the continuous-time formulation of the scheme. For $s \in [0, \mathcal{T}]$, define  $\tau_n(s) := \frac{\lfloor ns \rfloor}{n} = \sup \{t_{k}\in\{t_{l}\}_{l=0}^{n}:  t_{k}\leq s\}$ which represents the piecewise constant projection of $s$. The continuous-time Milstein-type approximation corresponding to \eqref{eq:modified_Milstein_method} then takes the form
\begin{equation} \label{eq:continuous-version-of-MMS}
    Y_{t}^{i,N} = Y_{0}^{i,N} + \int_{0}^{t}\Gamma_{1}\Big(f\Big(Y_{\tau_{n}(s)}^{i,N},\rho_{\tau_{n}(s)}^{Y,N}\Big),\Delta t\Big) \,\mathrm{d} s  + \sum_{j_{1}=1}^{m} \int_{0}^{t} \Gamma_{g} \Big(\hat{g}_{j_{1}}\Big(s,Y_{\tau_{n}(s)}^{i,N},\rho_{\tau_{n}(s)}^{Y,N}\Big),\Delta t \Big) \,\mathrm{d} W^{i,j_{1}}_{s} ,        
\end{equation}
where ${\Gamma}_{g}(\cdot)$ is defined as
\begin{equation}  \label{eq:Gamma_g}
\begin{aligned}  
&\Gamma_{g} \Big(\hat{g}_{j_{1}}\Big(s,Y_{\tau_{n}(s)}^{i,N},\rho_{\tau_{n}(s)}^{Y,N}\Big),\Delta t \Big)  \\
:=&~ \Gamma_{2}\Big(g_{j_{1}} \Big(Y_{\tau_{n}(s)}^{i, N}, \rho_{\tau_{n}(s)}^{Y, N} \Big), \Delta t\Big) +\sum_{j_{2}=1}^{m} \int_{\tau_{n}(s)}^{s} \Gamma_{3}\Big(\mathcal{L}_{y}^{j_{2}}~g_{j_{1}}\Big(Y_{\tau_{n}(r)}^{i, N}, \rho_{\tau_{n}(r)}^{Y, N} \Big),\Delta t\Big)\,   \mathrm{d}W^{i,j_{2}}_{r}    \\
&+\frac{1}{N}\sum_{j_{2}=1}^{m} \sum_{k_{1}=1}^{N} \int_{\tau_{n}(s)}^{s} \Gamma_{4}\Big( \mathcal{L}_{\rho}^{j_{2}}~g_{j_{1}} \Big(Y_{\tau_{n}(r)}^{i,N},\rho_{\tau_{n}(r)}^{Y,N},Y_{\tau_{n}(r)}^{k_{1},N}\Big) ,\Delta t\Big)   \, \mathrm{d}W^{k_{1},j_{2}}_{r}.
\end{aligned}
\end{equation}

To ensure convergence, we impose the following condition on the difference between each approximation operator $\Gamma_l$ and its corresponding original function $\mathds{F}_l$, where $\mathds{F}_l$ stands for $f,g$, $\mathcal{L}_{y}^{j_{2}}g_{j_{1}}$, $\mathcal{L}_{\rho}^{j_{2}}g_{j_{1}}$
for $l=1,2,3,4$, respectively.

\begin{assumption} \label{ass:Coefficient-comparison-conditions-of-Gamma1-Gamma4}
There exist constants $C>0,~\delta_{1}, \delta_{2} \ge 1$, $ \delta_{3},\delta_{4}\ge \frac{1}{2}$ and $\gamma_{l}\ge1$ such that
\begin{align*}
\left|\Gamma_{l}\left(\mathds{F}_l(\cdot),\Delta t\right)- \mathds{F}_l(\cdot)\right| \le & ~C \Delta t ^{\delta_{l} } \left|\mathds{F}_l(\cdot)\right|^{\gamma_{l}},  \quad l=1,2,3,4.
\end{align*}
\end{assumption}
This assumption strengthens Assumption \ref{ass:Coefficient-comparison-conditions-of-f} by providing more precise control over the approximation operators. 

We summarize our main result below, the complete proof and the necessary auxiliary estimates are provided in Section \ref{sec:conver-analy}.

\begin{theorem} \label{thm:convergence-result-particle-scheme}
Let Assumptions \ref{ass:assumptions_for_MV_coefficients}, \ref{ass:Initial-value}-\ref{ass:Gamma-control-conditions} and \ref{ass:Coefficient-comparison-conditions-of-Gamma1-Gamma4} be satisfied. Then there exist $C,\beta>0$ such that 
$$
\sup _{t \in[0, \mathcal{T}]} \sup _{i\in \mathcal{I}_{N}} \mathbb{E}\left[\Big|X_{t}^{i,N}-Y_{t}^{i,N}\Big|^{p}\right] \le C\bigg(1+\Big(\mathbb{E}\Big[\big|Y_{0}\big|^{\bar{p}}\Big]\Big)^{\beta}\bigg) \Delta t^{p}, \quad p\in[2,\tilde{p}\wedge \hat{p}],
$$
where $\hat{p}$ and $\tilde{p}$ are from Lemmas \ref{lem:g_and_Gamma_g_difference} and \ref{lem:e_and_f_difference}. Also, $\bar{p}$ is from Lemma \ref{lem:moment-bounds-of-numerical-solution}.
\end{theorem}

By combining this result with Proposition \ref{prop:propagation_of_chaos}, we establish a quantitative error bound for the Milstein-type approximation \eqref{eq:continuous-version-of-MMS} of the MV-SDEs \eqref{eq:MV_SDE}, under the assumption that the initial value has sufficiently high bounded moments.

\begin{corollary}  \label{cor:convergence-theorem}
Under the same conditions of Theorem \ref{thm:convergence-result-particle-scheme}, we have
\[
\sup_{t \in [0,\mathcal{T}]} \sup_{i \in \mathcal{I}_N} 
\mathbb{E}\big[\,|X_t^i - Y_t^{i,N}|^2\,\big] 
\le C\big( \eta_d(N) + \Delta t^2 \big),
\]
where $\eta_d(N)$ is defined in Proposition \ref{prop:propagation_of_chaos} and $C > 0$ is independent of $n$ and $N$.
\end{corollary}

\subsection{Some examples of Milstein-type schemes}
\label{subsec:examples_of_scheme}
We present several examples with specified $\Gamma_{l}$ for $l=1,2,3,4$ in \eqref{eq:modified_Milstein_method}. These examples demonstrate possible forms of the operators $\Gamma_{l}$ for $l=1,2,3,4$ within the proposed framework.
 \begin{itemize}
 \item Tanh Milstein (TanhM) 
  \begin{equation}
  \label{ex-eq:tanh-misltein}
 \Gamma_{l}\big(\mathds{F}_l(\cdot),\Delta t\big) = \Delta t^{-1} \tanh\big(\Delta t\,\mathds{F}_l(\cdot) \big), \quad l=1,2,3,4.
\end{equation}
\item Sine Milstein (SineM)  
 \begin{equation}
     \label{ex-eq:sin-misltein}
\Gamma_{l}\big(\mathds{F}_l(\cdot),\Delta t\big) = \Delta t^{-1} \sin\big(\Delta t\,\mathds{F}_l(\cdot) \big), \quad l=1,2,3,4.
 \end{equation}

\item Tamed Milstein (TameM)  
  \begin{equation}
  \label{ex:tamed-euler-scheme}
\Gamma_{l}\big(\mathds{F}_l(\cdot),\Delta t\big) = \frac{\mathds{F}_l(\cdot)}{1+\Delta t\left|\mathds{F}_l(\cdot)\right|}, \quad  l=1,2,3,4.
 \end{equation}

\item Mixed Milstein (MixM) 
  \begin{align}
  \label{ex-eq:mix-milstein-method}
\Gamma_{1}\big(\mathds{F}_1(\cdot),\Delta t\big) &= \Delta t^{-1} \tanh\big(\Delta t\,\mathds{F}_1(\cdot) \big),  \quad 
         \Gamma_{2}\big(\mathds{F}_2(\cdot),\Delta t\big) = \Delta t^{-1} \sin\big(\Delta t\,\mathds{F}_2(\cdot) \big), \notag \\
        \Gamma_{3}\big(\mathds{F}_3(\cdot),\Delta t\big) &= \frac{\mathds{F}_3(\cdot)}{1+\Delta t\left|\mathds{F}_3(\cdot)\right|},   \quad 
        \Gamma_{4}\big(\mathds{F}_4(\cdot),\Delta t\big) = \Delta t^{-1} \tanh\big(\Delta t\,\mathds{F}_4(\cdot) \big). 
    \end{align}
\end{itemize}
For \eqref{ex-eq:tanh-misltein}, based on $|\tanh(y)| \le |y|$ and $|\tanh(y)| \le 1 $, we observe the following
$$
\left|\Gamma_{l}\left(\mathds{F}_l(\cdot),\Delta t\right)\right| \le \left|\mathds{F}_l(\cdot)\right|, \quad  \left|\Gamma_{l}\left(\mathds{F}_l(\cdot),\Delta t\right)\right| \le \Delta t ^{-1}.
$$
This implies that Assumption \ref{ass:Gamma-control-conditions} is valid with $\zeta_{l}=1,\,l=1,2,3,4$. Furthermore, applying the property of the hyperbolic tangent function for any $0\le \theta \le 1$, 
 $|y-\tanh(y)|\le  |y|^{3-2\theta}$, we arrive at
$$
\left|\Gamma_{l}(\mathds{F}_l(\cdot),\Delta t)-\mathds{F}_l(\cdot)\right|  =\Delta t^{-1} \left|\tanh\left(\Delta t \mathds{F}_l(\cdot)\right)-\Delta t \mathds{F}_l(\cdot)\right|   \le \Delta t |\mathds{F}_l(\cdot)|^{2}.
$$
Thus, Assumption \ref{ass:Coefficient-comparison-conditions-of-Gamma1-Gamma4} is fulfilled with $\delta_{l}=1$ and $\gamma_{l}=2,\,l=1,2,3,4$.

Turning to the Sine Milstein method \eqref{ex-eq:sin-misltein}, it is straightforward to verify that the assumptions \ref{ass:Gamma-control-conditions} and \ref{ass:Coefficient-comparison-conditions-of-Gamma1-Gamma4} hold with $\zeta_l = 1$, $\delta_l = 1$, and $\gamma_l = 2$ for $l = 1, 2, 3, 4$, following the similar arguments as for \eqref{ex-eq:tanh-misltein}.

For the Tamed Milstein method \eqref{ex:tamed-euler-scheme}, we verify that
$$
|\Gamma_{l}\left(\mathds{F}_l(\cdot),\Delta t\right)| \le \left|\mathds{F}_l(\cdot)\right|, \quad  |\Gamma_{l}\left(\mathds{F}_l(\cdot),\Delta t \right)| \le \frac{\left|\mathds{F}_l(\cdot)\right|}{\Delta t\left|\mathds{F}_l(\cdot)\right|} \le \Delta t^{-1}.
$$
This shows that Assumption \ref{ass:Gamma-control-conditions} holds with $\zeta_l = 1$ for $l = 1, 2, 3, 4$.
Moreover, according to \eqref{ex:tamed-euler-scheme}, we have
$$
\left|\Gamma_{l}\left(\mathds{F}_l(\cdot),\Delta t\right)-\mathds{F}_l(\cdot)\right|^{2}  \le  \frac{|\mathds{F}_l(\cdot)|^{2}\Delta t}{1+\Delta t |\mathds{F}_l(\cdot)|}  \le \Delta t  |\mathds{F}_l(\cdot)|^{2},
$$
i.e. Assumption \ref{ass:Coefficient-comparison-conditions-of-Gamma1-Gamma4} is satisfied with $\delta_{l}=1,$ and $\gamma_{l}=2,~l=1,2,3,4$.

Finally, in the case of the Mixed Milstein method \eqref{ex-eq:mix-milstein-method}, following a similar verification procedure as for the other methods, we confirm that Assumptions \ref{ass:Gamma-control-conditions} and \ref{ass:Coefficient-comparison-conditions-of-Gamma1-Gamma4} are satisfied with $\zeta_l = 1$, $\delta_l = 1$, and $\gamma_l = 2$ for $l = 1, 2, 3, 4$.

\section{Numerical experiments} \label{sec:numeri-experi}
This section tests the proposed Milstein-type numerical scheme \eqref{eq:modified_Milstein_method} on three models: the mean-field double well dynamics model, the Cucker-Smale flocking model \cite{erban2016cucker,chen2022flexible} and the FitzHugh-Nagumo model \cite{baladron2012mean}. These models feature superlinear growth in both the drift and diffusion coefficients with respect to the state variable. We present numerical results of the TanhM \eqref{ex-eq:tanh-misltein}, SineM \eqref{ex-eq:sin-misltein}, TamedM \eqref{ex:tamed-euler-scheme} and MixM \eqref{ex-eq:mix-milstein-method} methods, introduced in Section \ref{subsec:examples_of_scheme}.

To evaluate the applicability of the proposed schemes with limited computational resources, we include convergence tests with smaller particle numbers, assessing the robustness and sensitivity to sampling errors. When the number of particles $N$ is relatively small, a more accurate simulation of our numerical schemes \eqref{eq:modified_Milstein_method} requires retaining and computing the final term in each scheme. In all simulations, the corresponding Lions derivative is taken from Examples 1 and 4 in Section 5.2.2 of the monograph \cite{carmona2018probabilistic}. Off-diagonal iterated stochastic integrals are approximated using the Wiktorsson method \cite{wiktorsson2001joint}, with $K = 20$ terms in the truncated series expansion, which yields a mean-square error of order $\mathcal{O}(\Delta t^2 / K^2)$.

Unless otherwise stated, the number of particles is $N = 1000$. The mean-square error (MSE) at the terminal time $\mathcal{T}$ is computed as:
$$
{\rm MSE} \approx \left(\frac{1}{N}\sum_{i=1}^{N}\left|Y_{\mathcal{T}}^{i,N,\delta t}-Y_{\mathcal{T}}^{i,N,\Delta t} \right|^{2}\right)^{\frac{1}{2}},
$$
where $Y_{\mathcal{T}}^{i,N,\delta t}$, $Y_{\mathcal{T}}^{i,N,\Delta t}$ represent the numerical solutions for the $i$-th particle with step sizes $\Delta t$ and $\delta t$, respectively. 
\vspace{-0.3em}
\begin{example} \label{exam:double-well-model}
Consider the double-well dynamics in the form of
\begin{equation} 
\left\{\begin{array}{l}  \label{exam:num-double-well model}
\mathrm{d} X_{t}=  \left( \lambda_{1} X_{t}\left(1-X_{t}^{2}\right)+\lambda_{2 }\mathbb{E}\left[X_{t}\right] \right)\mathrm{d} t +  \left(\mu_{1}(1-X_{t}^{2})+\mu_{2}\int_{\mathbb{R}}\sin(X_{t}-z)\mathscr{L}_{X_{t}}(dz)\right) \mathrm{d} W_t,    \\
X_0=x_{0},
\end{array}\right.
\end{equation}
where $t\in[0,\mathcal{T}]$. 
\end{example}
Take two sets of model parameters:
\begin{itemize}
    \item \textbf{Case 1:} $\lambda_{1}=40$, $\lambda_{2} = 4$, $\mu_{1}=0.3$, $\mu_{2}=2$, and $x_{0}\sim \mathcal{N}(0,1)$;
    \item \textbf{Case 2:} $\lambda_{1}=5$, $\lambda_{2} = 1$,  $\mu_{1}=0.1$, $\mu_{2}=0.1$ and $x_{0}=0$.
\end{itemize}
All cases in Example \ref{exam:double-well-model} satisfy Assumptions \ref{ass:assumptions_for_MV_coefficients} and \ref{ass:polynomial-growth-of-f} with parameters $\bar{p} = 426$, $p^{*} = 4$ and $\gamma = 2$. In \textbf{Case 1}, we simulate $M = 500$ trajectories of a randomly chosen particle governed by equation \eqref{exam:num-double-well model}, using both the classical Milstein method \eqref{eq:classi_Milstein_method} and Milstein-type schemes \eqref{eq:modified_Milstein_method}. The simulations are conducted over the interval $[0, 1]$ with a time step size $\Delta t = 2^{-6}$ and $N = 1000$ particles. For simplicity, we omit the last term in the numerical schemes \eqref{eq:modified_Milstein_method} in all experiments unless otherwise stated, as it converges to zero in the mean-square sense as $N \to \infty$.
 
As shown in the left panel of Figure \ref{fig:1D_single_parti_M_realiza}, the classical Milstein method exhibits divergence, while Milstein-type schemes remain stable and converge toward the neighborhoods of $1$ or $-1$.

\begin{figure}[H]
\begin{minipage}[t]{0.5\linewidth} 
\centering
\includegraphics[width=7.5cm,height=5cm]{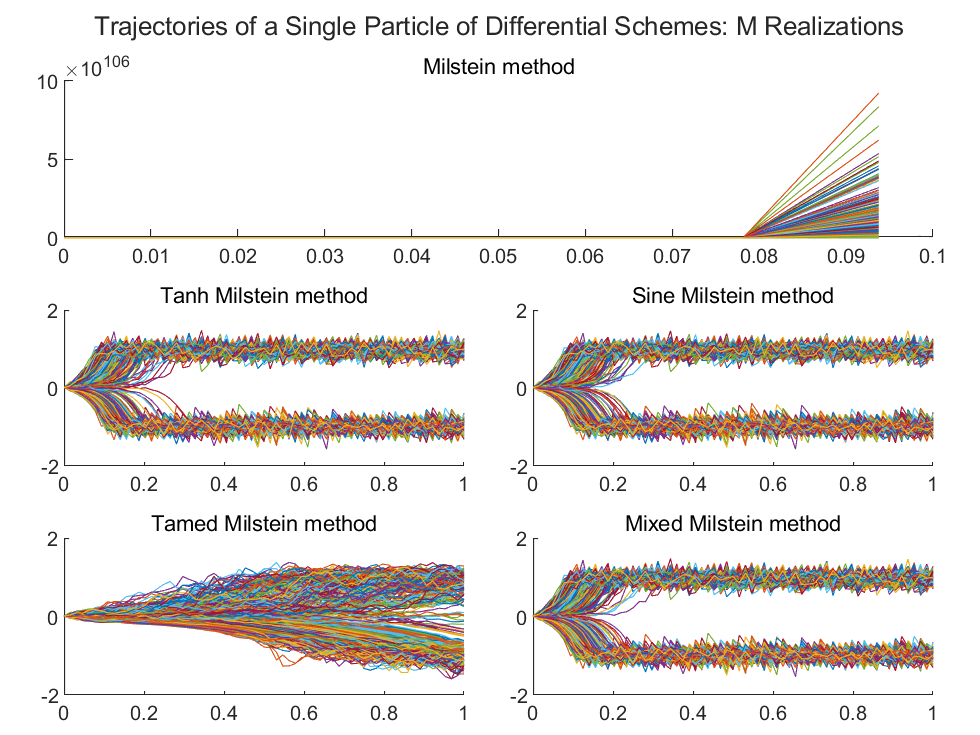}   
\end{minipage}
\hfill
\begin{minipage}[t]{0.5\linewidth}  
\centering
\includegraphics[width=7.5cm,height=5cm]{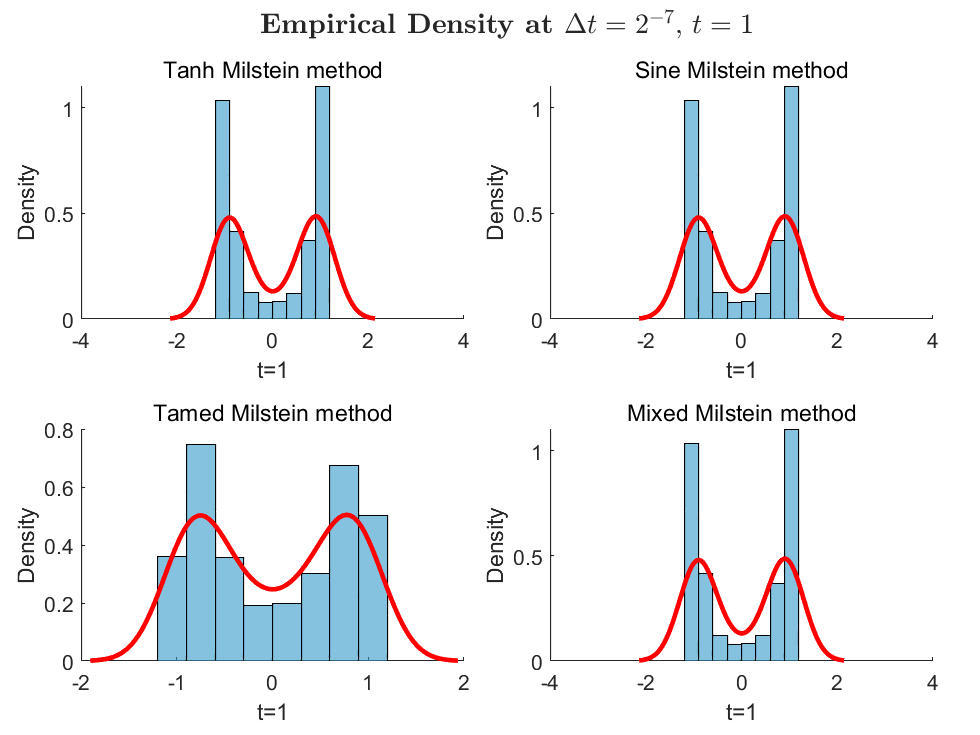} 
\end{minipage}
\vspace{-3em}
\caption{Trajectories of a single particle over $M=500$ realizations with $N=1000$ (left), and the empirical distribution at time $\mathcal{T}=1$ (right).}
\label{fig:1D_single_parti_M_realiza}
\end{figure}
\vspace{-1em}

The right panel of Figure \ref{fig:1D_single_parti_M_realiza} presents the empirical density at $\mathcal{T} = 1$ for \textbf{Case 2}, computed by Milstein-type schemes. 
The bimodal distribution observed at the final time suggests that the schemes accurately capture the key features of the distribution.

\begin{figure}[H]
\begin{minipage}[t]{0.5\linewidth} 
\centering
\includegraphics[width=7.5cm,height=5cm]{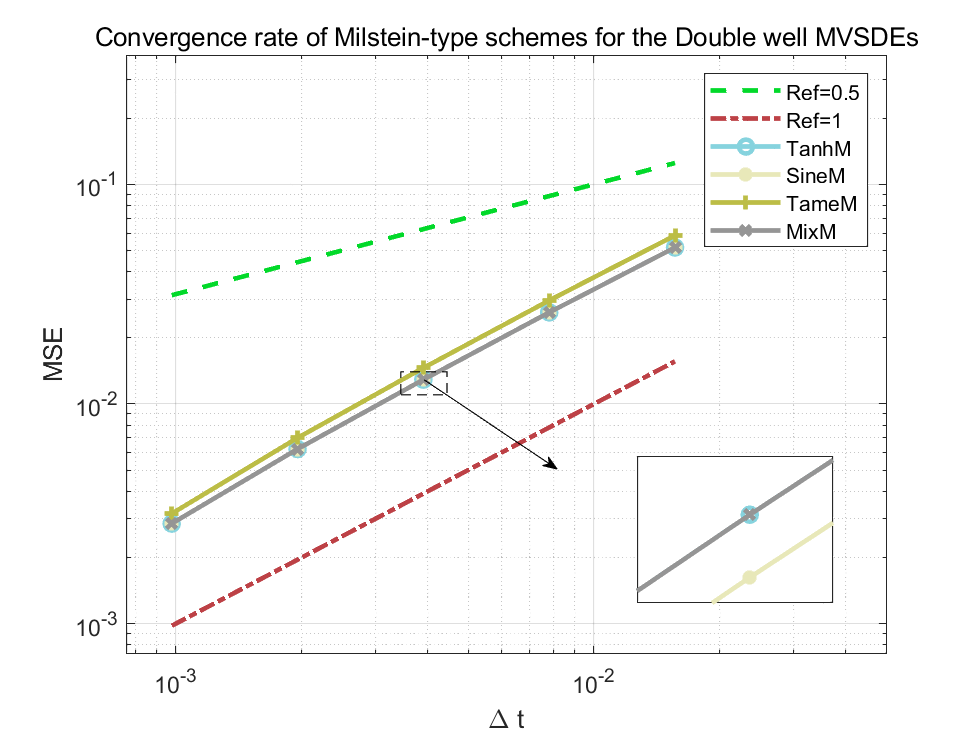}   
\end{minipage}
\hfill
\begin{minipage}[t]{0.5\linewidth}  
\centering
\includegraphics[width=7.5cm,height=5cm]{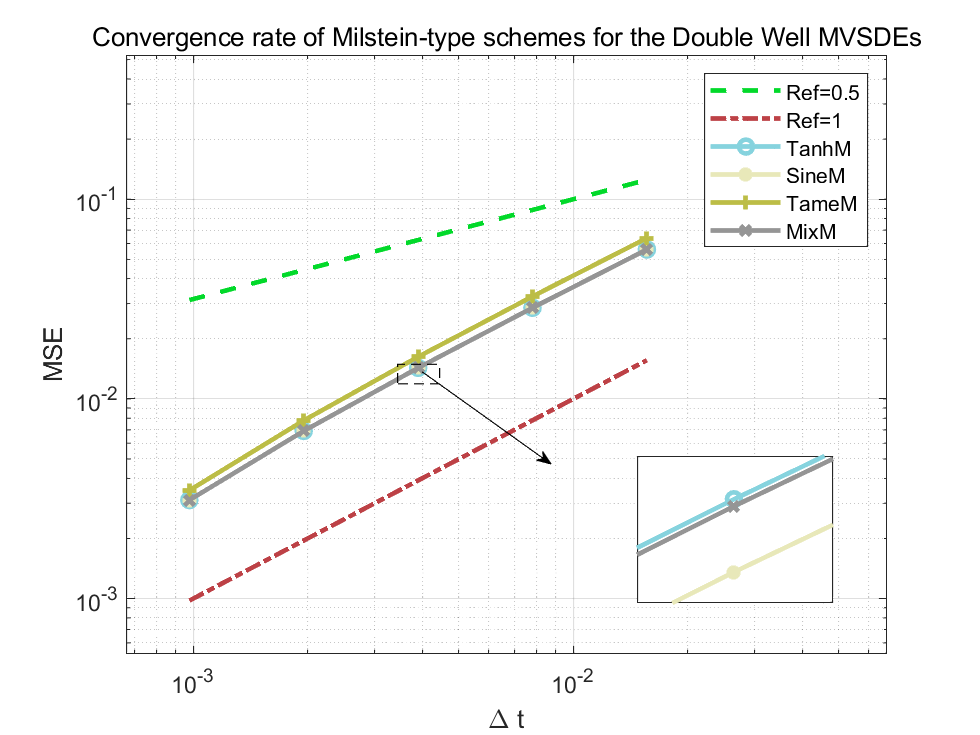} 
\end{minipage}
\vspace{-3em}
\caption{Convergence order of Milstein-type schemes for the double well model: $N=1000$ (left), $N=20$ (right).}
\label{fig:conver-rate-Mil-scheme-1D}
\end{figure}
\vspace{-1em}
Next, we examine the convergence order of the numerical schemes \eqref{eq:modified_Milstein_method} in \textbf{Case 2}. Since the analytical solution of the interacting particle system is not available, we use the numerical solution with a fine step size $\Delta t = 2^{-12}$ as the reference. Figure \ref{fig:conver-rate-Mil-scheme-1D} shows the log-log convergence plots for Milstein-type schemes with $N = 1000$ and $N = 20$, using step sizes $\Delta t = \{2^{-6}, 2^{-7}, 2^{-8}, 2^{-9}, 2^{-10}\}$. Reference lines with slopes $\frac12$ and $1$ are included for comparison.

The convergence curves of the four methods are nearly indistinguishable; therefore, we utilize the midpoint to illustrate the differences among them, as highlighted in the lower right corner. As shown in Figure \ref{fig:conver-rate-Mil-scheme-1D}, the convergence rate is close to $1$ for both small and large particle numbers ($N = 1000$ and $N = 20$), demonstrating consistent performance across different particle numbers $N$.
\vspace{-0.5em}
\begin{example}
\label{exa:Cucker-Smale-model}
For the two-dimensional Cucker-Smale flocking model, the dynamics are described by the MV-SDEs \eqref{eq:MV_SDE} with the drift and diffusion given by
\begin{equation} \label{eq:num-exam-Cuck-Sma-model}
\begin{gathered}
f\left(X_t, \mathscr{L}_{X_t}\right)=\binom{-\lambda_{1}v_t^3+1 +\lambda_{2} \int_{\mathbb{R}}\left(v_t-z\right) \mathscr{L}_{v_t}\left(dz\right)}{v_t}, \; \; g\left(X_t, \mathscr{L}_{X_t}\right)=\binom{\sigma_{1}v_t^{2}+\sigma_{2} \int_{\mathbb{R}}\left(v_t-z\right) \mathscr{L}_{v_t}\left(dz\right)}{0},
\end{gathered}
\end{equation}
where $X_{t}=(v_{t},x_{t}),\; t\in[0,1]$. Two cases are tested:  
\vspace{-0.8em}
\begin{itemize}
    \item \textbf{Case 1:} $\lambda_{1}=139.5$, $\lambda_{2} = -30$, $\sigma_{1} = 0.8$, $\sigma_{2}=100$, and $X_{0}\sim\mathcal{N}(\mu,\Sigma)$,\\ \hbox{where} $\mu = \begin{bmatrix} 20 \\ 30 \end{bmatrix}$ \hbox{and} $\Sigma = \begin{bmatrix} 4 & 3.5 \\ 3.5 & 4 \end{bmatrix}$;
    \item \textbf{Case 2:} $\lambda_{1}=1$, $\lambda_{2} = -0.5$, $\sigma_{1} = 0.01$, $\sigma_{2}=0.01$, and $X_{0}\sim (\mathcal{N}(0,1),\mathcal{N}(0,1)) $.
\end{itemize}
\end{example}
\vspace{-0.5em}
\begin{figure}[H]
\begin{minipage}[t]{0.5\linewidth} 
\centering
\includegraphics[width=7.5cm,height=5cm]{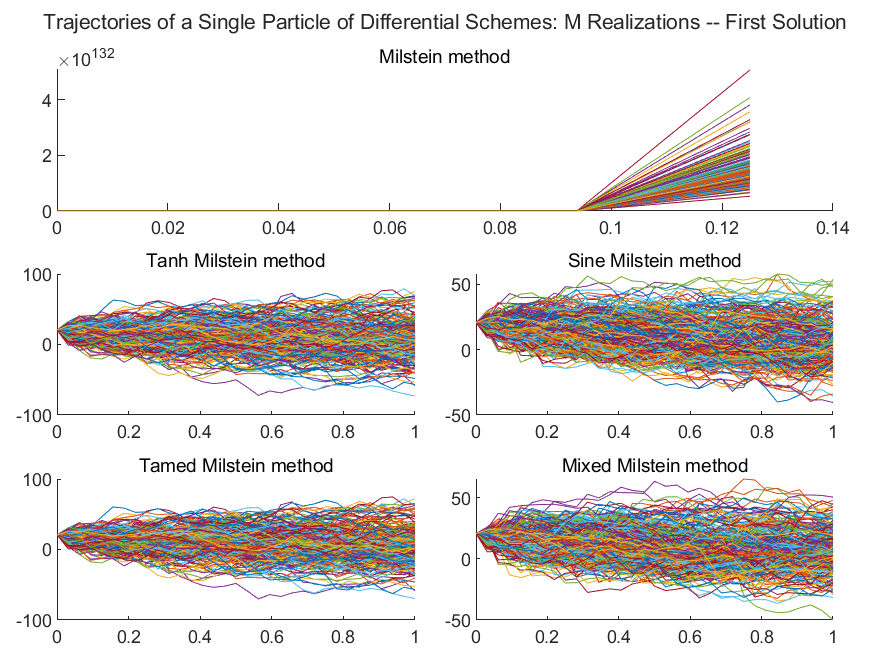}   
\end{minipage}
\hfill
\begin{minipage}[t]{0.5\linewidth}  
\centering
\includegraphics[width=7.5cm,height=5cm]{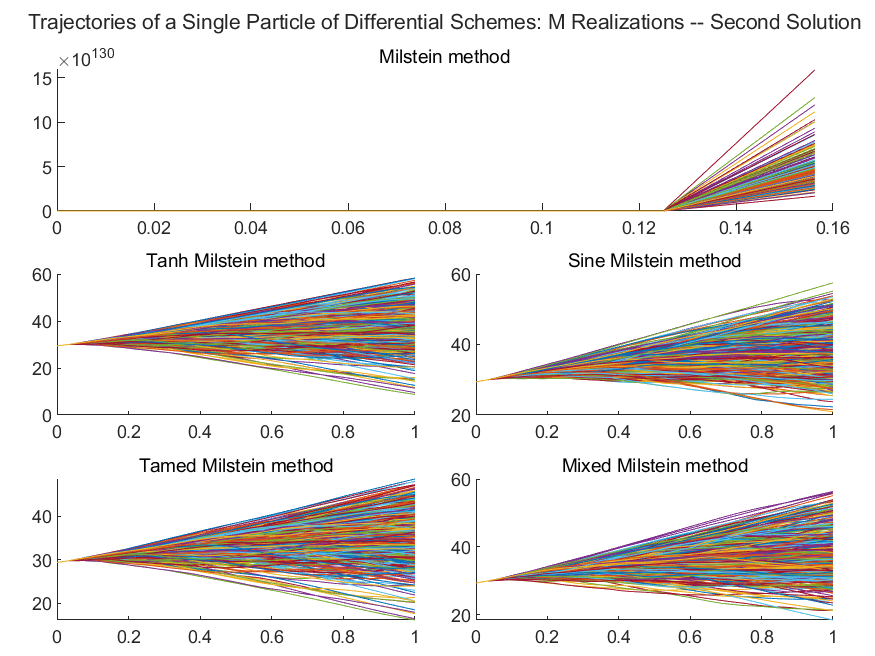} 
\end{minipage}
\vspace{-3em}
\caption{Trajectories of a single particle (2D output) for the Cucker-Smale Flocking model $N = 1000$, $M = 500$ realizations. Left: first component; Right: second component.}
\label{fig:traje-single-parti-N-large-2D}
\end{figure}
We verify that the coefficients in \eqref{eq:num-exam-Cuck-Sma-model} satisfy Assumptions \ref{ass:assumptions_for_MV_coefficients} and \ref{ass:polynomial-growth-of-f} as before. For \textbf{Case 1}, we simulate $M=500$ trajectories of a randomly selected particle using the classical Milstein method and the Milstein-type methods. Figure \ref{fig:traje-single-parti-N-large-2D} shows the component-wise trajectories for $N=1000$, where the classical method diverges while the proposed schemes remain stable.

For \textbf{Case 2}, Figure \ref{fig:2D_pdf_diff_step_size} shows the empirical densities at $t=1$ with step sizes $\Delta t=2^{-1}$ and $2^{-8}$. The schemes remain stable even for relatively large step $\Delta t=2^{-1}$, while relatively small step $\Delta t=2^{-8}$ yields a more accurate distribution. Figure \ref{fig:2D_pdf_differ_T} further illustrates the density evolution at $t=2$ and $t=8$, where the distribution stabilizes over time.
\vspace{-1em}
\begin{figure}[H]
\begin{minipage}[t]{0.5\linewidth} 
\centering
\includegraphics[width=7.5cm,height=5cm]{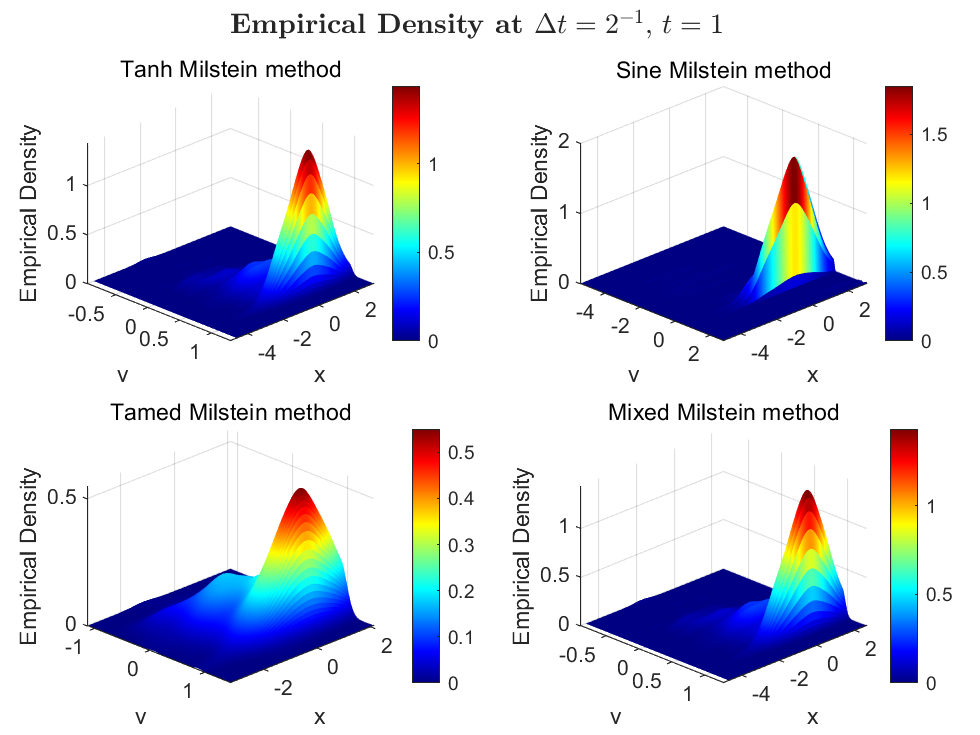}   
\end{minipage}
\hfill
\begin{minipage}[t]{0.5\linewidth}  
\centering
\includegraphics[width=7.5cm,height=5cm]{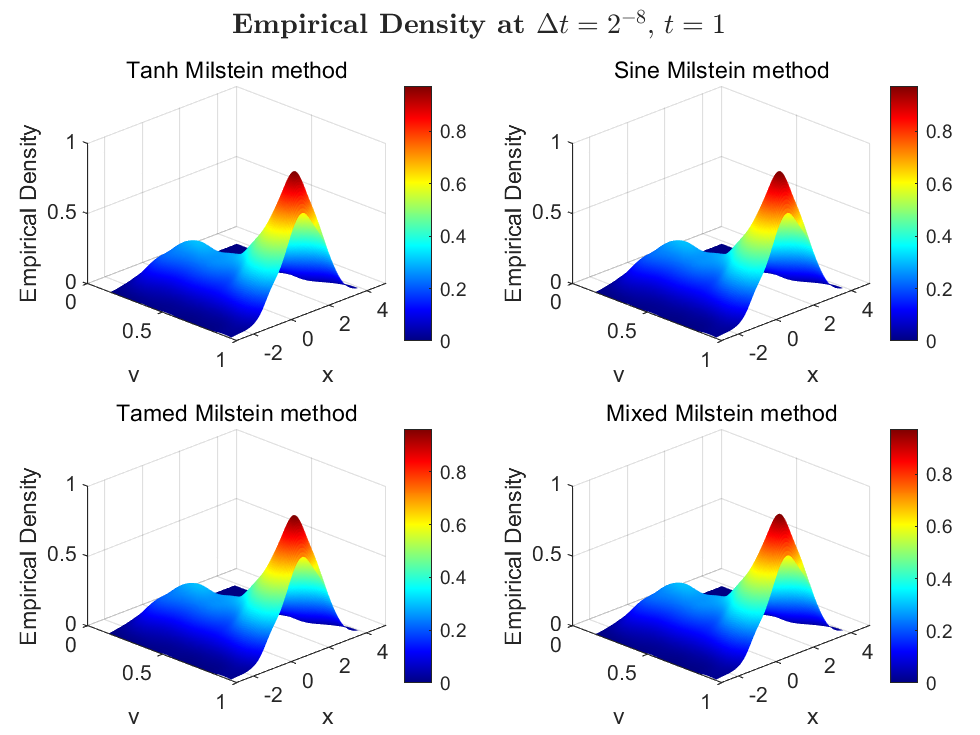} 
\end{minipage}
\vspace{-3em}
\caption{Empirical densities at $t = 1$ for Milstein-type schemes with $\Delta t = 2^{-1}$ (left) and $\Delta t = 2^{-8}$ (right).}
\label{fig:2D_pdf_diff_step_size}
\end{figure}
\begin{figure}[H]
\begin{minipage}[t]{0.5\linewidth} 
\centering
\includegraphics[width=7.5cm,height=5cm]{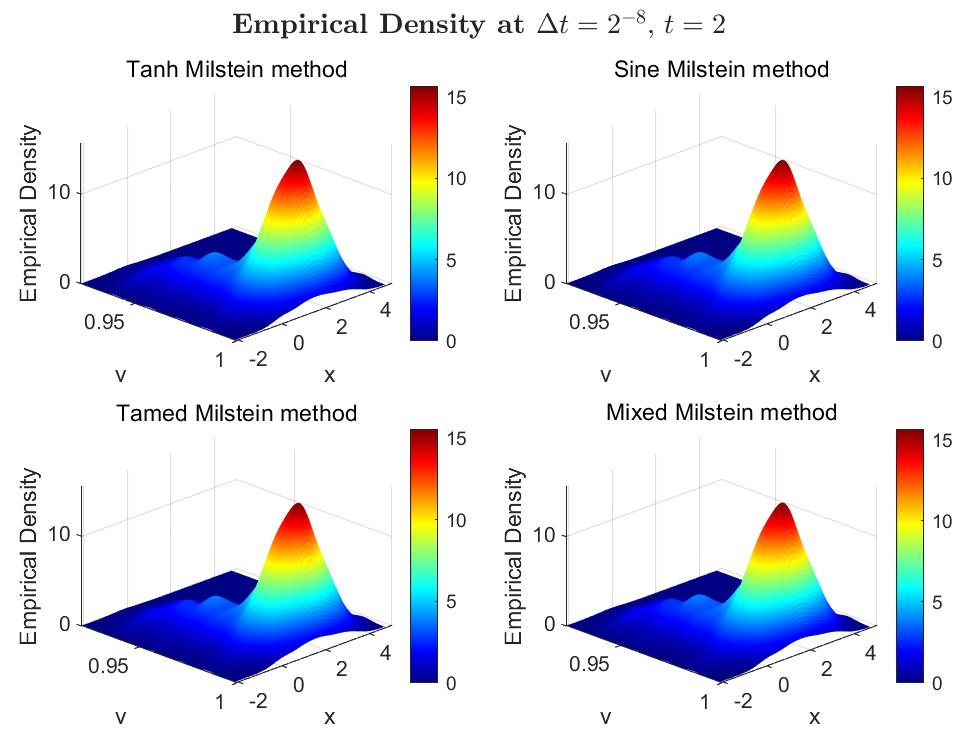}   
\end{minipage}
\hfill
\begin{minipage}[t]{0.5\linewidth}  
\centering
\includegraphics[width=7.5cm,height=5cm]{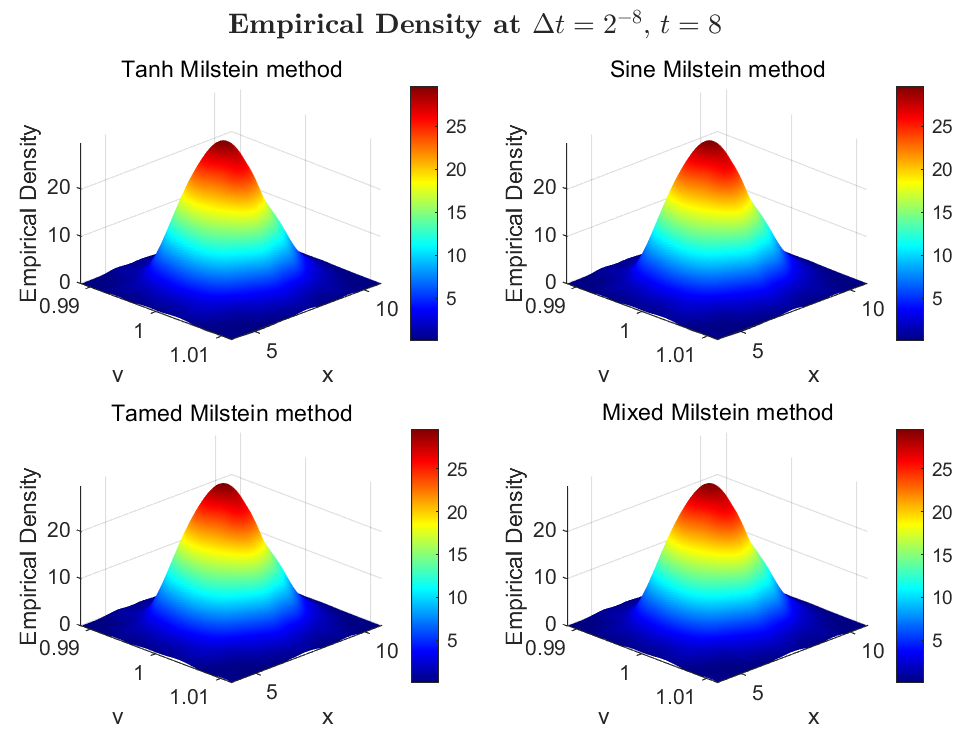} 
\end{minipage}
\vspace{-3em}
\caption{Empirical densities at $t = 2$ (left) and $t = 8$ (right) for Milstein-type schemes with $\Delta t = 2^{-8}$.}
\label{fig:2D_pdf_differ_T}
\end{figure}
\vspace{-1.5em}
We now examine the convergence order of our numerical schemes in the \textbf{Case 2}. Following Example \ref{exam:num-double-well model}, the solution with step size $\Delta t = 2^{-12}$ is used as the reference. The tested step sizes are $\Delta t = \{2^{-6}, 2^{-7}, 2^{-8}, 2^{-9}, 2^{-10}\}$. In the log-log plot, we also display reference lines with slopes $0.5$ and $1$ for comparison. In Figure \ref{fig:conver-rate-Mil-scheme-2D}, the mean-square errors for $N=1000$ and $N=20$ are reported and all error curves are observed to be parallel to the slope $1$ reference line, thereby confirming the theoretical results.
\begin{figure}[H]
\begin{minipage}[t]{0.5\linewidth} 
\centering
\includegraphics[width=7.5cm,height=5cm]{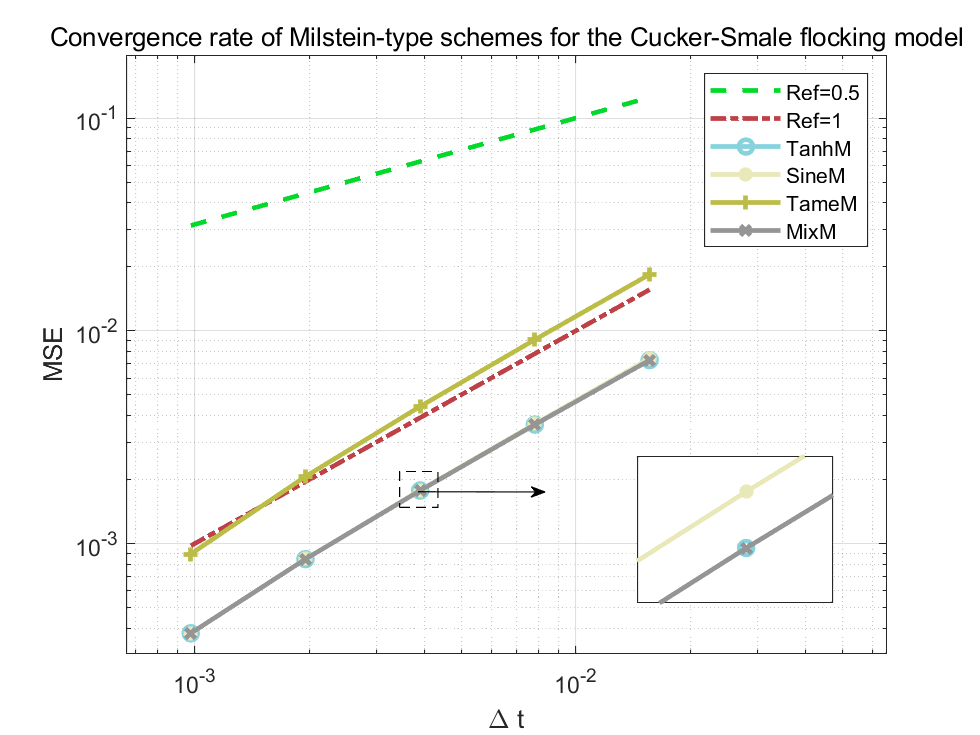}   
\end{minipage}
\hfill
\begin{minipage}[t]{0.5\linewidth}  
\centering
\includegraphics[width=7.5cm,height=5cm]{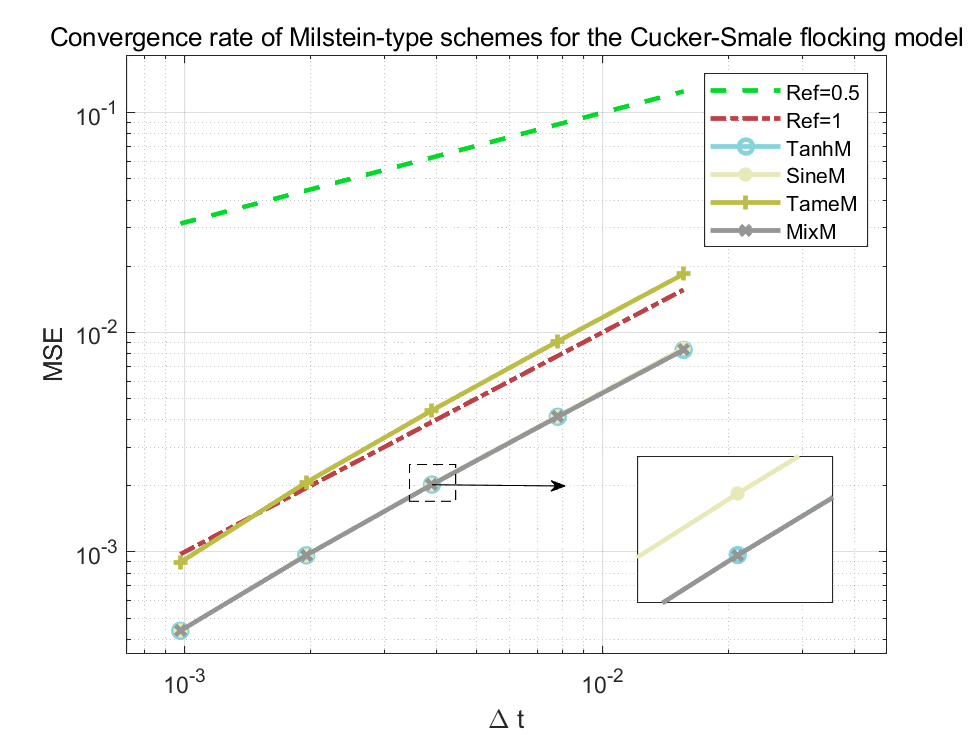} 
\end{minipage}
\vspace{-3em}
\caption{Convergence order of Milstein-type schemes for the Cucker-Smale flocking model: $N=1000$ (left), $N=20$ (right).}
\label{fig:conver-rate-Mil-scheme-2D}
\end{figure}
\vspace{-0.5em}
We proceed to examine the convergence of the proposed numerical schemes for MV-SDEs driven by multidimensional Brownian motion.
\begin{example}
\label{exam:FHN-model-numerical}
Consider the FitzHugh-Nagumo model, which corresponds to MV-SDEs \eqref{eq:MV_SDE} with 
\begin{align} \label{eq:Fitz-Nagu-dift}
f(X_{t}, \mathscr{L}_{X_{t}}) & :=\left(\begin{array}{c}
x_1- 35x_1^3 -x_2+I-\int_{\mathbb{R}^3} J\left(x_1-V_{\text{rev}}\right) z \,\mathscr{L}_{x_3}(\mathrm{d}z) \\ 
c\left(x_1+a-b x_2\right) \\ 
a_r \frac{T_{\max }\left(1-x_3\right)}{1+\exp \left(-\lambda\left(x_1-V_T\right)\right)}-a_d x_3
\end{array}\right),
\end{align}
\begin{align} \label{eq:Fitz-Nagu-diffusion}
g(X_{t}, \mathscr{L}_{X_{t}}) & :=\left(\begin{array}{ccc}
\sigma_{\text{ext}} & 0 & -\int_{\mathbb{R}^3} \sigma_J\left(x_1-V_{\text{rev}}\right) z \,\mathscr{L}_{x_3}(\mathrm{d}z) \\ 
0 & 0 & 0 \\
0 & \sigma_{32}(x) & 0
\end{array}\right),
\end{align}
where $X_{t}=(x_{1},x_{2},x_{3})$, $X_{0}=(V_{0},w_{0},x_{0})$, and
$$
\sigma_{32}(x):=\mathds{1}_{\left\{x_3 \in(0,1)\right\}} \sqrt{ \frac{a_r T_{\max }\left(1-x_3\right)}{1+\exp \left(-\lambda\left(x_1-V_T\right)\right)}+a_d x_3} ~\Gamma \exp \left(-\Lambda /\left(1-\left(2 x_3-1\right)^2\right)\right),
$$
$\mathcal{T}=2$ is chosen as the final time, and the parameters have the values
$$
\begin{array}{ccccccc}
V_0=0  & a=0.7 & b=0.8 & c=0.08 & I=0.5 & \sigma_{e x t}=0.5 \\
w_0=1  & V_{r e v}=1 & a_r=1 & a_d=1 & T_{\max }=1 & \lambda=0.2 \\
x_0=0.5 & J=1 & \sigma_J=0.2 & V_T=2 & \Gamma=0.1 & \Lambda=0.5.
\end{array}
$$
\end{example}
The coefficients in \eqref{eq:Fitz-Nagu-dift}-\eqref{eq:Fitz-Nagu-diffusion}, which satisfy Assumptions \ref{ass:assumptions_for_MV_coefficients} and \ref{ass:polynomial-growth-of-f} with $\bar{p}=426$, $p^{*}=6$ and $\gamma=2$. We test the schemes on $[0,10]$.
Figure \ref{fig:3D_single_parti_M_realiza_N_large} (left) shows $M = 500$ trajectories of a randomly selected particle from a system of $N=1000$, generated by the Tanh-Milstein method, the other three methods produce similar results and are omitted. The simulations confirm that the proposed schemes provide stable approximations.
\begin{figure}[H]
\begin{minipage}[t]{0.5\linewidth} 
\centering
\includegraphics[width=8.5cm,height=5cm]{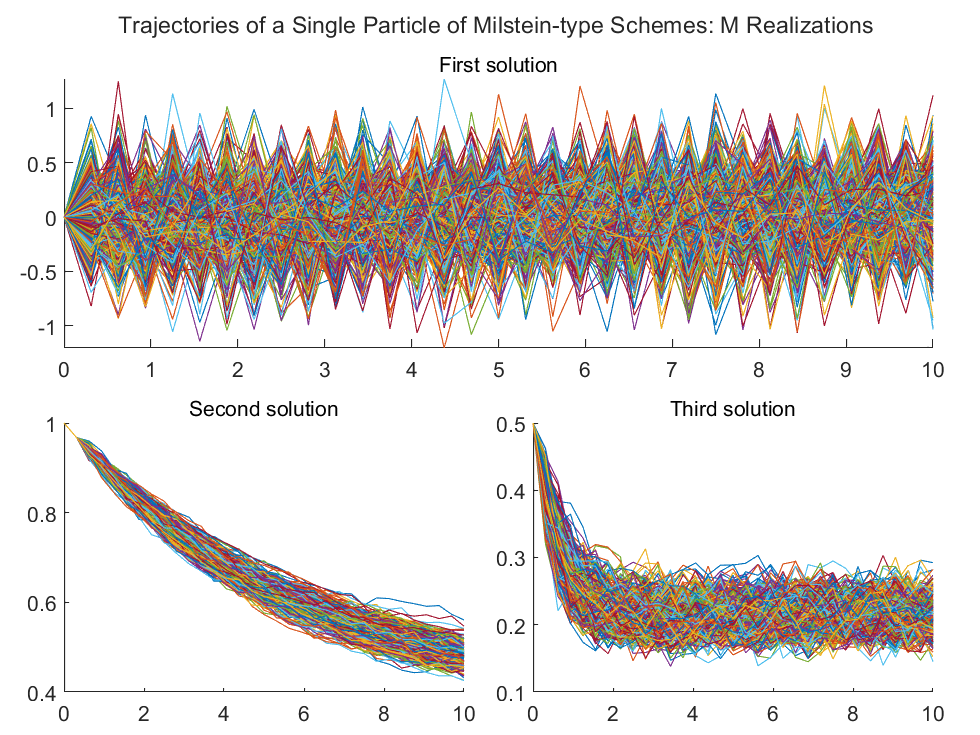}   
\end{minipage}
\hfill
\begin{minipage}[t]{0.5\linewidth}  
\centering
\includegraphics[width=7cm,height=5cm]{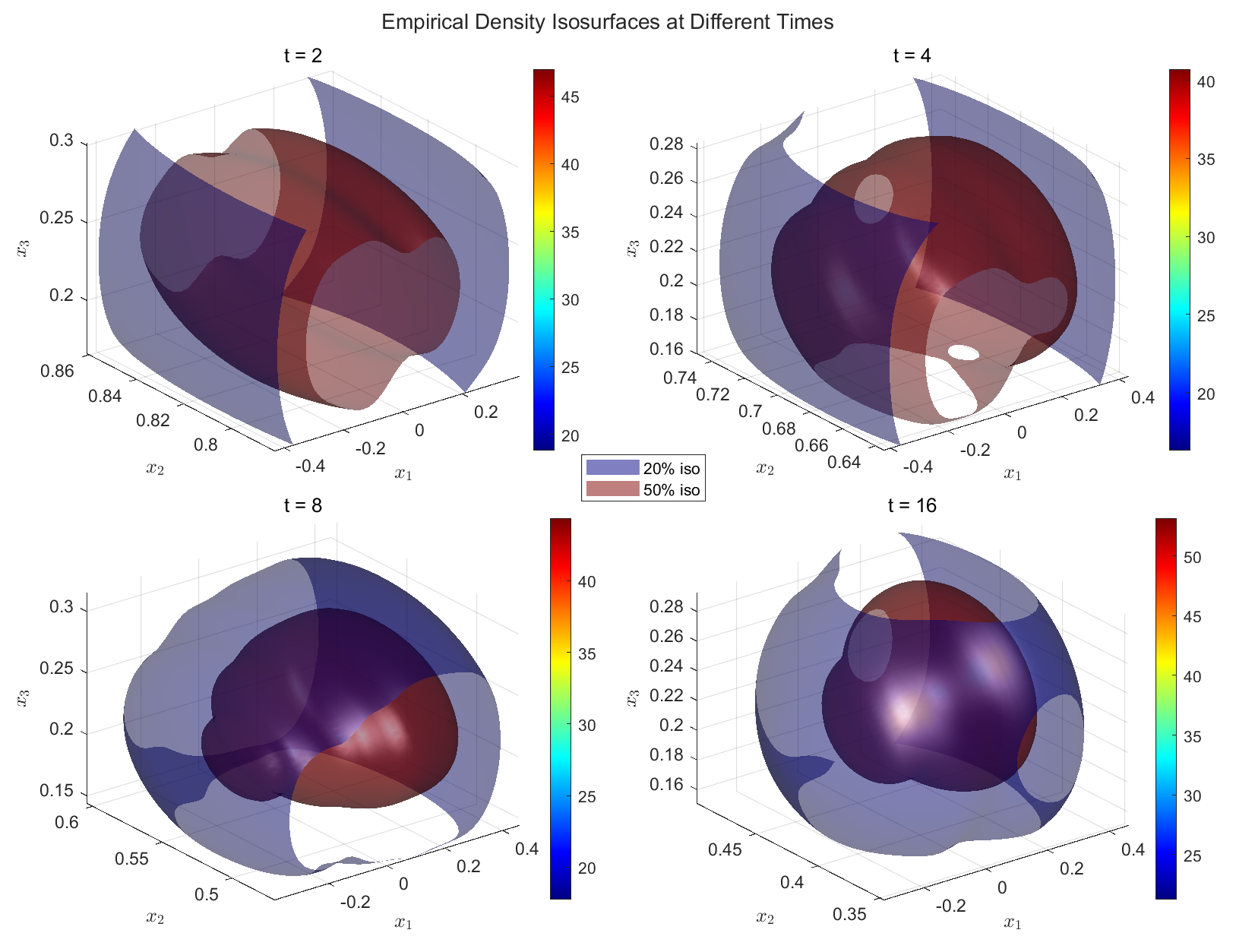} 
\end{minipage}
\vspace{-3em}
\caption{Trajectories of a single particle for the FitzHugh-Nagumo model (left) and empirical density evolution (right).}
\label{fig:3D_single_parti_M_realiza_N_large}
\end{figure}
\vspace{-1em}
The right panel of Figure \ref{fig:3D_single_parti_M_realiza_N_large} shows the empirical densities of the particle system at times $\mathcal{T}=2,4,8,16$, obtained with the TanhM scheme. In each subplot, the isosurfaces at $20\%$ and $50\%$ of the maximum density are displayed, and similar patterns are observed for the other three schemes. As time increases, the densities exhibit a unimodal structure and gradually approach a stable distribution.

\begin{figure}[H]
\begin{minipage}[t]{0.5\linewidth} 
\centering
\includegraphics[width=7.5cm,height=5cm]{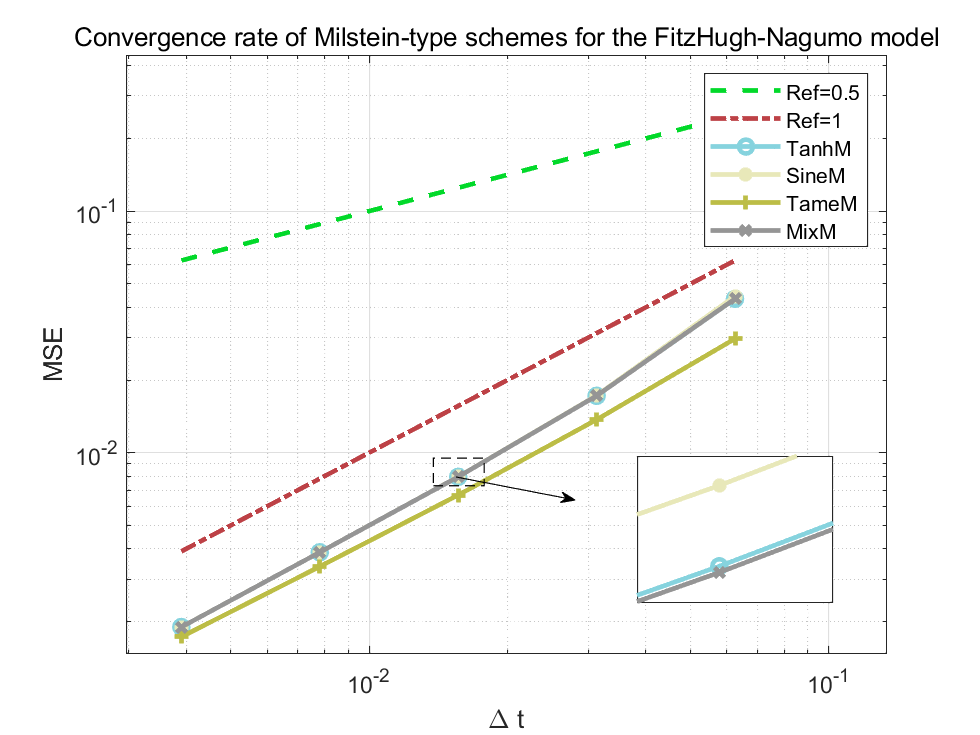}   
\end{minipage}
\hfill
\begin{minipage}[t]{0.5\linewidth}  
\centering
\includegraphics[width=7.5cm,height=5cm]{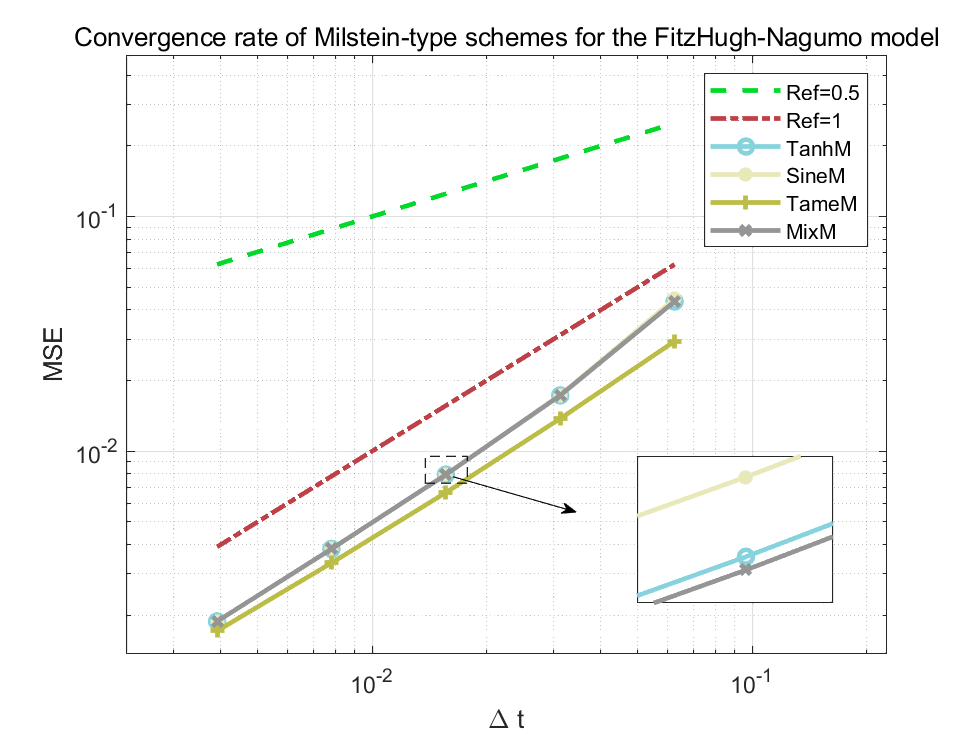} 
\end{minipage}
\vspace{-3em}
\caption{Convergence order of Milstein-type schemes for the FitzHugh-Nagumo model: $N=1000$ (left), $N=20$ (right).}
\label{fig:conver-rate-Mil-scheme-3D}
\end{figure}
\vspace{-1em}
We examine the convergence order of the proposed schemes for the FitzHugh-Nagumo model, using the solution with $\Delta t = 2^{-11}$ as reference and the same step sizes as in the previous examples. Figure \ref{fig:conver-rate-Mil-scheme-3D} shows the MSE for $N=1000$ and $N=20$, where TanhM \eqref{ex-eq:tanh-misltein}, SineM \eqref{ex-eq:sin-misltein}, TamedM \eqref{ex:tamed-euler-scheme} and MixM \eqref{ex-eq:mix-milstein-method} methods produce curves parallel to the slope-$1$ line, confirming first-order convergence and supporting the theory.
\section{Auxiliary estimates for strong convergence of Milstein-type schemes} \label{sec:conver-analy}
In this section, we provide the estimates required in the proof of the main theorem presented in Section \ref{sec:conver_rate_Milstein}. We begin with the following lemma, which establishes strong error bounds between $Y^{i,N}_{t}$ and $Y^{i,N}_{\tau_{n}(t)}$.
\begin{lemma}  \label{lem:The-difference-between-two-numerical-solutions}
Let Assumptions \ref{ass:assumptions_for_MV_coefficients} and \ref{ass:Initial-value}-\ref{ass:Gamma-control-conditions} hold. Then there exists $C>0$ such that, for any $q\ge 2$,
    \begin{align}
        \sup_{i\in \mathcal{I}_{N}}\mathbb{E}\bigg[\Big|Y_{t}^{i,N}-Y_{\tau_{n} (t)}^{i,N}\Big|^{q}\bigg] \le C \bigg(1+ \sup_{i\in\mathcal{I}_{N}}\mathbb{E}\bigg[\Big|Y_{\tau_{n}(t)}^{i,N}\Big|^{(\gamma+2)q}\bigg]\bigg)\Delta t^{\frac{q}{2}},
    \end{align}
where $\gamma$ is from Assumption \ref{ass:polynomial-growth-of-f}.
\end{lemma}       
\begin{proof}
According to \eqref{eq:continuous-version-of-MMS}, one can use the $\rm H\ddot{o}lder$ inequality, moment inequality, Assumption \ref{ass:Gamma-control-conditions}, \eqref{ineq:growth-condition-of-f}, \eqref{ineq:growth-condition-of-g}, \eqref{ineq:growth-condition-of-g-y} and \eqref{ineq:growth-condition-of-g-mu}, to obtain 
\begin{align*}
     &\mathbb{E}\bigg[\Big|Y_{t}^{i,N}-Y_{\tau_{n} (t)}^{i,N}\Big|^{q}\bigg]  \\
     \le &~C \mathbb{E}\bigg[\Big|\int_{\tau_{n}(t)}^{t}\Gamma_{1}\left(f \left(Y_{\tau_{n}(s)}^{i, N}, \rho_{\tau_{n}(s)}^{Y, N} \right),\Delta t \right) \mathrm{d} s\Big|^{q}\bigg]  \\
     &~+ C\mathbb{E}\bigg[\Big|\sum_{j_{1}=1}^{m} \int_{\tau_{n}(t)}^{t}\Gamma_{2}\left(g_{j_{1}} \left(Y_{\tau_{n}(s)}^{i, N}, \rho_{\tau_{n}(s)}^{Y, N} \right), \Delta t\right) \mathrm{d} W^{i,j_{1}}_{s} \Big|^{q}\bigg]  \\
     &~+ C\mathbb{E}\bigg[\Big|\sum_{j_{1},j_{2}=1}^{m}  \int_{\tau_{n}(t)}^{t} \int_{\tau_{n}(t)}^{s} \Gamma_{3}\left(\mathcal{L}_{y}^{j_{2}}~g_{j_{1}}\left(Y_{\tau_{n}(r)}^{i, N}, \rho_{\tau_{n}(r)}^{Y, N} \right),\Delta t\right)\,\mathrm{d} W^{i,j_{2}}_{r} \,\mathrm{d} W^{i,j_{1}}_{s} \Big|^{q}\bigg]  \\
     &~+C\mathbb{E}\bigg[\Big|\frac{1}{N}\sum_{j_{1},j_{2}=1}^{m} \sum_{k_{1}=1}^{N} \int_{\tau_{n}(t)}^{t}  \int_{\tau_{n}(t)}^{s} \Gamma_{4}\left( \mathcal{L}_{\rho}^{j_{2}}~g_{j_{1}} \left(Y_{\tau_{n}(r)}^{i,N},\rho_{\tau_{n}(r)}^{Y,N},Y_{\tau_{n}(r)}^{k_{1},N}\right) ,\Delta t\right)\mathrm{d} W^{k_{1},j_{2}}_{r} \mathrm{d} W^{i,j_{1}}_{s} \Big|^{q}\bigg]  \\
     \le &~C\mathbb{E}\bigg[1+\Big|Y_{\tau_{n}(t)}^{i,N}\Big|^{(\gamma+2)q}+\frac{1}{N}\sum_{k_{1}=1}^{N}\Big|Y_{\tau_{n}(t)}^{k_{1},N}\Big|^{(\gamma+2)q}+\mathbb{W}_{2}^{\left(\frac{\gamma}{2}+2\right)q}\left(\rho_{\tau_{n}(t)}^{Y,N},\delta_{0}\right)\bigg]\Delta t^{\frac{q}{2}}.
\end{align*}

Thanks to \eqref{eq:W_2^p}, we have
\begin{align*}
    \mathbb{E}\bigg[\Big|Y_{t}^{i,N}-Y_{\tau_{n} (t)}^{i,N}\Big|^{q}\bigg] \le C \bigg(1+\sup_{i\in\mathcal{I}_{N}}\mathbb{E}\bigg[\Big|Y_{\tau_{n}(t)}^{i,N}\Big|^{(\gamma+2)q}\bigg]\bigg)\Delta t^{\frac{q}{2}}.
\end{align*}
The proof of the lemma is thus complete.
\end{proof}
Lemmas \ref{lem:moment-bounds-of-numerical-solution} and \ref{lem:The-difference-between-two-numerical-solutions} provide the necessary estimates for deriving moment bounds of the continuous-time Milstein-type schemes \eqref{eq:continuous-version-of-MMS}.

\begin{lemma} \label{lem:moment_bound_of_continuous_form}
Let Assumptions \ref{ass:assumptions_for_MV_coefficients} and \ref{ass:Initial-value}-\ref{ass:Coefficient-comparison-conditions-of-f} hold. Then, for all $t\in [0,\mathcal{T}]$ and $p\in \big[2,\frac{\bar{p}-\Theta}{(\gamma+2)(1+\bar{\zeta}\Theta)}\big]$, there exist constants $C,\beta>0$, such that 
    $$
    \sup_{t\in[0,T]}\sup_{i\in\mathcal{I}_{N}}\mathbb{E}\Big[\big|Y_{t}^{i,N}\big|^{p}\Big]  \le C\Big(1+\Big(\mathbb{E}\Big[\big|Y_{0}\big|^{\bar{p}}\Big]\Big)^{\beta}\,\Big), 
    $$
where $\Theta$ is the constant appearing in Lemma \ref{lem:moment-bounds-of-numerical-solution} and $\bar{p} \ge 2(\gamma+2) (1+\bar{\zeta}\Theta)+ \Theta$.
\end{lemma}
\begin{proof}
    By applying Lemmas \ref{lem:moment-bounds-of-numerical-solution} and \ref{lem:The-difference-between-two-numerical-solutions} with $q=p$, for any $p\in \big[2,\frac{\bar{p}-\Theta}{(\gamma+2)(1+\bar{\zeta}\Theta)}\big]$, we obtain that
    \begin{align*}
         \mathbb{E}\Big[\big|Y_{t}^{i,N}\big|^{p}\Big] \le &~ C\mathbb{E}\Big[\big|Y_{t}^{i,N}-Y_{\tau_{n}(t)}^{i,N}\big|^{p}\Big] +C\mathbb{E}\Big[\big|Y_{\tau_{n}(t)}^{i,N}\big|^{p}\Big] \notag \\
         \le &~C\mathbb{E}\Big(1+\sup_{t\in[0,T]}\sup_{i\in\mathcal{I}_{N}}\mathbb{E}\Big[\big|Y_{\tau_{n}(t)}^{i,N}\big|^{(\gamma+2)p}\Big]\Big)\Delta t^{\frac{p}{2}} + C\sup_{t\in[0,T]}\sup_{i\in\mathcal{I}_{N}}\mathbb{E}\Big[\big|Y_{\tau_{n}(t)}^{i,N}\big|^{p}\Big] \notag \\
         \le &~C\Big(1+\Big(\mathbb{E}\Big[\big|Y_{0}\big|^{\bar{p}}\Big]\Big)^{\beta}\,\Big).
    \end{align*}
This completes the proof.
\end{proof}
We next recall the following result from \cite[Corollary 4.2]{kumar2021explicit}, which will be instrumental in our proof.

\begin{lemma} \cite[Corollary 4.2]{kumar2021explicit} \label{lem:mean-value-lemma}
Let $h: \mathbb{M}  \rightarrow \mathbb{R}$ be a function such that its derivative $\partial_y h:$ $\mathbb{M} \rightarrow \mathbb{R}^d$ and its measure derivative $\partial_\rho h: \bar{\mathbb{M}} \rightarrow \mathbb{R}^d$ exists. Then, there exists $\theta \in(0,1)$ such that,
\begin{align*}
h\bigg(y, \frac{1}{N} \sum_{j=1}^N \delta_{z^j}\bigg)&-h\bigg(\bar{y}, \frac{1}{N} \sum_{j=1}^N \delta_{\bar{z}^j}\bigg)=\int_{0}^{1} \partial_y h^{\top}\Big(\bar{y}+\theta(y-\bar{y}), \frac{1}{N} \sum_{j=1}^N \delta_{z^j}\Big) (y-\bar{y}) \,\mathrm{d}\theta \\
&+\frac{1}{N} \sum_{k_{1}=1}^N\int_{0}^{1}\partial_\rho h^{\top}\Big(\bar{y}, \frac{1}{N} \sum_{j=1}^N \delta_{\bar{z}^j+\theta(z^j-\bar{z}^j)},\bar{z}^{k_{1}}+\theta\left(z^{k_{1}}-\bar{z}^{k_{1}}\right)\Big) \left(z^{k_{1}}-\bar{z}^{k_{1}}\right)\,\mathrm{d} \theta,
\end{align*}
for all  $y$, $\bar{y}$, $z^j$, $\bar{z}^j$, $z^{k}$, $\bar{z}^{k}\in \mathbb{R}^d $.
\end{lemma}

To proceed with the subsequent analysis, we first introduce the following notation.
\begin{equation}
\begin{aligned}
&e_{s}^{i,N}:=X_{s}^{i,N}-Y_{s}^{i,N}, \; \Delta f_{s}^{X,Y}:=f \left(X_s^{i, N}, \rho_s^{X, N} \right) -f \left(Y_s^{i,N}, \rho_s^{Y,N} \right), \\
&\Delta f^{X,Y;\,\Gamma_{1}}_{s,\tau_{n}(s)}:= f \Big(X_s^{i, N}, \rho_s^{X, N} \Big) -\Gamma_{1}\Big(f\left(Y_{\tau_{n}(s)}^{i,N},\rho_{\tau_{n}(s)}^{Y,N}\right),\Delta t\Big), \\
&\Delta f_{s,\tau_{n}(s)}^{Y,Y} := f \Big(Y_s^{i,N}, \rho_s^{Y,N} \Big)-f\left(Y_{\tau_{n}(s)}^{i,N},\rho_{\tau_{n}(s)}^{Y,N}\right), \\
&\Delta f^{Y,Y;\,\Gamma_{1}}_{\tau_{n}(s)}:= f\Big(Y_{\tau_{n}(s)}^{i,N},\rho_{\tau_{n}(s)}^{Y,N}\Big)-\Gamma_{1}\left(f\left(Y_{\tau_{n}(s)}^{i,N},\rho_{\tau_{n}(s)}^{Y,N}\right),\Delta t\right),   \\
&\Delta g_{j,s,\tau_{n}(s)}^{X,Y;\,\Gamma_{g}}:= g_{j} \Big(X_s^{i, N}, \rho_s^{X, N} \Big)-\Gamma_{g} \Big(\hat{g}_{j}\left(s,Y_{\tau_{n}(s)}^{i,N},\rho_{\tau_{n}(s)}^{Y,N}\right),\Delta t \Big), \\
&\Delta g_{j,s}^{X,Y}:= g_{j} \Big(X_s^{i, N}, \rho_s^{X, N} \Big)-g_{j} \left(Y_s^{i, N}, \rho_s^{Y, N} \right),  \\
&\Delta g_{j,s,\tau_{n}(s)}^{Y,Y;\,\Gamma_{g}}:= g_{j} \Big(Y_s^{i, N}, \rho_s^{Y, N} \Big)-\Gamma_{g} \left(\hat{g}_{j}\left(s,Y_{\tau_{n}(s)}^{i,N},\rho_{\tau_{n}(s)}^{Y,N}\right),\Delta t \right).
\end{aligned}
\end{equation}
Also, we use $\partial_{y}f^{(v)},\partial_{y}g_{j}^{(v)},\partial_{\rho}f^{(v)}$ and $\partial_{\rho}g_{j}^{(v)} $ denote the $v$-th components of the matrix $ \partial_{y}f, \partial_{y}g_{j}, \partial_{\rho}f $ and $ \partial_{\rho}g_{j} $, respectively.
\begin{lemma} \label{lem:g_and_Gamma_g_difference}
Let Assumptions \ref{ass:assumptions_for_MV_coefficients}, \ref{ass:Initial-value}-\ref{ass:Gamma-control-conditions}, and \ref{ass:Coefficient-comparison-conditions-of-Gamma1-Gamma4} hold. Then, for all $p\in[2,\hat{p}]$, there exist constants $C,\beta>0$ such that 
    \begin{align}
        \sum_{j_{1}=1}^{m} \mathbb{E}\left[\int_{0}^{t}\left|\Delta g_{j_{1},s,\tau_{n}(s)}^{Y,Y;\,\Gamma_{g}}\right|^{p}\mathrm{d}s\right] \le C\bigg(1+\Big(\mathbb{E}\Big[\big|Y_{0}\big|^{\bar{p}}\Big]\Big)^{\beta}\bigg) \Delta t^{p}.
    \end{align}
Here, the upper bound $\hat{p}$ for the admissible moment order is given by
$$
\hat{p}:=\left( \frac{1}{(\gamma+2)(\gamma+\epsilon-2)} \wedge\frac{1}{4(\gamma+2)} \wedge  \frac{1}{(\frac{\gamma}{2}+1)(1+\gamma_{2})} \wedge \frac{1}{(\gamma+1)\gamma_{3}} \wedge \frac{1}{(\gamma+2)\gamma_{4}}\right) \times \frac{\bar{p}-\Theta}{1+\bar{\zeta}\Theta},
$$
where $\epsilon>0$, and $\bar{p}, \bar{\zeta}, \Theta, \gamma$ are defined in Lemma \ref{lem:moment-bounds-of-numerical-solution}, 
while $\bar{p}$ comes from Assumption~\ref{ass:assumptions_for_MV_coefficients}, and $\gamma_{2}, \gamma_{3}, \gamma_{4}$ from Assumption~\ref{ass:Coefficient-comparison-conditions-of-Gamma1-Gamma4}, respectively.
\end{lemma}
\begin{proof}
Observe that
\begin{equation*}
    \left|\Delta g_{j_{1},s,\tau_{n}(s)}^{Y,Y;\,\Gamma_{g}}\right|^{p}\le C\sum_{v=1}^{d} \left|\Delta g_{j_{1},s,\tau_{n}(s)}^{(v),Y,Y;\,\Gamma_{g}}\right|^{p},
\end{equation*}
where $\Delta g_{j_{1},s,\tau_{n}(s)}^{(v),Y,Y;\,\Gamma_{g}}$ denotes the $v$-th component of the vector $\Delta g_{j_{1},s,\tau_{n}(s)}^{Y,Y;\,\Gamma_{g}}$. Thus, it follows from $\Gamma_{g}$ in \eqref{eq:Gamma_g} that
 \begin{align} \label{Inter-proce:J} 
 &\mathbb{E}\left[\left|\Delta g_{j_{1},s,\tau_{n}(s)}^{(v),Y,Y;\,\Gamma_{g}}\right|^{p}\right]\notag \\
  \le&~C\mathbb{E}\bigg[ \Big|g_{j_{1}}^{(v)}\left(Y_{s}^{i,N},\rho_{s}^{Y,N}\right)-g_{j_{1}}^{(v)}\left(Y_{\tau_{n}(s)}^{i,N},\rho_{\tau_{n}(s)}^{Y,N}\right)
  -\partial_{y}g_{j_{1}}^{(v)\top}\left(Y_{\tau_{n}(s)}^{i,N},\rho_{\tau_{n}(s)}^{Y,N}\right)\left(Y_{s}^{i,N}-Y_{\tau_{n}(s)}^{i,N}\right) \notag \\
  &~ \qquad \quad -\frac{1}{N}\sum_{k_{1}=1}^{N}\partial_{\rho} g_{j_{1}}^{(v)\top}\left(Y_{\tau_{n}(s)}^{i,N},\rho_{\tau_{n}(s)}^{Y,N},Y_{\tau_{n}(s)}^{k_{1},N}\right)\left(Y_{s}^{k_{1},N}-Y_{\tau_{n}(s)}^{k_{1},N}\right)\Big|^{p}\bigg]\notag \\
  &~+C\mathbb{E} \bigg[\left|g_{j_{1}}^{(v)}\left(Y_{\tau_{n}(s)}^{i,N},\rho_{\tau_{n}(s)}^{Y,N}\right)-\Gamma_{2}\left(g_{j_{1}}^{(v)} \left(Y_{\tau_{n}(s)}^{i, N}, \rho_{\tau_{n}(s)}^{Y, N} \right), \Delta t\right)\right|^{p}\bigg] \notag \\
  &~+C\mathbb{E}\bigg[\Big|\partial_{y}g_{j_{1}}^{(v)\top}\left(Y_{\tau_{n}(s)}^{i,N},\rho_{\tau_{n}(s)}^{Y,N}\right)\left(Y_{s}^{i,N}-Y_{\tau_{n}(s)}^{i,N}\right) \notag \\
  &~\qquad \qquad \qquad 
  -\sum_{j_{2}=1}^{m} \int_{\tau_{n}(s)}^{s} \Gamma_{3}\left(\mathcal{L}_{y}^{j_{2}}~g_{j_{1}}^{(v)}\left(Y_{\tau_{n}(r)}^{i, N}, \rho_{\tau_{n}(r)}^{Y, N} \right),\Delta t\right)   \mathrm{d}W^{i,j_{2}}_{r}\Big|^{p}\bigg] \notag\\
  &~+C\mathbb{E}\bigg[\Big|\frac{1}{N}\sum_{k_{1}=1}^{N}\partial_{\rho} g_{j_{1}}^{(v)\top}\left(Y_{\tau_{n}(s)}^{i,N},\rho_{\tau_{n}(s)}^{Y,N},Y_{\tau_{n}(s)}^{k_{1},N}\right)\left(Y_{s}^{k_{1},N}-Y_{\tau_{n}(s)}^{k_{1},N}\right)   \notag \\
  &\qquad  \qquad \quad -\frac{1}{N} \sum_{k_{1}=1}^{N} \sum_{j_{2}=1}^{m}\int_{\tau_{n}(s)}^{s} \Gamma_{4}\left( \mathcal{L}_{\rho}^{j_{2}}~g_{j_{1}}^{(v)} \left(Y_{\tau_{n}(r)}^{i,N},\rho_{\tau_{n}(r)}^{Y,N},Y_{\tau_{n}(r)}^{k_{1},N}\right) ,\Delta t\right)   \mathrm{d}W^{k_{1},j_{2}}_{r} \Big|^{p}\bigg] \notag \\
 :=&~C\mathbb{E}\left[\left|J_{1}\right|^{p}\right] + C\mathbb{E}\left[\left|J_{2}\right|^{p}\right] + C\mathbb{E}\left[\left|J_{3}\right|^{p}\right] + C\mathbb{E}\left[\left|J_{4}\right|^{p}\right].
 \end{align}   
 
Applying Lemma \ref{lem:mean-value-lemma}, Assumption \ref{ass:Derivative-of-f-and-g-with-respect-to-y-mu}, Young's inequality, and $\rm H\ddot{o}lder$'s inequality, one can obtain
 \begin{align*} 
 &\mathbb{E}\left[\left|J_{1}\right|^{p}\right] \notag \\
  =&~ \mathbb{E} 
  \bigg[ \left|\int_{0}^{1}\left(\partial _{y}g_{j_{1}}^{(v)\top}\left(Y_{\tau_{n}(s)}^{i,N}+\theta \left(Y_{s}^{i,N}-Y_{\tau_{n}(s)}^{i,N}\right),\rho_{s}^{Y,N}\right)
  -\partial_{y}g_{j_{1}}^{(v)\top}\left(Y_{\tau_{n}(s)}^{i,N},\rho_{\tau_{n}(s)}^{Y,N}\right)\right) \left(Y_{s}^{i,N}-Y_{\tau_{n}(s)}^{i,N}\right)\,\mathrm{d}\theta\right.  \notag \\
  &~ \qquad +\frac{1}{N}\sum_{k_{1}=1}^{N} \int_{0}^{1}  \left(\partial_{\rho}  g_{j_{1}}^{(v)\top}\left(Y_{\tau_{n}(s)}^{i,N},\rho_{\tau_{n}(s)+\theta(s-\tau_{n}(s))}^{Y,N},Y_{\tau_{n}(s)}^{k_{1},N}+\theta \left(Y_{s}^{k_{1},N}-Y_{\tau_{n}(s)}^{k_{1},N}\right)\right)\right.\notag  \\
  &~\qquad  \qquad \qquad \qquad \qquad \qquad   \left. \left.-\partial_{\rho}  g_{j_{1}}^{(v)\top}\left(Y_{\tau_{n}(s)}^{i,N},\rho_{\tau_{n}(s)}^{Y,N},Y_{\tau_{n}(s)}^{k_{1},N}\right)\right)\left(Y_{s}^{k_{1},N}-Y_{\tau_{n}(s)}^{k_{1},N}\right)\,\mathrm{d}\theta\right|^{p} \bigg] \notag \\
  \le &~C\bigg(\mathbb{E}\bigg[\Big( 1+\Big|Y_{s}^{i,N}\Big|^{\frac{\gamma}{2}-1}+\Big|Y_{\tau_{n}(s)}^{i,N}\Big|^{\frac{\gamma}{2}-1}\Big)^{2p}\bigg]\bigg)^{\frac{1}{2}} \bigg(\mathbb{E}\bigg[\Big|Y_{s}^{i,N}-Y_{\tau_{n}(s)}^{i,N}\Big|^{4p}\bigg]\bigg)^{\frac{1}{2}} \notag  \\
  &~+ C\mathbb{E}\bigg[\mathbb{W}_{2}^{2p}\left(\rho_{s}^{Y,N},\rho_{\tau_{n}(s)}^{Y,N}\right)\bigg] + \mathbb{E}\left[\left|Y_{s}^{i,N}-Y_{\tau_{n}(s)}^{i,N}\right|^{2p}\right] + C\mathbb{E} \bigg[\frac{1}{N}\sum_{k_{1}=1}^{N} \left|Y_{s}^{k_{1},N}-Y_{\tau_{n}(s)}^{k_{1},N}\right|^{2p} \bigg] \notag \\
  &~+ C\mathbb{E}\bigg[\mathbb{W}_{2}^{2p}\Big(\rho_{\tau_{n}(s)+\theta (s-\tau_{n}(s))}^{Y,N},\rho_{\tau_{n}(s)}^{Y,N}\Big)\bigg].
 \end{align*}
 
According to estimate of $\mathbb{W}_{2}^{2}\left(\rho_{s}^{X,N},\rho_{s}^{Y,N}\right)$ in \cite[Section 2]{goncalo2021simulation}, we further apply Jensen's inequality for $q \ge 2 $ to get
\begin{equation} \label{ineq:W_p-difference}
    \mathbb{W}_{2}^{q}\left(\rho_{s}^{X,N},\rho_{s}^{Y,N}\right) \le \frac{1}{N} \sum_{i=1}^{N} \left|X_{s}^{i,N}-Y_{s}^{i,N}\right|^{q}.
\end{equation}

Combining \eqref{ineq:W_p-difference} with Lemmas \ref{lem:The-difference-between-two-numerical-solutions}, \ref{lem:moment_bound_of_continuous_form} and \ref{lem:moment-bounds-of-numerical-solution}, we deduce that for all $\epsilon>0$ and $p \in [2, \hat{p}]$,
 \begin{align}  \label{Inter-proce:J1}
   \mathbb{E}\left[\left|J_{1}\right|^{p}\right]  
   \le &~C\bigg(\mathbb{E}\bigg[\Big( 1+\Big|Y_{s}^{i,N}\Big|^{\frac{\gamma+\epsilon}{2}-1}+\Big|Y_{\tau_{n}(s)}^{i,N}\Big|^{\frac{\gamma+\epsilon}{2}-1}\Big)^{2p}\bigg]\bigg)^{\frac{1}{2}} \bigg(1+\sup_{i\in \mathcal{I}_{N}}\mathbb{E}\bigg[\Big|Y_{\tau_{n}(s)}^{i,N}\Big|^{4p(\gamma+2)}\bigg]\bigg)^{\frac{1}{2}} \Delta t^{p} \notag  \\
  &~+C\bigg(1+\sup_{i\in \mathcal{I}_{N}}\mathbb{E}\bigg[\Big|Y_{\tau_{n}(s)}^{i,N}\Big|^{2p(\gamma+2)}\bigg]\bigg)^{\frac{1}{2}} \Delta t^{p} \notag \\
  \le&~  C \bigg(1+\Big(\mathbb{E}\Big[\big|Y_{0}\big|^{\bar{p}}\Big]\Big)^{\beta}\bigg) \Delta t^{p}.
 \end{align}

From Assumption \ref{ass:Coefficient-comparison-conditions-of-Gamma1-Gamma4}, together with \eqref{ineq:growth-condition-of-g}, the $\mathbb{W}_2$ estimate \eqref{eq:W_2^p} and Lemma \ref{lem:moment-bounds-of-numerical-solution}, it follows that 
 \begin{align} \label{Inter-proce:J2}
     \mathbb{E}\left[\left|J_{2}\right|^{p}\right]
     \le &~C \mathbb{E} \left[\left|g_{j_{1}}^{(v)}\left(Y_{\tau_{n}(s)}^{i,N},\rho_{\tau_{n}(s)}^{Y,N}\right)\right|^{p\gamma_{2}}\right]\Delta t^{p\delta_{2}} \notag \\
     \le &~C \mathbb{E} \bigg[\Big(1+\left|Y_{\tau_{n}(s)}^{i,N}\right|^{\frac{\gamma}{2}+1}+\mathbb{W}_{2}\left(\rho_{\tau_{n}(s)}^{Y,N},\delta_{0}\right)\Big)^{p\gamma_{2}}\bigg]\Delta t^{p\delta_{2}} 
     \le ~C \bigg(1+\Big(\mathbb{E}\Big[\big|Y_{0}\big|^{\bar{p}}\Big]\Big)^{\beta}\bigg) \Delta t^{p\delta_{2}}.
 \end{align}

Based on \eqref{eq:continuous-version-of-MMS}, one can infer that
 \begin{align*} 
    J_{3} 
   =&~\partial_{y}g_{j_{1}}^{(v)\top}\left(Y_{\tau_{n}(s)}^{i,N},\rho_{\tau_{n}(s)}^{Y,N}\right) \bigg\{ \int_{\tau_{n}(s)}^{s}\Gamma_{1}\left(f \left(Y_{\tau_{n}(r)}^{i, N}, \rho_{\tau_{n}(r)}^{Y, N} \right),\Delta t \right) \mathrm{d}r \notag \\
   &+ \sum_{j_{2}=1}^{m} \int_{\tau_{n}(s)}^{s}\Gamma_{2}\left(g_{j_{2}} \left(Y_{\tau_{n}(r)}^{i, N}, \rho_{\tau_{n}(r)}^{Y, N} \right), \Delta t\right)-g_{j_{2}} \left(Y_{\tau_{n}(r)}^{i, N}, \rho_{\tau_{n}(r)}^{Y, N} \right)  \,\mathrm{d} W^{i,j_{2}}_{r}  \notag \\
   &~+\sum_{j_{2},j_{3}=1}^{m}  \int_{\tau_{n}(s)}^{s} \int_{\tau_{n}(s)}^{r} \Gamma_{3}\left(\mathcal{L}_{y}^{j_{3}}~g_{j_{2}}\left(Y_{\tau_{n}(r_{1})}^{i, N}, \rho_{\tau_{n}(r_{1})}^{Y, N} \right),\Delta t\right) \,\mathrm{d} W^{i,j_{3}}_{r_{1}}\,\mathrm{d} W^{i,j_{2}}_{r}  \notag \\
   &~+ \frac{1}{N}\sum_{j_{2},j_{3}=1}^{m} \sum_{k_{1}=1}^{N} \int_{\tau_{n}(s)}^{s}  \int_{\tau_{n}(s)}^{r} \Gamma_{4}\left( \mathcal{L}_{\rho}^{j_{3}}~g_{j_{2}} \left(Y_{\tau_{n}(r_{1})}^{i,N},\rho_{\tau_{n}(r_{1})}^{Y,N},Y_{\tau_{n}(r_{1})}^{k_{1},N}\right) ,\Delta t\right)\,\mathrm{d} W^{k_{1},j_{3}}_{r_{1}}\,\mathrm{d} W^{i,j_{2}}_{r} \notag \\
   &~+ \sum_{j_{2}=1}^{m} \int_{\tau_{n}(s)}^{s}g_{j_{2}} \left(Y_{\tau_{n}(r)}^{i, N}, \rho_{\tau_{n}(r)}^{Y, N}\right) \,\mathrm{d} W^{i,j_{2}}_{r}  \bigg\}
    ~-\sum_{j_{2}=1}^{m} \int_{\tau_{n}(s)}^{s} \Gamma_{3}\left(\mathcal{L}_{y}^{j_{2}}~g_{j_{1}}^{(v)}\left(Y_{\tau_{n}(r)}^{i, N}, \rho_{\tau_{n}(r)}^{Y, N} \right),\Delta t\right)\,   \mathrm{d}W^{i,j_{2}}_{r}.
\end{align*}
Applying the Cauchy-Schwarz inequality, properties of conditional expectation, $\rm H\ddot{o}lder$'s inequality and moment inequalities, we deduce that
\begin{align*}  
  \mathbb{E}\left[\left|J_{3}\right|^{p}\right] 
   \le &~C \mathbb{E}\Bigg[\left|\partial_{y}g_{j_{1}}^{(v)}\left(Y_{\tau_{n}(s)}^{i,N},\rho_{\tau_{n}(s)}^{Y,N}\right)\right|^{p} ~\bigg\{\Big|\Gamma_{1}\left(f \left(Y_{\tau_{n}(s)}^{i, N}, \rho_{\tau_{n}(s)}^{Y, N} \right),\Delta t \right) \Big|^{p} \Delta t^{p} \notag \\
   &\qquad \quad +\sum_{j_{2}=1}^{m} \Big|\Gamma_{2}\left(g_{j_{2}} \left(Y_{\tau_{n}(s)}^{i, N}, \rho_{\tau_{n}(s)}^{Y, N} \right), \Delta t\right) -g_{j_{2}} \left(Y_{\tau_{n}(s)}^{i, N}, \rho_{\tau_{n}(s)}^{Y, N} \right) \Big|^{p} \Delta t^{\frac{p}{2}} \notag \\
   &\qquad \quad +\sum_{j_{2},j_{3}=1}^{m} \Big| \Gamma_{3}\left(\mathcal{L}_{y}^{j_{3}}~g_{j_{2}}\left(Y_{\tau_{n}(s)}^{i, N}, \rho_{\tau_{n}(s)}^{Y, N} \right),\Delta t\right) \Big|^{p}~\Delta t^{p} \notag \\
   &\qquad \quad +\frac{1}{N}\sum_{j_{2},j_{3}=1}^{m} \sum_{k_{1}=1}^{N} \Big|
    \Gamma_{4}\left( \mathcal{L}_{\rho}^{j_{3}}~g_{j_{2}} \left(Y_{\tau_{n}(s)}^{i,N},\rho_{\tau_{n}(s)}^{Y,N},Y_{\tau_{n}(s)}^{k_{1},N}\right) ,\Delta t\right) \Big|^{p} \Delta t^{p}~\bigg\}\Bigg] \notag  \\
   &~+C\mathbb{E} \bigg[\sum_{j_{2}=1}^{m} \Big|   \mathcal{L}_{y}^{j_{2}}~g_{j_{1}}^{(v)}\left(Y_{\tau_{n}(s)}^{i, N}, \rho_{\tau_{n}(s)}^{Y, N} \right)-\Gamma_{3}\left(\mathcal{L}_{y}^{j_{2}}~g_{j_{1}}^{(v)}\left(Y_{\tau_{n}(s)}^{i, N}, \rho_{\tau_{n}(s)}^{Y, N} \right),\Delta t\right) \Big|^{p}  \Delta t^{\frac{p}{2}}\bigg].
\end{align*}
In view of Assumptions \ref{ass:Gamma-control-conditions} and \ref{ass:Coefficient-comparison-conditions-of-Gamma1-Gamma4}, as well as the growth conditions of $f$, $g$, $\partial_{y}g_{j}$, and $\partial_{\rho}g_{j}$ given in \eqref{ineq:growth-condition-of-f}, \eqref{ineq:growth-condition-of-g}, \eqref{ineq:growth-condition-of-g-y}, and \eqref{ineq:growth-condition-of-g-mu}, respectively, together with Young's inequality, the $\mathbb{W}_{2}$ estimate \eqref{eq:W_2^p}, and Lemma \ref{lem:moment-bounds-of-numerical-solution}, it follows that for all $p \in [2, \hat{p}]$,
\begin{align}  \label{Inter-proce:J3} 
   \mathbb{E}\left[\left|J_{3}\right|^{p}\right] 
   \le
   &~ C \mathbb{E}\left[1+\left|Y_{\tau_{n}(s)}^{i,N}\right|^{\left(\frac{\gamma}{2} +\left(\frac{\gamma}{2}+1\right)\gamma_{2}\right) p}+\mathbb{W}_{2}^{\left(\left(\frac{\gamma}{2}+1\right)\gamma_{2}+1\right)p}\left(\rho_{\tau_{n}(s)}^{Y,N},\delta_{0}\right)\right] \Delta t^{\left(\frac{1}{2}+\delta_{2}\right)p}\notag \\
   &~+ C \mathbb{E}\left[1+\left|Y_{\tau_{n}(s)}^{i,N}\right|^{\left(\frac{3}{2}\gamma + 2\right) p}+\frac{1}{N}\sum_{k_{1}=1}^{N}\left|Y_{\tau_{n}(s)}^{k_{1},N}\right|^{\left(\frac{3}{2}\gamma+2\right)p}+\mathbb{W}_{2}^{\left(\gamma+2\right)p}\left(\rho_{\tau_{n}(s)}^{Y,N},\delta_{0}\right)\right] \Delta t^{p} \notag \\
   &~+ C \mathbb{E}\left[1+\left|Y_{\tau_{n}(s)}^{i,N}\right|^{\left(\gamma+1\right)\gamma_{3} p}+\mathbb{W}_{2}^{\left(\frac{\gamma}{2}+2\right)\gamma_{3}p}\left(\rho_{\tau_{n}(s)}^{Y,N},\delta_{0}\right)\right] \Delta t^{\left(\frac{1}{2}+\delta_{3}\right)p}\notag \\
   \le &~  C \bigg(1+\Big(\mathbb{E}\Big[\big|Y_{0}\big|^{\bar{p}}\Big]\Big)^{\beta}\bigg) \Delta t^{p}.
\end{align}

By further applying \eqref{eq:continuous-version-of-MMS}, we deduce that
 \begin{align*}  
   J_{4}
  =&~ \frac{1}{N}\sum_{k_{1}=1}^{N}\partial_{\rho} g_{j_{1}}^{(v)\top}\left(Y_{\tau_{n}(s)}^{i,N},\rho_{\tau_{n}(s)}^{Y,N},Y_{\tau_{n}(s)}^{k_{1},N}\right) \bigg\{ \int_{\tau_{n}(s)}^{s}\Gamma_{1}\left(f \left(Y_{\tau_{n}(r)}^{k_{1}, N}, \rho_{\tau_{n}(r)}^{Y, N} \right),\Delta t \right) \mathrm{d}r  \notag \\
  &~\qquad + \sum_{j_{2}=1}^{m} \int_{\tau_{n}(s)}^{s}\Gamma_{2}\left(g_{j_{2}} \left(Y_{\tau_{n}(r)}^{k_{1}, N}, \rho_{\tau_{n}(r)}^{Y, N} \right), \Delta t\right)
  -g_{j_{2}} \left(Y_{\tau_{n}(r)}^{k_{1}, N}, \rho_{\tau_{n}(r)}^{Y, N} \right) \, \mathrm{d}W^{k_{1},j_{2}}_{r} \notag \\ 
  &~\qquad + \sum_{j_{2},j_{3}=1}^{m} \int_{\tau_{n}(s)}^{s} \int_{\tau_{n}(s)}^{r}  \Gamma_{3}\left(\mathcal{L}_{y}^{j_{3}}~g_{j_{2}}\left(Y_{\tau_{n}(r_{1})}^{k_{1}, N}, \rho_{\tau_{n}(r_{1})}^{Y, N} \right),\Delta t\right) \,\mathrm{d} W^{k_{1},j_{3}}_{r_{1}} \,\mathrm{d} W^{k_{1},j_{2}}_{r}   \notag \\
  &~\qquad  + \frac{1}{N}\sum_{j_{2},j_{3}=1}^{m} \sum_{k_{2}=1}^{N} \int_{\tau_{n}(s)}^{s}\int_{\tau_{n}(s)}^{r}    \Gamma_{4}\left( \mathcal{L}_{\rho}^{j_{3}}~g_{j_{2}} \left(Y_{\tau_{n}(r_{1})}^{k_{1},N},\rho_{\tau_{n}(r_{1})}^{Y,N},Y_{\tau_{n}(r_{1})}^{k_{2},N}\right) ,\Delta t\right) \,\mathrm{d} W^{k_{2},j_{3}}_{r_{1}} \, \mathrm{d} W^{k_{1},j_{2}}_{r} \notag \\
  &~\qquad + \sum_{j_{2}=1}^{m} \int_{\tau_{n}(s)}^{s} g_{j_{2}} \left(Y_{\tau_{n}(r)}^{k_{1}, N}, \rho_{\tau_{n}(r)}^{Y, N} \right) \mathrm{d} W^{k_{1},j_{2}}_{r} \bigg\} \notag \\
  &~- \frac{1}{N}\sum_{j_{2}=1}^{m} \sum_{k_{1}=1}^{N} \int_{\tau_{n}(s)}^{s} \Gamma_{4}\left( \mathcal{L}_{\rho}^{j_{2}}~g_{j_{1}}^{(v)} \left(Y_{\tau_{n}(r)}^{i,N},\rho_{\tau_{n}(r)}^{Y,N},Y_{\tau_{n}(r)}^{k_{1},N}\right) ,\Delta t\right)   \mathrm{d} W^{k_{1},j_{2}}_{r}.
 \end{align*}
 Utilizing a technique analogous to the one demonstrated in $\mathbb{E}\left[|J_{3}|^{p}\right]$, we deduce that for all $p\in[2,\hat{p}]$, 
 \begin{align} \label{Inter-proce:J4}
    \mathbb{E}\left[\left|J_{4}\right|^{p}\right]
  \le ~C \bigg(1+\Big(\mathbb{E}\Big[\big|Y_{0}\big|^{\bar{p}}\Big]\Big)^{\beta}\bigg) \Delta t^{p}.
\end{align}
This completes the proof by substituting \eqref{Inter-proce:J1}-\eqref {Inter-proce:J4} into \eqref{Inter-proce:J}.
\end{proof}

To facilitate subsequent convergence analysis, we introduce the following notation.
Let
\begin{align} \label{defi:V}
V(\partial_{y}f^{(v)},\Gamma_{g}) :=&\sum_{j_{1}=1}^{m} \partial_{y}f^{(v)\top}\left(Y_{\tau_{n}(s)}^{i,N},\rho_{\tau_{n}(s)}^{Y,N} \right) \int_{\tau_{n}(s)}^{s} \Gamma_{g} \left(\hat{g}_{j_{1}}\left(r,Y_{\tau_{n}(r)}^{i,N},\rho_{\tau_{n}(r)}^{Y,N}\right),\Delta t \right) \,\mathrm{d} W^{i,j_{1}}_{r},   \notag \\
\bar{V}(\partial_{y}f^{(v)},\Gamma_{g}) :=&\sum_{j_{1}=1}^{m} \partial_{y}f^{(v)\top}\left(Y_{\tau_{n}(s)}^{i,N},\rho_{\tau_{n}(s)}^{Y,N} \right) \Gamma_{g} \left(\hat{g}_{j_{1}}\left(s,Y_{\tau_{n}(s)}^{i,N},\rho_{\tau_{n}(s)}^{Y,N}\right),\Delta t \right),  \notag \\
 V(\partial_{y}f^{(v)},\Gamma_{g},\Gamma_{2}) :=&\sum_{j_{1}=1}^{m} \partial_{y}f^{(v)\top}\left(Y_{\tau_{n}(s)}^{i,N},\rho_{\tau_{n}(s)}^{Y,N} \right) \int_{\tau_{n}(s)}^{s} \bigg(\Gamma_{g} \left(\hat{g}_{j_{1}}\left(r,Y_{\tau_{n}(r)}^{i,N},\rho_{\tau_{n}(r)}^{Y,N}\right),\Delta t \right)  \notag \\ 
 &~\qquad  \qquad \qquad    -\Gamma_{2}\left(g_{j_{1}}\left(Y_{\tau_{n}(r)}^{i,N},\rho_{\tau_{n}(r)}^{Y,N}\right),\Delta t\right) \bigg) \, \mathrm{d} W^{i,j_{1}}_{r}.  
\end{align}

\begin{lemma} \label{lem:V_estimate}
Let Assumptions \ref{ass:assumptions_for_MV_coefficients} and \ref{ass:Initial-value}-\ref{ass:Gamma-control-conditions} hold. Then, for any $ q\ge 2$, there exists $C>0$ such that 
\begin{align}
\mathbb{E}\left[\Big|V(\partial_{y}f^{(v)},\Gamma_{g})\Big|^{q}\right] &\le~C\mathbb{E}\big[\Phi\big] \Delta t^{\frac{q}{2}}, \label{ineq:V}\\ 
\mathbb{E}\Big[\Big|\bar{V} (\partial_{y}f^{(v)},\Gamma_{g})\Big|^{q}\Big] \le~ C\mathbb{E}\left[\Phi\right],  \quad &
\mathbb{E}\left[\left|V(\partial_{y}f^{(v)},\Gamma_{g},\Gamma_{2})\right|^{q}\right] \le~C\mathbb{E}\big[\Phi\big] \Delta t^{q},  \label{ineq:V-Gamma_2} 
\end{align}  
where $V(\partial_{y}f^{(v)},\Gamma_{g})$, $\bar{V} (\partial_{y}f^{(v)},\Gamma_{g})$, and $V(\partial_{y}f^{(v)},\Gamma_{g},\Gamma_{2})$  are defined in \eqref{defi:V}, and 
 \begin{equation*}
 \Phi:=\Phi\Big(Y_{\tau_{n}(s)}^{i,N},\!~Y_{\tau_{n}(s)}^{k_{1},N},\!~\mathbb{W}_{2}\big(\rho_{\tau_{n}(s)}^{Y,N},\delta_{0}\big)\Big) 
 =1+\!~\Big|Y_{\tau_{n}(s)}^{i,N}\Big|^{(2\gamma+2)q}+\!~\frac{1}{N}\sum_{k_{1}=1}^{N}\Big|Y_{\tau_{n}(s)}^{k_{1},N}\Big|^{(2\gamma+2)q}+\!~\mathbb{W}_{2}^{\left(\frac{3}{2}\gamma+2\right)q}\big(\rho_{\tau_{n}(s)}^{Y,N},\delta_{0}\big).
 \end{equation*}
\end{lemma}

\begin{proof}
Using the Cauchy-Schwarz inequality, moment inequalities together with \eqref{eq:Gamma_g}, one can deduce that
\begin{align*}
&\mathbb{E}\left[\Big|V(\partial_{y}f^{(v)},\Gamma_{g})\right|^{q}\Big]  \notag \\
\le &~C \sum_{j_{1}=1}^{m} \mathbb{E}\bigg[\left|\partial_{y}f^{(v)}\left(Y_{\tau_{n}(s)}^{i,N},\rho_{\tau_{n}(s)}^{Y,N} \right)\right|^{q}  \mathbb{E}\Big(\int_{\tau_{n}(s)}^{s} \left|\Gamma_{g} \left(\hat{g}_{j_{1}}\left(r,Y_{\tau_{n}(r)}^{i,N},\rho_{\tau_{n}(r)}^{Y,N}\right),\Delta t \right)\right|^{q} \mathrm{d} r \Big \arrowvert \mathcal{F}_{\tau_{n}(s)}\Big)\bigg] \Delta t^{\frac{q}{2}-1} \notag \\
\le &~C\sum_{j_{1}=1}^{m} \mathbb{E}\bigg[\left|\partial_{y}f^{(v)}\left(Y_{\tau_{n}(s)}^{i,N},\rho_{\tau_{n}(s)}^{Y,N} \right)\right|^{q} \bigg\{\,\mathbb{E}\bigg(\int_{\tau_{n}(s)}^{s} \left|\Gamma_{2} \left(g_{j_{1}}\left(Y_{\tau_{n}(r)}^{i,N},\rho_{\tau_{n}(r)}^{Y,N}\right),\Delta t \right)\right|^{q} \,\mathrm{d} r\,\Delta t^{\frac{q}{2}-1} \notag \\
&~\qquad \qquad \quad + \frac{1}{N} \sum_{k_{1}=1}^{N} \sum_{j_{2}=1}^{m}    \int_{\tau_{n}(s)}^{s}  \int_{\tau_{n}(s)}^{r} \Big| \Gamma_{4}\left( \mathcal{L}_{\rho}^{j_{2}}~g_{j_{1}} \left(Y_{\tau_{n}(r_{1})}^{i,N},\rho_{\tau_{n}(r_{1})}^{Y,N},Y_{\tau_{n}(r_{1})}^{k_{1},N}\right) ,\Delta t\right) \Big|^{q}  \,\mathrm{d}r_{1} \,\mathrm{d} r  \, \Delta t^{q-2}  \notag \\
&~\qquad \qquad \quad + \sum_{j_{2}=1}^{m} \int_{\tau_{n}(s)}^{s} \int_{\tau_{n}(s)}^{r} \Big| \Gamma_{3}\left(\mathcal{L}_{y}^{j_{2}}~g_{j_{1}}\left(Y_{\tau_{n}(r_{1})}^{i, N}, \rho_{\tau_{n}(r_{1})}^{Y, N} \right),\Delta t\right)\Big|^{q}  \,\mathrm{d}r_{1} \,\mathrm{d} r \,\Delta t^{q-2}\Big \arrowvert \mathcal{F}_{\tau_{n}(s)}\bigg) \bigg\}\,\bigg] .
\end{align*}
Based on Assumption \ref{ass:Gamma-control-conditions}, \eqref{ineq:growth-condition-of-g}, \eqref{ineq:growth-condition-of-f-y}, \eqref{ineq:growth-condition-of-g-y} and \eqref{ineq:growth-condition-of-g-mu}, one can get
\begin{align}
\mathbb{E}\left[\Big|V(\partial_{y}f^{(v)},\Gamma_{g})\Big|^{q}\right] 
\le &~C\sum_{j_{1}=1}^{m} \mathbb{E}\bigg[\left|\partial_{y}f^{(v)}\left(Y_{\tau_{n}(s)}^{i,N},\rho_{\tau_{n}(s)}^{Y,N} \right)\right|^{q} \bigg\{\, \left|g_{j_{1}}\left(Y_{\tau_{n}(s)}^{i,N},\rho_{\tau_{n}(s)}^{Y,N}\right)\right|^{q}\,\Delta t^{\frac{q}{2}} \notag \\
&~\qquad \qquad \quad + \frac{1}{N} \sum_{k_{1}=1}^{N}\sum_{u=1}^{d}  \Big|\partial_{\rho}g_{j_{1}}^{(u)} \left(Y_{\tau_{n}(s)}^{i,N},\rho_{\tau_{n}(s)}^{Y,N},Y_{\tau_{n}(s)}^{k_{1},N}\right) \Big|^{q} \Big|g \left(Y_{\tau_{n}(s)}^{k_{1},N},\rho_{\tau_{n}(s)}^{Y,N}\right) \Big|^{q}  \, \Delta t^{q}  \notag \\
&~\qquad \qquad \quad + \sum_{u=1}^{d} \Big| \partial_{y}g_{j_{1}}^{(u)}\left(Y_{\tau_{n}(s)}^{i, N}, \rho_{\tau_{n}(s)}^{Y, N} \right)\Big|^{q} \Big|g\left(Y_{\tau_{n}(s)}^{i, N}, \rho_{\tau_{n}(s)}^{Y, N} \right)\Big|^{q} \,\Delta t^{q} \bigg\}\,\bigg] \notag \\
\le &~C\mathbb{E}\bigg[1+\left|Y_{\tau_{n}(s)}^{i,N}\right|^{\left(2\gamma+2\right)q}+\frac{1}{N}\sum_{k_{1}=1}^{N}\left|Y_{\tau_{n}(s)}^{k_{1},N}\right|^{(2\gamma+2)q}+\mathbb{W}_{2}^{\left(\frac{3}{2}\gamma+2\right)q}\left(\rho_{\tau_{n}(s)}^{Y,N},\delta_{0}\right)\bigg] \Delta t^{\frac{q}{2}}.
\end{align}
Hence, \eqref{ineq:V} follows directly, and with minor adjustments to the argument, the estimate in \eqref{ineq:V-Gamma_2} can also be established.
\end{proof}

Similarly, we are also given the following notation. Let 
\begin{align} \label{eq:mathcal_V}
 &\mathcal{V}\left(\partial_{\rho}f^{(v)},\Gamma_{g}\right)  
 :=~\frac{1}{N}\sum_{k_{1}=1}^{N}\partial_{\rho}f^{(v)\top}\Big(Y_{\tau_{n}(s)}^{i,N},\rho_{\tau_{n}(s)}^{Y,N},Y_{\tau_{n}(s)}^{k_{1},N} \Big) \sum_{j_{1}=1}^{m} \int_{\tau_{n}(s)}^{s} \Gamma_{g} \Big(\hat{g}_{j_{1}}\big(r,Y_{\tau_{n}(r)}^{k_{1},N},\rho_{\tau_{n}(r)}^{Y,N}\big),\Delta t \Big)\,\mathrm{d} W^{k_{1},j_{1}}_{r},  \notag  \\
 &\bar{\mathcal{V}}\left(\partial_{\rho}f^{(v)},\Gamma_{g}\right) 
 :=\frac{1}{N}\sum_{k_{1}=1}^{N}\partial_{\rho}f^{(v)\top}\left(Y_{\tau_{n}(s)}^{i,N},\rho_{\tau_{n}(s)}^{Y,N},Y_{\tau_{n}(s)}^{k_{1},N} \right) \sum_{j_{1}=1}^{m} \Gamma_{g} \left(\hat{g}_{j_{1}}\left(s,Y_{\tau_{n}(s)}^{k_{1},N},\rho_{\tau_{n}(s)}^{Y,N}\right),\Delta t \right),  \notag \\ 
 &\mathcal{V}\left(\partial_{\rho}f^{(v)},\Gamma_{g},\Gamma_{2}\right)  
 :=~\frac{1}{N}\sum_{k_{1}=1}^{N}\partial_{\rho}f^{(v)\top}\left(Y_{\tau_{n}(s)}^{i,N},\rho_{\tau_{n}(s)}^{Y,N},Y_{\tau_{n}(s)}^{k_{1},N} \right) \sum_{j_{1}=1}^{m} \int_{\tau_{n}(s)}^{s}\bigg( \Gamma_{g} \left(\hat{g}_{j_{1}}\left(r,Y_{\tau_{n}(r)}^{k_{1},N},\rho_{\tau_{n}(r)}^{Y,N}\right),\Delta t \right)  \notag \\ 
 &~\qquad  \qquad \qquad \qquad  \qquad     -\Gamma_{2}\left(g_{j_{1}}\left(Y_{\tau_{n}(r)}^{k_{1},N},\rho_{\tau_{n}(r)}^{Y,N}\right),\Delta t\right) \bigg)~\mathrm{d} W^{k_{1},j_{1}}_{r}.  
\end{align}

\begin{lemma} \label{lem:mathcal_V_estimate}
Let Assumptions \ref{ass:assumptions_for_MV_coefficients} and \ref{ass:Initial-value}-\ref{ass:Gamma-control-conditions} hold. Then, for any $ q\ge 2$, there exists $C>0$ such that 
\begin{align}
\mathbb{E}\left[\Big|\mathcal{V}(\partial_{\rho}f^{(v)},\Gamma_{g})\Big|^{q}\right] &\le~C\mathbb{E}\big[\Psi\big] \Delta t^{\frac{q}{2}}, \label{ineq:mathcal_V}\\ 
\mathbb{E}\Big[\Big|\bar{\mathcal{V}} (\partial_{\rho}f^{(v)},\Gamma_{g})\Big|^{q}\Big] \le~ C\mathbb{E}\left[\Psi\right],  \quad &
\mathbb{E}\left[\left|\mathcal{V}(\partial_{\rho}f^{(v)},\Gamma_{g},\Gamma_{2})\right|^{q}\right] \le~C\mathbb{E}\big[\Psi\big] \Delta t^{q},  \label{ineq:mathcal_V-Gamma_2} 
\end{align}  
where $\mathcal{V}(\partial_{\rho}f^{(v)},\Gamma_{g})$, $\bar{\mathcal{V}} (\partial_{\rho}f^{(v)},\Gamma_{g})$, and $\mathcal{V}(\partial_{\rho}f^{(v)},\Gamma_{g},\Gamma_{2})$  are defined in \eqref{eq:mathcal_V}, and 
\begin{equation*}
\Psi:=\Psi\Big(Y_{\tau_{n}(s)}^{i,N},\!~Y_{\tau_{n}(s)}^{k_{1},N},\!~\mathbb{W}_{2}\big(\rho_{\tau_{n}(s)}^{Y,N},\delta_{0}\big)\Big) 
= 1\!~+\big|Y_{\tau_{n}(s)}^{i,N}\big|^{\left(2\gamma+3\right)q}\!~+\frac{1}{N}\sum_{k_{1}=1}^{N}\big|Y_{\tau_{n}(s)}^{k_{1},N}\big|^{(2\gamma+3)q}\!~+\mathbb{W}_{2}^{\left(\frac{3\gamma}{2}+3\right)q}\big(\rho_{\tau_{n}(s)}^{Y,N},\delta_{0}\big).
\end{equation*}
\end{lemma}
\begin{proof}
The proof proceeds in the same manner as that of Lemma \ref{lem:V_estimate}, except that 
$\partial_{y}f^{(v)\top}\!\big(Y_{\tau_{n}(s)}^{i,N},\!\rho_{\tau_{n}(s)}^{Y,N}\big)$ is replaced by 
$\frac{1}{N}\!\sum\limits_{k_{1}=1}^{N}\!\partial_{\rho}f^{(v)\top}\!\big(Y_{\tau_{n}(s)}^{i,N},\!\rho_{\tau_{n}(s)}^{Y,N},\!Y_{\tau_{n}(s)}^{k_{1},N}\big) $, and the growth condition for $\partial_{\rho}f^{(v)}\!\big(Y_{\tau_{n}(s)}^{i,N},\!\rho_{\tau_{n}(s)}^{Y,N},\!Y_{\tau_{n}(s)}^{k_{1},N}\big)$ is employed to complete the argument.
\end{proof}

Finally, the following lemma plays an important role in the proof of the main theorem.
\begin{lemma} \label{lem:e_and_f_difference}
Let Assumptions \ref{ass:assumptions_for_MV_coefficients}, \ref{ass:Initial-value}-\ref{ass:Gamma-control-conditions}, and \ref{ass:Coefficient-comparison-conditions-of-Gamma1-Gamma4} hold. Then, for all $p\in[2,\tilde{p}\wedge \hat{p}]$, there exist constants $C,\beta>0$ such that 
\begin{align}
    &\mathbb{E}\left[\int_0^t \left|e^{i,N}_{s}\right|^{p-2} \left\langle e^{i,N}_{s}, \Delta f_{s,\tau_{n}(s)}^{Y,Y}\right\rangle \mathrm{d}s\right] \le C\int_{0}^{t} \sup_{i\in \mathcal{I}_{N}}\sup_{ r \in[0,s]} \mathbb{E}\left[\left|e_{r}^{i,N}\right|^{p}\right] \mathrm{d}s + C\bigg(1+\Big(\mathbb{E}\Big[\big|Y_{0}\big|^{\bar{p}}\Big]\Big)^{\beta}\bigg) \Delta t^{p}.
\end{align}
Here, $\hat{p}$ from Lemma \ref{lem:g_and_Gamma_g_difference} and the upper bound $\tilde{p}$ for the admissible moment order is given by
$$
\tilde{p}:=\left(\frac{1}{(2\gamma+3)(\gamma+1)}\wedge   \frac{1}{(\gamma+1)\gamma_{1}}\right)\times \frac{\bar{p}-\Theta}{1+\bar{\zeta}\Theta},
$$
where $\bar{p}, \bar{\zeta}, \Theta, \gamma$ as defined in Lemma \ref{lem:moment-bounds-of-numerical-solution}, while $ \gamma_{1} $ from Assumption \ref{ass:Coefficient-comparison-conditions-of-Gamma1-Gamma4}.
\end{lemma}
\begin{proof}
One can observe that
\begin{align} \label{Inter-proce:K}
    &\int_0^t \left|e^{i,N}_{s}\right|^{p-2} \left\langle e^{i,N}_{s}, \Delta f_{s,\tau_{n}(s)}^{Y,Y}\right\rangle \mathrm{d}s 
    =\sum_{v=1}^{d}\int_0^t \left|e^{i,N}_{s}\right|^{p-2}  e^{(v),i,N}_{s} \Delta f_{s,\tau_{n}(s)}^{(v),Y,Y} \mathrm{d}s \notag \\
    =&~\sum_{v=1}^{d}\int_0^t \left|e^{i,N}_{s}\right|^{p-2}  e^{(v),i,N}_{s} \bigg\{f^{(v)}\left(Y_{s}^{i,N},\rho_{s}^{Y,N}\right) -f^{(v)}\left(Y_{\tau_{n}(s)}^{i,N},\rho_{\tau_{n}(s)}^{Y,N} \right)  \notag \\
    &~\qquad \qquad \qquad \qquad \qquad \quad  -\partial_{y}f^{(v)\top}\left(Y_{\tau_{n}(s)}^{i,N},\rho_{\tau_{n}(s)}^{Y,N} \right) \left(Y_{s}^{i,N}-Y_{\tau_{n}(s)}^{i,N}\right)\notag \\
    &~\qquad \qquad \qquad \qquad \qquad \quad  -\frac{1}{N}\sum_{k_{1}=1}^{N}\partial_{\rho}f^{(v)\top}\left(Y_{\tau_{n}(s)}^{i,N},\rho_{\tau_{n}(s)}^{Y,N},Y_{\tau_{n}(s)}^{k_{1},N} \right)\left(Y_{s}^{k_{1},N}-Y_{\tau_{n}(s)}^{k_{1},N}\right) \bigg\} \mathrm{d}s \notag \\
    &~+\sum_{v=1}^{d}\int_0^t \left|e^{i,N}_{s}\right|^{p-2}  e^{(v),i,N}_{s} \left(\partial_{y}f^{(v)\top}\left(Y_{\tau_{n}(s)}^{i,N},\rho_{\tau_{n}(s)}^{Y,N} \right) \left(Y_{s}^{i,N}-Y_{\tau_{n}(s)}^{i,N}\right)\right) \mathrm{d}s \notag \\
    &~+\sum_{v=1}^{d}\int_0^t \left|e^{i,N}_{s}\right|^{p-2}  e^{(v),i,N}_{s} \bigg(\frac{1}{N}\sum_{k_{1}=1}^{N}\partial_{\rho}f^{(v)\top}\left(Y_{\tau_{n}(s)}^{i,N},\rho_{\tau_{n}(s)}^{Y,N},Y_{\tau_{n}(s)}^{k_{1},N} \right)\left(Y_{s}^{k_{1},N}-Y_{\tau_{n}(s)}^{k_{1},N}\right)\bigg)\mathrm{d}s  \notag  \\ 
    =: &~K_{1} + K_{2} + K_{3} .  
\end{align}

Based on Lemma \ref{lem:mean-value-lemma}, the Cauchy-Schwarz inequality, Young's inequality, Assumption \ref{ass:Derivative-of-f-and-g-with-respect-to-y-mu}, $\rm H\ddot{o}lder$'s inequality, the $\mathbb{W}_2$ estimate in \eqref{ineq:W_p-difference}, and Lemmas \ref{lem:The-difference-between-two-numerical-solutions}, \ref{lem:moment-bounds-of-numerical-solution} and \ref{lem:moment_bound_of_continuous_form}, for any $p\in[2,\tilde{p}]$, we get
\begin{align} \label{Inter-proce:K1}
    &\mathbb{E}\left[K_{1}\right] \notag \\
    \le&~\mathbb{E}\bigg[\sum_{v=1}^{d}\int_0^t \left|e^{i,N}_{s}\right|^{p-1} \!\bigg\{\!\int_{0}^{1}\!\Big|\partial_{y}f^{(v)}\big(Y_{\tau_{n}(s)}^{i,N}\!+\!\theta\big(Y_{s}^{i,N}\!-\!Y_{\tau_{n}(s)}^{i,N}\big),\rho_{s}^{Y,N}\big) \!-\!\partial_{y}f^{(v)}\!\big(Y_{\tau_{n}(s)}^{i,N},\!\rho_{\tau_{n}(s)}^{Y,N} \big) \Big|\,\Big|Y_{s}^{i,N}\!-\!Y_{\tau_{n}(s)}^{i,N}\Big|\,\mathrm{d}\theta\notag \\
    &~\qquad \qquad \qquad \qquad \qquad + \frac{1}{N}\sum_{k_{1}=1}^{N}\int_{0}^{1}\Big|\partial_{\rho}f^{(v)}\Big(Y_{\tau_{n}(s)}^{i,N},\rho_{\tau_{n}(s)+\theta (s-\tau_{n}(s))}^{Y,N},Y_{\tau_{n}(s)}^{k_{1},N}+\theta\big(Y_{s}^{k_{1},N}-Y_{\tau_{n}(s)}^{k_{1},N}\big) \Big) \notag \\
    &~\qquad \qquad \qquad \qquad \qquad \qquad \qquad \qquad \quad -\partial_{\rho}f^{(v)}\big(Y_{\tau_{n}(s)}^{i,N},\rho_{\tau_{n}(s)}^{Y,N},Y_{\tau_{n}(s)}^{k_{1},N} \big)\Big|\,\Big|Y_{s}^{k_{1},N}-Y_{\tau_{n}(s)}^{k_{1},N}\Big|\,\mathrm{d}\theta\, \bigg\} \, \mathrm{d}s\bigg] \notag \\
    \le &~C\mathbb{E}\bigg[\int_0^t \left|e^{i,N}_{s}\right|^{p} \,\mathrm{d}s \bigg]+ C\bigg(1+\Big(\mathbb{E}\Big[\big|Y_{0}\big|^{\bar{p}}\Big]\Big)^{\beta}\bigg) \Delta t^{p}.
\end{align}

Combining the Milstein-type schemes in \eqref{eq:continuous-version-of-MMS} with the definition of $V(\partial_{y}f^{(v)},\Gamma_{g})$ in \eqref{defi:V}, we can derive that 
\begin{align} \label{Inter-proce:K2}
    K_{2}=&~\sum_{v=1}^{d}\int_0^t \left|e^{i,N}_{s}\right|^{p-2}  e^{(v),i,N}_{s} \bigg(\partial_{y}f^{(v)\top}\big(Y_{\tau_{n}(s)}^{i,N},\rho_{\tau_{n}(s)}^{Y,N} \big) \int_{\tau_{n}(s)}^{s}\Gamma_{1}\Big(f\big(Y_{\tau_{n}(r)}^{i,N},\rho_{\tau_{n}(r)}^{Y,N}\big),\Delta t  \Big)\mathrm{d}r\bigg) \mathrm{d}s \notag \\
    &~+\sum_{v=1}^{d}\int_0^t \left|e^{i,N}_{s}\right|^{p-2}  e^{(v),i,N}_{s}~ V(\partial_{y}f^{(v)},\Gamma_{g}) \,\mathrm{d}s \notag \\
    =:&~K_{2,1}+K_{2,2}.
\end{align}
Applying the Cauchy-Schwarz, Young's, and $\rm H\ddot{o}lder$'s inequalities, along with Assumption \ref{ass:Gamma-control-conditions}, \eqref{ineq:growth-condition-of-f-y} and \eqref{ineq:growth-condition-of-f}, the $\mathbb{W}_{2}$ estimate in \eqref{eq:W_2^p} and Lemma \ref{lem:moment-bounds-of-numerical-solution}, it follows that for all $p\in[2,\tilde{p}]$, 
\begin{align} \label{Inter-proce:K21}
    \mathbb{E}\left[K_{2,1}\right]
    \le &~\mathbb{E}\bigg[\sum_{v=1}^{d}\int_0^t \left|e^{i,N}_{s}\right|^{p-1}  \left|\partial_{y}f^{(v)}\left(Y_{\tau_{n}(s)}^{i,N},\rho_{\tau_{n}(s)}^{Y,N} \right)\right| \Big|\int_{\tau_{n}(s)}^{s}\Gamma_{1}\left(f\left(Y_{\tau_{n}(r)}^{i,N},\rho_{\tau_{n}(r)}^{Y,N}\right),\Delta t  \right)\mathrm{d} r\Big|\mathrm{d} s \bigg] \notag \\
    \le &~C\mathbb{E}\left[\int_0^t \left|e^{i,N}_{s}\right|^{p} \mathrm{d}s\right] + C\bigg(1+\Big(\mathbb{E}\Big[\big|Y_{0}\big|^{\bar{p}}\Big]\Big)^{\beta}\bigg) \Delta t^{p}.
\end{align}
We now proceed to estimate $K_{2,2}$, noting that 
\begin{equation} \label{eq:error-expression}
e^{i,N}_{t} = \int_{0}^{t} \Delta f^{X,Y;\,\Gamma_{1}}_{s,\tau_{n}(s)}\,\mathrm{d}s + \sum_{j=1}^{m}\int_{0}^{t} \Delta g_{j,s,\tau_{n}(s)}^{X,Y;\,\Gamma_{g}}\,\mathrm{d}W_{s}^{i,j}.
\end{equation}
Subsequently, applying $\rm It\hat{o}$'s formula to $|e_{s}^{i,N}|^{p-2}e_{s}^{(v),i,N}$ yields
\begin{align}  \label{ineq:ito_derivation}
    &~\left|e^{i,N}_{s}\right|^{p-2}  e^{(v),i,N}_{s}   \notag \\
    =&~ \left|e^{i,N}_{\tau_{n}(s)}\right|^{p-2}  e^{(v),i,N}_{\tau_{n}(s)} + \int_{\tau_{n}(s)}^{s}\left|e^{i,N}_{r}\right|^{p-2}\Delta f_{r,\tau_{n}(r)}^{(v),X,Y;\,\Gamma_{1}}\,\mathrm{d}r +\sum_{j=1}^{m}\int_{\tau_{n}(s)}^{s}\left|e^{i,N}_{r}\right|^{p-2}\Delta g^{(v),X,Y;\,\Gamma_{g}}_{j,r,\tau_{n}(r)}\,\mathrm{d} W^{i,j}_{r}   \notag \\
    &~+(p-2)\int_{\tau_{n}(s)}^{s}e^{(v),i,N}_{r}\left|e^{i,N}_{r}\right|^{p-4}e^{i,N \top}_{r}\Delta f_{r,\tau_{n}(r)}^{X,Y;\,\Gamma_{1}} \,\mathrm{d}r  \notag  \\
    &~+(p-2)\sum_{j=1}^{m}\int_{\tau_{n}(s)}^{s}e^{(v),i,N}_{r}\left|e^{i,N}_{r}\right|^{p-4}e^{i,N \top}_{r}  \Delta g_{j,r,\tau_{n}(r)}^{X,Y;\,\Gamma_{g}}
     \,\mathrm{d} W^{i,j}_{r}   \notag \\
    &~+\frac{(p-2)(p-4)}{2}\sum_{j=1}^{m}\int_{\tau_{n}(s)}^{s}e^{(v),i,N}_{r}\left|e^{i,N}_{r}\right|^{p-6}  
    \left|e^{i,N \top}_{r}  \Delta g_{j,r,\tau_{n}(r)}^{X,Y;\,\Gamma_{g}} \right|^{2} \,\mathrm{d}r  \notag \\
    &~+\frac{(p-2)}{2}\sum_{j=1}^{m}\int_{\tau_{n}(s)}^{s}e^{(v),i,N}_{r}\left|e^{i,N}_{r}\right|^{p-4}  
    \left|\Delta g_{j,r,\tau_{n}(r)}^{X,Y;\,\Gamma_{g}} \right|^{2} \,\mathrm{d}r  \notag \\
    &~+(p-2)\sum_{j=1}^{m}\sum_{u=1}^{d}\int_{\tau_{n}(s)}^{s}\left|e^{i,N}_{r}\right|^{p-4}  
    \Delta g_{j,r,\tau_{n}(r)}^{(v),X,Y;\,\Gamma_{g}}~ e^{(u),i,N}_{r}~ \Delta g_{j,r,\tau_{n}(r)}^{(u),X,Y;\,\Gamma_{g}}  \,\mathrm{d}r.
\end{align}
Therefore, substituting \eqref{ineq:ito_derivation} into $K_{2,2}$, and applying the Cauchy-Schwarz inequality, we derive that
\begin{align} \label{Inter-proce:K22}
\mathbb{E}\left[K_{2,2}\right] 
\le&~\mathbb{E}\bigg[\sum_{v=1}^{d}\int_0^t \big|e^{i,N}_{\tau_{n}(s)}\big|^{p-2}  e^{(v),i,N}_{\tau_{n}(s)}~V(\partial_{y}f^{(v)},\Gamma_{g}) \,\mathrm{d}s \bigg] \notag \\
&~+C\mathbb{E}\bigg[\sum_{v=1}^{d}\int_0^t \Big(\int_{\tau_{n}(s)}^{s}\big|e^{i,N}_{r}\big|^{p-2}\big|\Delta f_{r,\tau_{n}(r)}^{X,Y;\,\Gamma_{1}}\big|\,\mathrm{d}r \Big)~\big|V(\partial_{y}f^{(v)},\Gamma_{g})\big| \,\mathrm{d}s \bigg] \notag \\
&~+\mathbb{E}\bigg[\sum_{v=1}^{d}\int_0^t \Big(\sum_{j=1}^{m}\int_{\tau_{n}(s)}^{s}\left|e^{i,N}_{r}\right|^{p-2}\Delta g^{(v),X,Y;\,\Gamma_{g}}_{j,r,\tau_{n}(r)}\,\mathrm{d} W^{i,j}_{r} \Big)~V(\partial_{y}f^{(v)},\Gamma_{g}) \,\mathrm{d} s\bigg] \notag \\
&~+\mathbb{E}\bigg[\sum_{v=1}^{d}\int_0^t \Big((p-2)\sum_{j=1}^{m}\int_{\tau_{n}(s)}^{s}e^{(v),i,N}_{r}\left|e^{i,N}_{r}\right|^{p-4}e^{i,N \top}_{r}  \Delta g_{j,r,\tau_{n}(r)}^{X,Y;\,\Gamma_{g}}
     \,\mathrm{d} W^{i,j}_{r} \Big)~V(\partial_{y}f^{(v)},\Gamma_{g}) \,\mathrm{d} s\bigg] \notag \\
&~+C \mathbb{E}\bigg[\sum_{v=1}^{d}\int_0^t \Big(\sum_{j=1}^{m}\int_{\tau_{n}(s)}^{s}\left|e^{i,N}_{r}\right|^{p-3}  
    \big|\Delta g_{j,r,\tau_{n}(r)}^{X,Y;\,\Gamma_{g}} \big|^{2} \,\mathrm{d}r\Big)~\big|V(\partial_{y}f^{(v)},\Gamma_{g})\big| \,\mathrm{d}s\bigg] \notag \\
=:&~\mathbb{E}\left[K_{2,2,1}\right]+\mathbb{E}\left[K_{2,2,2}\right]+\mathbb{E}\left[K_{2,2,3}\right]+\mathbb{E}\left[K_{2,2,4}\right]+\mathbb{E}\left[K_{2,2,5}\right].
\end{align}
In what follows, we evaluate each of the above expressions separately. Using \eqref{defi:V}, together with martingale properties, Young's inequality, the $\mathbb{W}_{2}$ estimate in \eqref{eq:W_2^p}, and Lemmas \ref{lem:V_estimate} and \ref{lem:moment-bounds-of-numerical-solution}, one can show that for all $p\in[2,\tilde{p}]$,  
\begin{align} \label{Inter-proce:K221}
    \mathbb{E}\left[K_{2,2,1}\right] 
    =&~\mathbb{E}\bigg[\sum_{v=1}^{d}\int_0^t \big|e^{i,N}_{\tau_{n}(s)}\big|^{p-2}  e^{(v),i,N}_{\tau_{n}(s)}~V\Big(\partial_{y}f^{(v)},\Gamma_{g},\Gamma_{2}\Big)\,\mathrm{d}s\bigg] \notag \\ 
    \le &~C\mathbb{E}\bigg[\int_{0}^{t}\big|e^{i,N}_{\tau_{n}(s)}\big|^{p}\,\mathrm{d}s\bigg] + C \bigg(1+\sup_{i\in\mathcal{I}_{N}}\mathbb{E}\Big[\Big|Y_{\tau_{n}(s)}^{i,N}\Big|^{2(\gamma+1)p}\Big]\bigg) \Delta t^{p} \notag \\
    \le&~C\mathbb{E}\bigg[\int_0^t \big|e^{i,N}_{\tau_{n}(s)}\big|^{p} \mathrm{d}s\bigg] + C\bigg(1+\Big(\mathbb{E}\Big[\big|Y_{0}\big|^{\bar{p}}\Big]\Big)^{\beta}\bigg) \Delta t^{p}.
\end{align}
By applying the triangle inequality, together with Assumptions \ref{ass:polynomial-growth-of-f} and \ref{ass:Gamma-control-conditions},  one can derive that  
\begin{align*}
    \mathbb{E}\left[K_{2,2,2}\right] 
    \le &~C\mathbb{E}\bigg[\sum_{v=1}^{d} \int_0^t \Big(\int_{\tau_{n}(s)}^{s}\big|e^{i,N}_{r}\big|^{p-2}\big|\Delta f_{r}^{X,Y}\big|\mathrm{d}r\Big)   ~\Big|V(\partial_{y}f^{(v)},\Gamma_{g})\Big| \,\mathrm{d}s \bigg]\notag  \\
    &~+C\mathbb{E}\bigg[\sum_{v=1}^{d}\int_0^t \Big(\int_{\tau_{n}(s)}^{s}\big|e^{i,N}_{r}\big|^{p-2}\big|\Delta f_{r,\tau_{n}(r)}^{Y,Y} \big|\mathrm{d}r\Big) ~\Big|V(\partial_{y}f^{(v)},\Gamma_{g})\Big| \,\mathrm{d}s \bigg
    ]\notag  \\
    &~+C\mathbb{E}\bigg[\sum_{v=1}^{d}\int_0^t \Big
    (\int_{\tau_{n}(s)}^{s}\big|e^{i,N}_{r}\big|^{p-2}\big|\Delta f_{\tau_{n}(r)}^{Y,Y;\,\Gamma_{1}}\big|\mathrm{d}r\Big)  ~\Big|V(\partial_{y}f^{(v)},\Gamma_{g})\Big| \,\mathrm{d}s \bigg]\notag  \\
    \le &~C\mathbb{E}\bigg[\sum_{v=1}^{d} \int_0^t \Big(\int_{\tau_{n}(s)}^{s}\big|e^{i,N}_{r}\big|^{p-1}\Delta t^{\frac{1-p}{p}} \Delta t^{\frac{p-1}{p}} \Big(1+\big|X_r^{i, N}\big|^{\gamma}+\big|Y_r^{i,N}\big|^{\gamma}\Big) \mathrm{d}r \Big)  ~\Big|V(\partial_{y}f^{(v)},\Gamma_{g})\Big| \,\mathrm{d}s \bigg]\notag  \\
    &~+C\mathbb{E}\bigg[ \sum_{v=1}^{d}\int_0^t \int_{\tau_{n}(s)}^{s}\big|e^{i,N}_{r}\big|^{p-2}\,\mathbb{W}_{2}\big(\rho_r^{X, N},\rho_r^{Y,N}\big) \Delta t^{\frac{1-p}{p}} \Delta t^{\frac{p-1}{p}}~\mathrm{d}r~\Big|V(\partial_{y}f^{(v)},\Gamma_{g})\Big| \,\mathrm{d}s \bigg]\notag  \\
    &~+C\mathbb{E}\bigg[\sum_{v=1}^{d} \int_0^t \Big(\int_{\tau_{n}(s)}^{s}\big|e^{i,N}_{r}\big|^{p-2}\Delta t^{\frac{2-p}{p}}\Delta t^{\frac{p-2}{p}} \Big(1+\big|Y_r^{i,N}\big|^{\gamma}+\big|Y_{\tau_{n}(r)}^{i,N}\big|^{\gamma}\Big)\big|Y_r^{i,N}-Y_{\tau_{n}(r)}^{i,N}\big|\mathrm{d}r \Big)  \notag \\ 
    &~\qquad   \qquad \times ~\Big|V(\partial_{y}f^{(v)},\Gamma_{g})\Big| \,\mathrm{d}s \bigg]\notag  \\
    &~+C\mathbb{E}\bigg[\sum_{v=1}^{d} \int_0^t \Big(\int_{\tau_{n}(s)}^{s}\big|e^{i,N}_{r}\big|^{p-2}\Delta t^{\frac{2-p}{p}}\Delta t^{\frac{p-2}{p}}\,\mathbb{W}_{2}\big(\rho_r^{Y,N},\rho_{\tau_{n}(r)}^{Y,N}\big)\mathrm{d}r\Big) ~\left|V(\partial_{y}f^{(v)},\Gamma_{g})\right| \mathrm{d}s \bigg]\notag  \\
    &~+C\mathbb{E}\bigg[ \sum_{v=1}^{d}\int_0^t \Big(\int_{\tau_{n}(s)}^{s}\big|e^{i,N}_{r}\big|^{p-2}\Delta t^{\frac{2-p}{p}}\Delta t^{\frac{p-2}{p}}\big|f\big(Y_{\tau_{n}(r)}^{i,N},\rho_{\tau_{n}(r)}^{Y,N}\big)\big|^{\gamma_{1}}\Delta t^{\delta_{1}}\mathrm{d}r\Big) ~\Big|V(\partial_{y}f^{(v)},\Gamma_{g})\Big| \,\mathrm{d}s \bigg].
\end{align*}
Combining Young's and $\rm H\ddot{o}lder$'s inequalities with the $\mathbb{W}_{2}$ estimate in \eqref{ineq:W_p-difference}, as well as  \eqref{remak:well_posedness_interact_particle}, Lemmas \ref{lem:The-difference-between-two-numerical-solutions}, \ref{lem:V_estimate}, \ref{lem:moment-bounds-of-numerical-solution} and \ref{lem:moment_bound_of_continuous_form}, and \eqref{ineq:growth-condition-of-f}, it follows that for all $p\in[2,\tilde{p}]$,
\begin{align} \label{Inter-proce:K222}
    &\mathbb{E}\left[K_{2,2,2}\right] \notag \\
   \le &~C\mathbb{E}\bigg[\int_0^t\int_{\tau_{n}(s)}^{s}\big|e^{i,N}_{r}\big|^{p} \Delta t^{-1}\,\mathrm{d}r\,\mathrm{d}s\bigg] \notag \\
    &~+ C\mathbb{E}\bigg[\sum_{v=1}^{d}\int_0^t \Big(\int_{\tau_{n}(s)}^{s} \Big(1+\big|X_r^{i, N}\big|^{\gamma}+\big|Y_r^{i,N}\big|^{\gamma}\Big)^{p}\,\mathrm{d}r \Big) \Big|V(\partial_{y}f^{(v)},\Gamma_{g}) \Big|^{p}\, \mathrm{d}s \bigg]\Delta t^{p-1}\notag  \\
    &~+ C\mathbb{E}\bigg[\int_0^t \int_{\tau_{n}(s)}^{s} \mathbb{W}_{2}^{p}\big(\rho_{r}^{X,N},\rho_{r}^{Y,N}\big)\Delta t^{-1}\,\mathrm{d}r \,\mathrm{d}s \bigg] + C\mathbb{E}\bigg[\sum_{v=1}^{d}\int_0^t  \Big|V(\partial_{y}f^{(v)},\Gamma_{g}) \Big|^{p} \mathrm{d}s \bigg]\Delta t^{p}\notag  \\
    &~+C\mathbb{E}\bigg[\sum_{v=1}^{d} \int_0^t \Big(\int_{\tau_{n}(s)}^{s}\Big(1+\big|Y_r^{i,N}\big|^{\gamma}+\big|Y_{\tau_{n}(r)}^{i,N}\big|^{\gamma}\Big)^{\frac{p}{2}}\big|Y_r^{i,N}-Y_{\tau_{n}(r)}^{i,N}\big|^{\frac{p}{2}}\mathrm{d}r \Big)  \Big|V(\partial_{y}f^{(v)},\Gamma_{g}) \Big|^{\frac{p}{2}} \,\mathrm{d}s \bigg] \Delta t^{\frac{p-2}{2}} \notag \\
    &~+C\mathbb{E}\bigg[\sum_{v=1}^{d} \int_0^t \Big(\int_{\tau_{n}(s)}^{s}\mathbb{W}_{2}^{\frac{p}{2}}\big(\rho_r^{Y,N},\rho_{\tau_{n}(r)}^{Y,N}\big)\mathrm{d}r\Big)\Big|V(\partial_{y}f^{(v)},\Gamma_{g}) \Big|^{\frac{p}{2}} \,\mathrm{d}s \bigg]\Delta t^{\frac{p-2}{2}}\notag  \\
    &~+C\mathbb{E}\bigg[ \sum_{v=1}^{d}\int_0^t \Big(\int_{\tau_{n}(s)}^{s}\big|f\big(Y_{\tau_{n}(r)}^{i,N},\rho_{\tau_{n}(r)}^{Y,N}\big)\big|^{\frac{\gamma_{1}p}{2}}\Delta t^{\frac{\delta_{1}p}{2}}\mathrm{d}r\Big)   \Big|V(\partial_{y}f^{(v)},\Gamma_{g}) \Big|^{\frac{p}{2}} \, \mathrm{d}s \bigg] \Delta t^{\frac{p-2}{2}} \notag \\
    \le &~C  \bigg[ \int_0^t \sup_{i\in \mathcal{I}_{N}}\sup_{r\in[0,s]}\mathbb{E}\Big[\big|e^{i,N}_{r}\big|^{p}\Big] \,\mathrm{d}s\bigg] + C\bigg(1+\Big(\mathbb{E}\Big[\big|Y_{0}\big|^{\bar{p}}\Big]\Big)^{\beta}\bigg) \Delta t^{p}.
\end{align}
Based on the definitions of $V(\partial_{y}f^{(v)},\Gamma_{g})$ and $\bar{V}(\partial_{y}f^{(v)},\Gamma_{g})$ in \eqref{defi:V}, together with the $\rm It\hat{o}$ product formula, martingale properties, and the Cauchy-Schwarz inequality, it can be shown that
\begin{align}
\label{proof-eq:middle-proc-K223}
  &\mathbb{E}\left[K_{2,2,3}\right]\\  
  = &~\mathbb{E}\bigg[\sum_{v=1}^{d}\int_{0}^{t}\Big(\sum_{j=1}^{m}\int_{\tau_{n}(s)}^{s}\big|e^{i,N}_{r}\big|^{p-2}\,\Delta g^{(v),X,Y;\,\Gamma_{g}}_{j,r,\tau_{n}(r)}\, \partial_{y}f^{(v)\top}\big(Y_{\tau_{n}(r)}^{i,N},\rho_{\tau_{n}(r)}^{Y,N} \big) \Gamma_{g} \big(\hat{g}_{j}\big(r,Y_{\tau_{n}(r)}^{i,N},\rho_{\tau_{n}(r)}^{Y,N}\big),\Delta t \big) \;d r\Big)\;\mathrm{d}s\bigg] \notag \\
  \le&~\mathbb{E}\bigg[\sum_{v=1}^{d}\int_{0}^{t}\Big(\sum_{j=1}^{m}\int_{\tau_{n}(s)}^{s}\big|e^{i,N}_{r}\big|^{p-2}\big|\Delta g_{j,r,\tau_{n}(r)}^{X,Y;\,\Gamma_{g}} \big|    \big|\bar{V}\big(\partial_{y}f^{(v)},\Gamma_{g}\big) \big|\;d r\Big)\;\mathrm{d}s\bigg].
\end{align}
Therefore, beginning with the triangle inequality and the polynomial growth of g in \eqref{ineq:polynomial-growth-of-g}, an argument analogous to \text{$\mathbb{E}\left[K_{2,2,2}\right]$} is developed: applying Young's inequality, $\rm H\ddot{o}lder$'s inequality, \eqref{remak:well_posedness_interact_particle}, $\mathbb{W}_{2}$ estimate in  \eqref{ineq:W_p-difference}, as well as Lemmas \ref{lem:g_and_Gamma_g_difference}, \ref{lem:V_estimate}, \ref{lem:moment-bounds-of-numerical-solution} and \ref{lem:moment_bound_of_continuous_form}, we conclude that for all $p\in[2,\hat{p}\wedge\tilde{p}]$
\begin{align} \label{Inter-proce:K223}
  \mathbb{E}\left[K_{2,2,3}\right]  
  \le&~\mathbb{E}\bigg[\sum_{v=1}^{d}\int_{0}^{t}\Big(\sum_{j=1}^{m}\int_{\tau_{n}(s)}^{s}\big|e^{i,N}_{r}\big|^{p-2}\big|\Delta g_{j,r}^{X,Y} \big|\,\big|\bar{V}\big(\partial_{y}f^{(v)},\Gamma_{g}\big) \big|\;d r\Big)\;\mathrm{d}s\bigg] \notag \\
  &~+\mathbb{E}\bigg[\sum_{v=1}^{d}\int_{0}^{t}\Big(\sum_{j=1}^{m}\int_{\tau_{n}(s)}^{s}\big|e^{i,N}_{r}\big|^{p-2}\big|\Delta g_{j,r,\tau_{n}(r)}^{Y,Y;\,\Gamma_{g}} \big| \,   \big|\bar{V}\big(\partial_{y}f^{(v)},\Gamma_{g}\big) \big|\;d r\Big)\;\mathrm{d}s\bigg]\notag \\
\le&~C\bigg[\int_0^t \sup_{r\in[0,s]} \sup_{i\in\mathcal{I}_{N}}\mathbb{E}\Big[\big|e^{i,N}_{r}\big|^{p}\big] \,\mathrm{d}s\bigg] + C\bigg(1+\Big(\mathbb{E}\Big[\big|Y_{0}\big|^{\bar{p}}\Big]\Big)^{\beta}\bigg) \Delta t^{p}.
\end{align}
For the term $\mathbb{E}\left[K_{2,2,4}\right]$, similar techniques used in \eqref{proof-eq:middle-proc-K223} help us to show
\begin{align*}
   \mathbb{E}\left[K_{2,2,4}\right]
    \le&~C\mathbb{E}\bigg[\sum_{v=1}^{d}\sum_{j=1}^{m}\int_0^t \int_{\tau_{n}(s)}^{s}e^{(v),i,N}_{r}\big|e^{i,N}_{r}\big|^{p-4}e^{i,N \top}_{r} 
    \Delta g_{j,r,\tau_{n}(r)}^{X,Y;\,\Gamma_{g}} \bar{V}\big(\partial_{y}f^{(v)},\Gamma_{g}\big) \,\mathrm{d}r \,\mathrm{d}s \bigg]\notag \\
    \le&~C\mathbb{E}\bigg[\sum_{v=1}^{d}\sum_{j=1}^{m}\int_{0}^{t}\int_{\tau_{n}(s)}^{s}\big|e^{i,N}_{r}\big|^{p-2}\big|\Delta g_{j,r,\tau_{n}(r)}^{X,Y;\,\Gamma_{g}} \big|\,    \big|\bar{V}\big(\partial_{y}f^{(v)},\Gamma_{g}\big) \big|\;d r\;\mathrm{d}s\bigg].
\end{align*}
By following an identical derivation process as in \eqref{Inter-proce:K223}, we obtain
\begin{align} \label{Inter-proce:K224}
    \mathbb{E}\left[K_{2,2,4}\right] \le C\bigg[\int_0^t \sup_{r\in[0,s]} \sup_{i\in\mathcal{I}_{N}}\mathbb{E}\Big[\big|e^{i,N}_{r}\big|^{p}\big] \,\mathrm{d}s\bigg] + C\bigg(1+\Big(\mathbb{E}\Big[\big|Y_{0}\big|^{\bar{p}}\Big]\Big)^{\beta}\bigg) \Delta t^{p}.
\end{align}
Based on the triangle inequality, the polynomial growth condition of $g$ in \eqref{ineq:polynomial-growth-of-g}, following an approach analogous to $\mathbb{E}\left[K_{2,2,2}\right]$, it follows that for all $p\in[2,\tilde{p}\wedge\hat{p}]$, 
\begin{align} \label{Inter-proce:K225}
   \mathbb{E}\left[K_{2,2,5}\right]
    \le&~C\mathbb{E}\bigg[\sum_{v=1}^{d} \sum_{j=1}^{m} \int_0^t \int_{\tau_{n}(s)}^{s}\big|e^{i,N}_{r}\big|^{p-3}
    \big|\Delta g_{j,r}^{X,Y}\big|^{2} \,\mathrm{d}r \,\big|V(\partial_{y}f^{(v)},\Gamma_{g})\big| \,\mathrm{d}s \bigg]\notag  \\
    &~+C\mathbb{E}\bigg[\sum_{v=1}^{d} \sum_{j=1}^{m} \int_0^t \int_{\tau_{n}(s)}^{s} \big|e^{i,N}_{r}\big|^{p-3}
    \big|\Delta g_{j,r,\tau_{n}(r)}^{Y,Y;\,\Gamma_{g}}\big|^{2} \,\mathrm{d}r\,\big|V(\partial_{y}f^{(v)},\Gamma_{g})\big| \,\mathrm{d}s \bigg]\notag \\ 
    \le&~C\int_0^t\sup_{r\in[0,s]}\sup_{i\in\mathcal{I}_{N}}\mathbb{E}\left[\left|e^{i,N}_{r}\right|^{p}\right] \,\mathrm{d}s + C\bigg(1+\Big(\mathbb{E}\Big[\big|Y_{0}\big|^{\bar{p}}\Big]\Big)^{\beta}\bigg) \Delta t^{p}.
\end{align}
The derivation from \eqref{Inter-proce:K21} to \eqref{Inter-proce:K225} implies that, for all $p\in[2,\tilde{p}\wedge\hat{p}]$, we have
\begin{equation} \label{Inter-proce:all_K2_deriva}
    \mathbb{E}\left[K_{2}\right] \le C\int_0^t\sup_{r\in[0,s]}\sup_{i\in\mathcal{I}_{N}}\mathbb{E}\left[\left|e^{i,N}_{r}\right|^{p}\right] \,\mathrm{d}s + C\bigg(1+\Big(\mathbb{E}\Big[\big|Y_{0}\big|^{\bar{p}}\Big]\Big)^{\beta}\bigg) \Delta t^{p}.
\end{equation}

Using the Milstein-type schemes given in \eqref{eq:continuous-version-of-MMS} together with the definition of $\mathcal{V}(\partial_{\rho}f^{(v)},\Gamma_{g})$ in \eqref{eq:mathcal_V}, we deduce that
\begin{align}
    K_{3}:=&~\sum_{v=1}^{d}\int_0^t \big|e^{i,N}_{s}\big|^{p-2}  e^{(v),i,N}_{s} \Big(\frac{1}{N}\sum_{k_{1}=1}^{N}\partial_{\rho}f^{(v)\top}\left(Y_{\tau_{n}(s)}^{i,N},\rho_{\tau_{n}(s)}^{Y,N},Y_{\tau_{n}(s)}^{k_{1},N} \right)  \notag \\
    &~\qquad \qquad \qquad   \qquad \qquad \qquad \qquad \times \int_{\tau_{n}(s)}^{s}\Gamma_{1}\left(f\left(Y_{\tau_{n}(s)}^{k_{1},N},\rho_{\tau_{n}(s)}^{Y,N}\right),\Delta t  \right)\mathrm{d}r\Big)\mathrm{d}s  \notag  \\ 
    &~+\sum_{v=1}^{d}\int_0^t \left|e^{i,N}_{s}\right|^{p-2}  e^{(v),i,N}_{s} \,\mathcal{V}\Big(\partial_{\rho}f^{(v)},\Gamma_{g}\Big) \,\mathrm{d}s. 
\end{align}
Subsequently, by replacing $\partial_{y}f^{(v)}$ and $V(\partial_{\rho}f^{(v)},\Gamma_{g})$ in equation \eqref{Inter-proce:K2} with $\partial_{\rho}f^{(v)}$ and $\mathcal{V}(\partial_{\rho}f^{(v)},\Gamma_{g})$, respectively, and following a derivation similar to that from \eqref{Inter-proce:K21} to \eqref{Inter-proce:K225} in the proof of $\mathbb{E}\left[K_{2}\right]$, while invoking the growth condition on $\partial_{\rho}f$ in \eqref{ineq:growth-condition-of-f-mu} and Lemma \ref{lem:mathcal_V_estimate}, we obtain
\begin{align} \label{Inter-proce:all_K3_deriva}
\mathbb{E}\left[K_{3}\right] 
\le &~C\mathbb{E}\bigg[  \int_0^t \int_{\tau_{n}(s)}^{s}\left|e^{i,N}_{r}\right|^{p}\Delta t^{-1}\,dr\,ds\bigg] \notag \\
&~+C\bigg[\sum_{v=1}^{d} \int_0^t \int_{\tau_{n}(s)}^{s}\left(\mathbb{E}\Big[\left(1+\left|X_{r}^{i,N}\right|^{(\gamma+1)p}+\left|Y_{r}^{i,N}\right|^{(\gamma+1)p}\right)\Big]\right)^{\frac{\gamma}{\gamma+1}} \notag \\
&~\qquad \qquad \ \qquad \qquad \qquad \qquad \times \Big(1+\sup_{i\in\mathcal{I}_{N}}\mathbb{E}\Big[\big|Y_{\tau_{n}(r)}^{i,N}\big|^{(2\gamma+3)(\gamma+1)p}\Big]\Big)^{\frac{1}{\gamma+1}}\,dr\,ds\bigg]\Delta t^{p-1}  \notag \\
\le&~ C\int_0^t\sup_{r\in[0,s]}\sup_{i\in\mathcal{I}_{N}}\mathbb{E}\left[\left|e^{i,N}_{r}\right|^{p}\right] \,\mathrm{d}s + C\bigg(1+\Big(\mathbb{E}\Big[\big|Y_{0}\big|^{\bar{p}}\Big]\Big)^{\beta}\bigg) \Delta t^{p}.
\end{align}
Combining \eqref{Inter-proce:K1}, \eqref{Inter-proce:all_K2_deriva}, and \eqref{Inter-proce:all_K3_deriva} yields the desired result and completes the proof of the lemma.
\end{proof}

We adopt the proof technique developed in \cite{kumar2019milstein}, which was originally introduced for the tamed Milstein method applied to SDEs with superlinear growth coefficients. 

\begin{proof}[Proof of Theorem
\ref{thm:convergence-result-particle-scheme}]
Based on \eqref{eq:error-expression}, we apply $\rm It\hat{o}$'s formula to $|e_{t}^{i,N}|^{p}$ and use the Cauchy-Schwarz inequality to derive that
\begin{align*}
    \big|e^{i,N}_{t}\big|^{p} 
     \le &~ p \int_0^t \big|e^{i,N}_{s}\big|^{p-2} \Big\langle e^{i,N}_{s}, \Delta f^{X,Y;\,\Gamma_{1}}_{s,\tau_{n}(s)}\Big\rangle \,\mathrm{d}s +  p\sum_{j=1}^{m}\int_0^t \left|e^{i,N}_{s}\right|^{p-2} \left\langle e^{i,N}_{s}, \Delta g_{j,s,\tau_{n}(s)}^{X,Y;\,\Gamma_{g}}  \right \rangle \mathrm{d}W^{i,j}_{s}  \\
     &~+\frac{p(p-1)}{2}\sum_{j=1}^{m} \int_0^t \left|e^{i,N}_{s}\right|^{p-2}\Big|\Delta g_{j,s,\tau_{n}(s)}^{X,Y;\,\Gamma_{g}}\Big|^{2} \mathrm{d}s.
\end{align*}

By taking expectations on both sides and applying the Cauchy-Schwarz and Young’s inequalities (with $\epsilon>0$), together with Assumptions \ref{ass:assumptions_for_MV_coefficients} and \ref{ass:Coefficient-comparison-conditions-of-Gamma1-Gamma4}, Lemmas \ref{lem:g_and_Gamma_g_difference}, \ref{lem:e_and_f_difference}, and \ref{lem:moment-bounds-of-numerical-solution}, the $\mathbb{W}_{2}$ estimate \eqref{eq:W_2^p}, and the growth condition of $f$ in \eqref{ineq:growth-condition-of-f}, we deduce that, for all $p\in[2,\tilde{p}\wedge \hat{p}]$, the following estimate holds:
\begin{align*}
    \mathbb{E}\left[\big|e^{i,N}_{t}\big|^{p}\right]
     \le &~ p \mathbb{E}\bigg[\int_0^t \big|e^{i,N}_{s}\big|^{p-2} \left\langle e^{i,N}_{s},\Delta f^{X,Y;\,\Gamma_{1}}_{s,\tau_{n}(s)}\right\rangle \,\mathrm{d}s\bigg] +\frac{p(p-1)}{2}\sum_{j=1}^{m} \mathbb{E}\bigg[\int_0^t \left|e^{i,N}_{s}\right|^{p-2}\Big|\Delta g_{j,s,\tau_{n}(s)}^{X,Y;\,\Gamma_{g}}\Big|^{2} \,\mathrm{d}s\bigg] \\
     \le&~p \mathbb{E}\Bigg[\int_0^t \left|e^{i,N}_{s}\right|^{p-2} \Big\langle e^{i,N}_{s}, \Delta f_{s}^{X,Y}\Big\rangle \,\mathrm{d}s\bigg] + p \mathbb{E}\left[\int_0^t \left|e^{i,N}_{s}\right|^{p-2} \left\langle e^{i,N}_{s}, \Delta f_{s,\tau_{n}(s)}^{Y,Y}\right\rangle \mathrm{d}s\right] \\
     &~ + p \mathbb{E}\left[\int_0^t \!\left|e^{i,N}_{s}\right|^{p-2}\! \left\langle e^{i,N}_{s}, \Delta f^{Y,Y;\,\Gamma_{1}}_{\tau_{n}(s)}\right\rangle \mathrm{d}s\right] + \frac{p(p-1)}{2}(1+\epsilon)\sum_{j=1}^{m} \mathbb{E}\left[\int_0^t \!\left|e^{i,N}_{s}\right|^{p-2}\!\big|\Delta g_{j,s}^{X,Y} \big|^{2} \mathrm{d}s\right] \\
     &~+C \sum_{j=1}^{m} \mathbb{E}\left[\int_0^t \left|e^{i,N}_{s}\right|^{p-2}\left|\Delta g_{j,s,\tau_{n}(s)}^{Y,Y;\,\Gamma_{g}}\right|^{2} \mathrm{d}s\right] \notag \\
    \le &~ C\int_{0}^{t} \sup_{ r \in[0,s]} \sup_{i\in \mathcal{I}_{N}}\mathbb{E}\left[\left|e_{r}^{i,N}\right|^{p}\right] \mathrm{d}s + C\bigg(1+\Big(\mathbb{E}\Big[\big|Y_{0}\big|^{\bar{p}}\Big]\Big)^{\beta}\bigg) \Delta t^{p} < \infty.
\end{align*}
Applying the \text{Gr\"onwall} inequality yields the desired conclusion.
\end{proof}

\section{Conclusion and discussion}
This work develops a unified framework of Milstein-type discretizations for the IPS associated with MV-SDEs whose drift and diffusion coefficients grow superlinearly.
Within this framework, we establish moment bounds and prove order-one strong convergence under mild regularity conditions, requiring only once differentiability of the coefficients and bounded operator modifications that scale with a negative power of the time step size. Numerical experiments on mean-square errors show slopes near $1$ in log-log plots across the tested step sizes, consistent with the predicted convergence rate for Milstein-type schemes (TanhM, SineM, TamedM and MixM). 
Compared with the classical Milstein method, which diverges in the superlinear growth, the proposed schemes remain stable and accurately reproduce the qualitative behavior of the underlying dynamics. 
The empirical distributions preserve expected bimodal or unimodal structures while maintaining numerical stability even for relatively large step sizes.
A promising direction for future research is the long-time analysis of the proposed schemes, particularly their ergodic behavior and asymptotic stability in high-dimensional particle systems, which remains a mathematically challenging problem.
 
\section*{Acknowledgements}
The authors would like to thank Prof. Xiaojie Wang, Dr. Chenxu Pang, and Bin Yang for their insightful suggestions and helpful discussions.
This work was supported by the National Natural Science Foundation of China (Grant No. 12371417), the Postdoctoral Fellowship Program of the China Postdoctoral Science Foundation (Grant Nos. GZC2024205 and 2025M773122), and the Innovative Project for Graduate Students of Central South University (Grant No. 2023ZZTS0161).
\bibliographystyle{elsarticle-num-names} 
\bibliography{ref}
\appendix
\section{Differentiability of functions defined on measures.} 
\label{appen:differen_measure}
A function $h: \mathscr{P}_2(\mathbb{R}^d) \to \mathbb{R} $ is said to be differentiable at $ \nu_0 \in \mathscr{P}_2(\mathbb{R}^d) $ if there exists a square-integrable $ \mathbb{R}^d $-valued random variable $ Y_0 \in L^2(\tilde{\Omega}; \mathbb{R}^d) $ on an atomless Polish space $(\tilde{\Omega},\tilde{\mathcal{F}},\tilde{\mathbb{P}})$ such that $ \mathscr{L}_{Y_0} = \nu_0 $, and the lifted function $ H(Z) := h(\mathscr{L}_Z) $ is Fr\'echet differentiable at $Y_0 $. The function $ h $ is said to belong to $ C^1(\mathscr{P}_2(\mathbb{R}^{d})) $ if its lifting $ H(X) = h(\mathscr{L}_X) $, defined on $ L^2(\tilde{\Omega}; \mathbb{R}^d) $, lies in $ C^1(L^2(\tilde{\Omega}; \mathbb{R}^d)) $. By the Riesz representation theorem, the derivative $ H'(Y_0) \in (L^2(\tilde{\Omega}; \mathbb{R}^d))^* $ admits a unique representation:
$$
H'(Y_0)(Z) = \tilde{\mathbb{E}} \left[ \langle D H(Y_0), Z \rangle \right], \quad \forall Z \in L^2(\tilde{\Omega}; \mathbb{R}^d),
$$
with $ DH(Y_0) \in L^2(\tilde{\Omega}; \mathbb{R}^d) $. It follows from \cite[Proposition 5.25]{carmona2018probabilistic} that there exists a measurable function $\partial_\rho h(\nu_0): \mathbb{R}^d \to \mathbb{R}^d$, independent of the choice of $Y_0$, such that
$$
DH(Y_0) = \partial_\rho h(\nu_0)(Y_0), \quad \text{and} \quad \int_{\mathbb{R}^d} |\partial_\rho h(\nu_0)(x)|^2 \nu_0(dx) < \infty.
$$
This function $ \partial_\rho h(\nu_0)(\cdot) $ is called the \emph{Lions derivative} of $ h $, and we write $ \partial_\rho h: \mathscr{P}_2(\mathbb{R}^d) \times \mathbb{R}^d \to \mathbb{R}^d $, with $ \partial_\rho h(\nu, x) := \partial_\rho h(\nu)(x) $. When dealing with vector- or matrix-valued functions, this definition is understood to hold for each component individually.

\section{Proof of Lemma \ref{lem:moment-bounds-of-numerical-solution}}
\label{appen:proof-moment-bound-modi-method}
\begin{proof}
Let $\mathcal{R} > 0$ be sufficiently large and define a sequence of decreasing subevents
$$
\Omega_{\mathcal R, k}^{i} = \left\{\omega \in \Omega : \left\vert Y_{t_j}^{i, N}(\omega)\right\vert \leq \mathcal R, \, j = 0,1,\dots, k, i\in \mathcal{I}_{N}\right\},
$$
for $k = 0,1, \dots, n$. We denote the complement of $\Omega_{\mathcal R, k}^{i}$ by $\Omega_{\mathcal R, k}^{i,c}$.

To begin, we show that the boundedness of higher-order moments is preserved within the selected family of subevents. Firstly, if $\bar{p}$ is a sufficiently large even number, we notice that
\begin{align}
\label{Inter-proce:I}
& \mathbb{E} \left[\mathds{1}_{\Omega_{\mathcal{R}, k+1}^{i}} \left\vert Y_{t_{k+1}}^{i, N} \right\vert^{\bar p} \right] \leq \mathbb{E} \left[\mathds{1}_{\Omega_{\mathcal R, k}^{i}} \left|Y_{t_{k+1}}^{i, N}\right|^{\bar p}\right]  
=  \mathbb{E} \left[\mathds{1}_{\Omega_{\mathcal R, k}^{i}} \left|Y_{t_{k+1}}^{i, N} - Y_{t_k}^{i, N} + Y_{t_k}^{i, N} \right|^{\bar p} \right] \notag \\
\leq &~ \mathbb{E} \left[\mathds{1}_{\Omega_{\mathcal R, k}^{i}} \left|Y_{t_k}^{i, N} \right|^{\bar{p}} \right] + \mathbb{E} \left[\mathds{1}_ {\Omega_{\mathcal R ,k}^{i}} \left| Y_{t_k}^{i, N} \right|^{\bar p - 2} \left( \bar p \left \langle Y_{t_k}^{i, N}, Y_{t_{k+1}}^{i, N} - Y_{t_k}^{i, N} \right\rangle + \frac{\bar{p}(\bar{p}-1)}{2} \left|Y_{t_{k+1}}^{i, N} - Y_{t_k}^{i, N} \right|^{2} \right) \right] \notag \\
&~ + C~ \sum_{l = 3}^{\bar p} \mathbb{E} \left[\mathds{1}_{\Omega_{\mathcal R, k}^{i}} \left|Y_{t_k}^{i, N}\right|^{\bar p - l} \left|Y_{t_{k+1}}^{i, N} - Y_{t_k}^{i, N} \right|^{l} \right] \notag \\
:= &~ \mathbb{E} \left[\mathds{1}_{\Omega_{\mathcal R, k}^{i} } \left|Y_{t_k}^{i, N} \right|^{\bar p} \right] + I_{1} + I_{2}.
\end{align}

According to the Milstein-type schemes \eqref{eq:modified_Milstein_method}, and by using the conditional expectation along with the martingale property of $\rm It\hat{o}$ integrals, one can deduce that
\begin{align} \label{Inter-proce:I1}
    I_{1} 
    =&~\bar{p}~ \mathbb{E} \bigg[\mathds{1}_{\Omega_{\mathcal R ,k}^{i}}\Big|Y_{t_{k}}^{i,N}\Big|^{\bar p - 2}  \left \langle Y_{t_{k}}^{i,N}, \Gamma_{1}\left(f \left(Y_{t_{k}}^{i,N}, \rho_{t_{k}}^{Y, N} \right),\Delta t \right) \Delta t \right \rangle \bigg] \notag \\
    &~+\frac{\bar{p}(\bar{p}-1)}{2}~\mathbb{E}\bigg[\mathds{1}_{\Omega_{\mathcal R ,k}^{i}}\left|Y_{t_{k}}^{i,N}\right|^{\bar p - 2}\Big|\Gamma_{1}\left(f \left(Y_{t_{k}}^{i, N}, \rho_{t_{k}}^{Y, N} \right),\Delta t \right) \Delta t\Big|^{2}\bigg] \notag \\ 
    &~+\frac{\bar{p}(\bar{p}-1)}{2}~\mathbb{E}\bigg[\mathds{1}_{\Omega_{\mathcal R ,k}^{i}}\left|Y_{t_{k}}^{i,N}\right|^{\bar p - 2}\Big|\sum_{j_{1}=1}^{m} \Gamma_{2}\left(g_{j_{1}} \left(Y_{t_{k}}^{i, N}, \rho_{t_{k}}^{Y, N} \right), \Delta t\right) \Delta W^{i,j_{1}}_{t_{k}}\Big|^{2}\bigg] \notag \\ 
    &~+\frac{\bar{p}(\bar{p}-1)}{2}~\mathbb{E}\bigg[\mathds{1}_{\Omega_{\mathcal R ,k}^{i}}\left|Y_{t_{k}}^{i,N}\right|^{\bar p - 2}\Big|\sum_{j_{1},j_{2}=1}^{m} \Gamma_{3}\left(\mathcal{L}_{y}^{j_{2}}~g_{j_{1}}\left(Y_{t_{k}}^{i, N}, \rho_{t_{k}}^{Y, N} \right),\Delta t\right) \int_{t_{k}}^{t_{k+1}} \int_{t_{k}}^{s} \,\mathrm{d} W^{i,j_{2}}_{r}  \,\mathrm{d} W^{i,j_{1}}_{s} \Big|^{2}\bigg] \notag \\
    &~+\frac{\bar{p}(\bar{p}-1)}{2}~\mathbb{E}\bigg[\mathds{1}_{\Omega_{\mathcal R ,k}^{i}}\left|Y_{t_{k}}^{i,N}\right|^{\bar p - 2}\Big|\frac{1}{N}\sum_{j_{1},j_{2}=1}^{m} \sum_{k_{1}=1}^{N} \Gamma_{4}\left( \mathcal{L}_{\rho}^{j_{2}}~g_{j_{1}} \left(Y_{t_{k}}^{i,N},\rho_{t_{k}}^{Y,N},Y_{t_{k}}^{k_{1},N}\right) ,\Delta t\right)    \notag \\
    &~\qquad \qquad \qquad \qquad  
    \times \int_{t_{k}}^{t_{k+1}}  \int_{t_{k}}^{s} \,\mathrm{d} W^{k_{1},j_{2}}_{r} \,\mathrm{d} W^{i,j_{1}}_{s} \Big|^{2}\bigg] \notag \\
    &~+\bar{p}(\bar{p}-1)~\mathbb{E} \bigg[\mathds{1}_{\Omega_{\mathcal R ,k}^{i}}\left|Y_{t_{k}}^{i,N}\right|^{\bar p - 2} \Big \langle \sum_{j_{1},j_{2}=1}^{m} \Gamma_{3}\left(\mathcal{L}_{y}^{j_{2}}~g_{j_{1}}\left(Y_{t_{k}}^{i, N}, \rho_{t_{k}}^{Y, N} \right),\Delta t\right) \int_{t_{k}}^{t_{k+1}} \int_{t_{k}}^{s} \,\mathrm{d} W^{i,j_{2}}_{r}  \,\mathrm{d} W^{i,j_{1}}_{s}  ,  \notag \\
    &~\qquad \qquad \qquad \qquad \frac{1}{N}\sum_{j_{1},j_{2}=1}^{m} \sum_{k_{1}=1}^{N} \Gamma_{4}\left( \mathcal{L}_{\rho}^{j_{2}}~g_{j_{1}} \left(Y_{t_{k}}^{i,N},\rho_{t_{k}}^{Y,N},Y_{t_{k}}^{k_{1},N}\right) ,\Delta t\right)\int_{t_{k}}^{t_{k+1}}  \int_{t_{k}}^{s} \,\mathrm{d} W^{k_{1},j_{2}}_{r} \,\mathrm{d} W^{i,j_{1}}_{s}  \Big\rangle\bigg]  \notag \\    :=&~I_{1,1}+I_{1,2}+I_{1,3}+I_{1,4}+I_{1,5}+I_{1,6}.
\end{align}
By applying the Cauchy-Schwarz inequality, Assumption \ref{ass:Coefficient-comparison-conditions-of-f}, the growth condition of $f$ given in \eqref{ineq:growth-condition-of-f}, and Young's inequality, one can derive
\begin{align} \label{Inter-proce:I11}
 I_{1,1} 
 =&~\bar{p}~\mathbb{E} \bigg[\mathds{1}_{\Omega_{\mathcal R ,k}^{i}}\left|Y_{t_{k}}^{i,N}\right|^{\bar p - 2} \left \langle Y_{t_{k}}^{i,N}, \Gamma_{1}\left(f \left(Y_{t_{k}}^{i, N}, \rho_{t_{k}}^{Y, N} \right),\Delta t \right) -f \left(Y_{t_{k}}^{i, N}, \rho_{t_{k}}^{Y, N} \right) \right \rangle  \bigg] \Delta t \notag \\
 &~+\bar{p}~\mathbb{E}\left[\mathds{1}_{\Omega_{\mathcal R ,k}^{i}}\left|Y_{t_{k}}^{i,N}\right|^{\bar p - 2} \left \langle Y_{t_{k}}^{i,N}, f \left(Y_{t_{k}}^{i, N}, \rho_{t_{k}}^{Y, N} \right) \right \rangle  \right] \Delta t \notag \\
 \le &~\underbrace{ C \mathbb{E} \bigg[\mathds{1}_{\Omega_{\mathcal R ,k}^{i}}\Big(1+\left|Y_{t_{k}}^{i,N}\right|^{\bar p - 1+ r_{2}(\gamma +1)}  +  \mathbb{W}_{2}^{\bar{p}-1+r_{2}}\Big( \rho_{t_{k}}^{Y, N},\delta_{0}\Big)\Big) \bigg] \Delta t^{1+r_{1}}}_{I_{1,A}} \notag \\
  &~+\underbrace{\bar{p}~\mathbb{E}\bigg[\mathds{1}_{\Omega_{\mathcal R ,k}^{i}}\left|Y_{t_{k}}^{i,N}\right|^{\bar p - 2} \Big \langle Y_{t_{k}}^{i,N}, f \left(Y_{t_{k}}^{i, N}, \rho_{t_{k}}^{Y, N} \right) \Big \rangle \bigg] \Delta t }_{I_{1,B}} .
\end{align} 
Using Assumption \ref{ass:Gamma-control-conditions}, the growth condition of $f$ in \eqref{ineq:growth-condition-of-f}, and Young's inequality, one can obtain
\begin{align}\label{Inter-proce:I12}
    I_{1,2}
    \le &~C \mathbb{E}\bigg[\mathds{1}_{\Omega_{\mathcal R ,k}^{i}}\left|Y_{t_{k}}^{i,N}\right|^{\bar p - 2} \left|f \left(Y_{t_{k}}^{i, N}, \rho_{t_{k}}^{Y, N} \right) \right|^{2}\bigg] \Delta t^{2}  \notag \\
     \le &~\underbrace{  C \mathbb{E} \bigg[\mathds{1}_{\Omega_{\mathcal R ,k}^{i}}\Big(1+\left|Y_{t_{k}}^{i,N}\right|^{\bar p + 2\gamma} + \mathbb{W}_{2}^{\bar{p}}\Big( \rho_{t_{k}}^{Y, N},\delta_{0}\Big) \Big) \bigg] \Delta t^{2}}_{I_{1,C}}.
\end{align}
In view of the independence of $\Delta W^{i,j_{1}}_{t_{k}}, j_{1}=1,\cdots,m$, the identity $\mathbb{E}\Big(\Big|\Delta W^{i,j_{1}}_{t_{k}}\Big|^{2}\Big) =  \Delta  t$, and Assumption \ref{ass:Gamma-control-conditions}, it follows that
\begin{align} \label{Inter-proce:I13}
    I_{1,3}
    = &~\frac{\bar{p}(\bar{p}-1)}{2}~\mathbb{E}\bigg[\mathds{1}_{\Omega_{\mathcal R ,k}^{i}}\left|Y_{t_{k}}^{i,N}\right|^{\bar p - 2} \sum_{j_{1}=1}^{m} \left| \Gamma_{2}\left(g_{j_{1}} \left(Y_{t_{k}}^{i, N}, \rho_{t_{k}}^{Y, N} \right), \Delta t\right) \right|^{2} \left.\mathbb{E}\left(\left|\Delta W^{i,j_{1}}_{t_{k}} \right|^{2} \right \arrowvert \mathcal{F}_{t_{k}} \right) \bigg]  \notag \\
    \le &~\underbrace{\frac{\bar{p}(\bar{p}-1)}{2}~\mathbb{E}\bigg[\mathds{1}_{\Omega_{\mathcal R ,k}^{i}}\left|Y_{t_{k}}^{i,N}\right|^{\bar p - 2} \sum_{j_{1}=1}^{m} \left| g_{j_{1}} \left(Y_{t_{k}}^{i, N}, \rho_{t_{k}}^{Y, N} \right) \right|^{2} \bigg] \Delta t}_{I_{1,D}}.
\end{align}
Based on the properties of the conditional expectation, the identity $\mathbb{E}\Big(\Big|\int_{t_{k}}^{t_{k+1}} \int_{t_{k}}^{s}\, \mathrm{d} W^{i,j_{2}}_{r} \,\mathrm{d} W^{i,j_{1}}_{s} \Big |^{2}\Big \arrowvert \mathcal{F}_{t_{k}}\Big) = \frac{1}{2}\Delta  t^{2},$ $j_{1},j_{2}=1,2,\dots,m$, Assumption \ref{ass:Gamma-control-conditions}, the growth conditions of $g$ and $\partial_{y}g_{j}$ given in \eqref{ineq:growth-condition-of-g} and \eqref{ineq:growth-condition-of-g-y}, respectively, and Young's inequality, one can infer that 
\begin{align} \label{Inter-proce:I14}
      I_{1,4}
    \le&~C\mathbb{E}\bigg[\mathds{1}_{\Omega_{\mathcal R ,k}^{i}}\left|Y_{t_{k}}^{i,N}\right|^{\bar p - 2} \sum_{j_{1},j_{2}=1}^{m} \left| \Gamma_{3}\left(\mathcal{L}_{y}^{j_{2}}~g_{j_{1}}\left(Y_{t_{k}}^{i, N}, \rho_{t_{k}}^{Y, N} \right),\Delta t\right)\right|^{2} \mathbb{E}\Big( \Big|\int_{t_{k}}^{t_{k+1}} \int_{t_{k}}^{s}\, \mathrm{d} W^{i,j_{2}}_{r}  \,\mathrm{d} W^{i,j_{1}}_{s} \Big|^{2}\Big \arrowvert \mathcal{F}_{t_{k}}\Big ) \bigg]  \notag \\
    \le &~ C \mathbb{E}\left[\mathds{1}_{\Omega_{\mathcal R ,k}^{i}}\left|Y_{t_{k}}^{i,N}\right|^{\bar p - 2} \sum_{j_{1},j_{2}=1}^{m} \left| \Gamma_{3}\left(\mathcal{L}_{y}^{j_{2}}~g_{j_{1}}\left(Y_{t_{k}}^{i, N}, \rho_{t_{k}}^{Y, N} \right),\Delta t\right)\right|^{2} \right] \Delta t ^{2}  \notag \\
    \le &~\underbrace{ C\mathbb{E}\left[\mathds{1}_{\Omega_{\mathcal R ,k}^{i}} \left(1+\left|Y_{t_k}^{i,N}\right|^{\bar{p}+2\gamma}+ \mathbb{W}_{2}^{\bar{p}+\gamma+2}\left( \rho_{t_{k}}^{Y, N},\delta_{0}\right) \right)\right] \Delta t^{2}.}_{I_{1,E}}
\end{align}
By the elementary inequality, the properties of conditional expectation, the $\rm It\hat{o}$ isometry, Assumption \ref{ass:Gamma-control-conditions}, the growth conditions of $g$ and $\partial_{\rho}g_{j}$ in \eqref{ineq:growth-condition-of-g} and \eqref{ineq:growth-condition-of-g-mu}, respectively, and Young's inequality, it follows that  
\begin{align} \label{Inter-proce:I15}
   I_{1,5}
  \le &~C \mathbb{E}\bigg[\mathds{1}_{\Omega_{\mathcal R ,k}^{i}}\left|Y_{t_{k}}^{i,N}\right|^{\bar p - 2} \frac{1}{N}\sum_{j_{1},j_{2}=1}^{m} \sum_{k_{1}=1}^{N} \left|\Gamma_{4}\left( \mathcal{L}_{\rho}^{j_{2}}~g_{j_{1}} \left(Y_{t_{k}}^{i,N},\rho_{t_{k}}^{Y,N},Y_{t_{k}}^{k_{1},N}\right) ,\Delta t\right) \right|^{2}    \notag \\
  &~\qquad \qquad \qquad \qquad  \times \mathbb{E}\Big(\Big|
  \int_{t_{k}}^{t_{k+1}}  \int_{t_{k}}^{s} \,\mathrm{d} W^{k_{1},j_{2}}_{r} \,\mathrm{d} W^{i,j_{1}}_{s} \Big|^{2}\Big\arrowvert\mathcal{F}_{t_{k}}\Big)\bigg] \notag \\
  \le&~C \mathbb{E}\bigg[\mathds{1}_{\Omega_{\mathcal R ,k}^{i}}\left|Y_{t_{k}}^{i,N}\right|^{\bar p - 2} \frac{1}{N}\sum_{j_{1},j_{2}=1}^{m} \sum_{k_{1}=1}^{N}  \left| \mathcal{L}_{\rho}^{j_{2}}~g_{j_{1}} \left(Y_{t_{k}}^{i,N},\rho_{t_{k}}^{Y,N},Y_{t_{k}}^{k_{1},N}\right) \right|^{2} \bigg]\Delta t ^{2}  \notag \\
  \le &~ \underbrace{C\mathbb{E} \bigg[\mathds{1}_{\Omega_{\mathcal R ,k}^{i}}\Big(1+\left|Y_{t_{k}}^{i,N}\right|^{\bar{p}+2\gamma+2}+ \frac{1}{N}\sum_{k_{1}=1}^{N}\left|Y_{t_{k}}^{k_{1},N}\right|^{\bar{p}+2\gamma+2}+ \mathbb{W}_{2}^{\bar{p}+\gamma+2}\left(\rho_{t_{k}}^{Y,N},\delta_{0}\right)\Big)\bigg] \Delta t^{2}}_{I_{1,F}}.
\end{align}
Taking into account the properties of conditional expectation, the independence of multiple stochastic integrals $\int_{t_{k}}^{t_{k+1}} \int_{t_{k}}^{s} \mathrm{d} W^{k_{1},j_{2}}_{r}  \mathrm{d} W^{i,j_{1}}_{s}, \,\text{for}\,i\ne k_{1} \, \text{or} \, j_{1} \ne j_{2},\, i,k\in\mathcal{I}_{N}, j_{1},j_{2}=1,2,\dots,m$, the martingale property of the $\rm It\hat{o}$ integral, the Cauchy-Schwarz inequality, Assumption \ref{ass:Gamma-control-conditions}, the growth conditions of $g$, $\partial_{y}g_{j}$ and $\partial_{\rho}g_{j}$ stated in \eqref{ineq:growth-condition-of-g}, \eqref{ineq:growth-condition-of-g-y}, and \eqref{ineq:growth-condition-of-g-mu}, respectively, as well as Young’s inequality, it follows that 
\begin{align}\label{Inter-proce:I16}
     I_{1,6} 
    =&~\bar{p}\left(\bar{p}-1\right) \mathbb{E} \bigg[\mathds{1}_{\Omega_{\mathcal R ,k}^{i}}\left|Y_{t_{k}}^{i,N}\right|^{\bar p - 2}  \frac{1}{N}  \sum_{j_{1},j_{2}=1}^{m}  \Big\langle  \Gamma_{3}\left(\mathcal{L}_{y}^{j_{2}}~g_{j_{1}}\left(Y_{t_{k}}^{i, N}, \rho_{t_{k}}^{Y, N} \right),\Delta t\right),  \notag \\
    &~\quad  \Gamma_{4}\left( \mathcal{L}_{\rho}^{j_{2}}~g_{j_{1}} \left(Y_{t_{k}}^{i,N},\rho_{t_{k}}^{Y,N},Y_{t_{k}}^{i,N}\right) ,\Delta t\right) \Big \rangle~ \mathbb{E}\Big( \Big(\int_{t_{k}}^{t_{k+1}} \int_{t_{k}}^{s} \,\mathrm{d} W^{i,j_{2}}_{r}  \,\mathrm{d} W^{i,j_{1}}_{s}  \Big)^{2}  \Big \arrowvert \mathcal{F}_{t_{k}}\Big) \bigg]  \notag  \\
    \le &~\underbrace{C \mathbb{E} \bigg[\mathds{1}_{\Omega_{\mathcal R ,k}^{i}}\Big(1+\left|Y_{t_{k}}^{i,N}\right|^{\bar{p}+2\gamma+1}+\mathbb{W}_{2}^{\bar{p}+\frac{3}{2}\gamma+2}\Big(\rho_{t_{k}}^{Y,N},\delta_{0}\Big)\Big) \bigg] \Delta t^{2}.}_{I_{1,G}}
\end{align}

Based on the Milstein-type schemes \eqref{eq:modified_Milstein_method}, the elementary inequality, properties of conditional expectation, moment inequalities, Assumption \ref{ass:Gamma-control-conditions}, the growth conditions of $f$, $g$, $\partial_{y}g_{j}$ and $\partial_{\rho}g_{j}$ given in \eqref{ineq:growth-condition-of-f}, \eqref{ineq:growth-condition-of-g}, \eqref{ineq:growth-condition-of-g-y} and \eqref{ineq:growth-condition-of-g-mu}, respectively, as well as Young's inequality, one can obtain
\begin{align} \label{Inter-proce:I2}
     I_{2} 
     \le &~C \sum_{l = 3}^{\bar p} \mathbb{E} \bigg[\mathds{1}_{\Omega_{\mathcal R, k}^{i}} \left|Y_{t_k}^{i, N}\right|^{\bar p - l} \Big|f \Big(Y_{t_{k}}^{i, N}, \rho_{t_{k}}^{Y, N} \Big) \Big|^{l}\bigg] \Delta t^{l}  \notag \\
     &+~C\sum_{l = 3}^{\bar p} \mathbb{E} \bigg[\mathds{1}_{\Omega_{\mathcal R, k}^{i}} \left|Y_{t_k}^{i, N}\right|^{\bar p - l}  \sum_{j_{1}=1}^{m} \Big| g_{j_{1}} \left(Y_{t_{k}}^{i, N}, \rho_{t_{k}}^{Y, N} \right)\Big|^{l}  ~\mathbb{E}\Big(\Big|\Delta W^{i,j_{1}}_{t_{k}}\Big|^{l}  \Big \arrowvert \mathcal{F}_{t_{k}}\Big) \bigg] \notag \\
     &+~C \sum_{l = 3}^{\bar p} \mathbb{E} \bigg[\mathds{1}_{\Omega_{\mathcal R, k}^{i}} \left|Y_{t_k}^{i, N}\right|^{\bar p - l} \sum_{j_{1},j_{2}=1}^{m} \Big|\mathcal{L}_{y}^{j_{2}}~g_{j_{1}}\left(Y_{t_{k}}^{i, N}, \rho_{t_{k}}^{Y, N} \right)\Big|^{l}  ~\mathbb{E} \Big(\Big|\int_{t_{k}}^{t_{k+1}} \int_{t_{k}}^{s} \,\mathrm{d} W^{i,j_{2}}_{r} \, \mathrm{d} W^{i,j_{1}}_{s} \Big|^{l} \Big \arrowvert \mathcal{F}_{t_{k}} \Big)\bigg] \notag \\
     &+~C\sum_{l = 3}^{\bar p} \mathbb{E} \bigg[\mathds{1}_{\Omega_{\mathcal R, k}^{i}} \left|Y_{t_k}^{i, N}\right|^{\bar p - l} \frac{1}{N}\sum_{j_{1},j_{2}=1}^{m} \sum_{k_{1}=1}^{N} \Big| \mathcal{L}_{\rho}^{j_{2}}~g_{j_{1}} \left(Y_{t_{k}}^{i,N},\rho_{t_{k}}^{Y,N},Y_{t_{k}}^{k_{1},N}\right) \Big|^{l}  \notag \\  
     &  \qquad \qquad \qquad \qquad \qquad \qquad \qquad \qquad \mathbb{E} \Big(\Big|\int_{t_{k}}^{t_{k+1}}  \int_{t_{k}}^{s} \,\mathrm{d} W^{k_{1},j_{2}}_{r} \,\mathrm{d} W^{i,j_{1}}_{s} \Big|^{l} \Big \arrowvert \mathcal{F}_{t_{k}}\Big)~\bigg] \notag \\
     \le &~\underbrace{C\sum_{l = 3}^{\bar p} \mathbb{E} \bigg[\mathds{1}_{\Omega_{\mathcal R, k}^{i}}  \Big(1+\left|Y_{t_k}^{i, N}\right|^{\bar{p}+\frac{\gamma}{2}l}+\mathbb{W}_{2}^{\bar{p}}\Big(\rho_{t_{k}}^{Y,N},\delta_{0}\Big)\Big)   \bigg]  \Delta t^{\frac{l}{2}}}_{I_{2,A}}\notag \\
     &+~\underbrace{C\sum_{l = 3}^{\bar p} \mathbb{E} \bigg[\mathds{1}_{\Omega_{\mathcal R, k}^{i}}   \Big(1+\left|Y_{t_k}^{i,N}\right|^{\bar{p}+\gamma l+l}+\frac{1}{N}\sum_{k_{1}=1}^{N}\Big|Y_{t_{k}}^{k_{1},N}\Big|^{\bar{p}+\gamma l+l}+\mathbb{W}_{2}^{\bar{p}+\frac{\gamma}{2} l+l}\Big(\rho_{t_{k}}^{Y,N},\delta_{0}\Big)\Big) \bigg] \Delta t^{l}}_{I_{2,B}}.
 \end{align}
 
Therefore, combining \eqref{Inter-proce:I1}-\eqref{Inter-proce:I2} and using the elementary inequality together with (A2) in Assumption \ref{ass:assumptions_for_MV_coefficients}, we deduce that
\begin{align} \label{Inter-proce:I1+I2}
    I_{1}+I_{2} \le &~  I_{1,A} +I_{1,B} + I_{1,C} + I_{1,D} + I_{1,E} + I_{1,F} + I_{1,G} + I_{2,A} + I_{2,B} \notag \\
    \le &~C\Delta t + C \mathbb{E} \bigg[\mathds{1}_{\Omega_{\mathcal R ,k}^{i}}\left|Y_{t_{k}}^{i,N}\right|^{\bar p }\bigg]\Delta t + C\mathbb{E}\bigg[\mathds{1}_{\Omega_{\mathcal R ,k}^{i}}  \mathbb{W}_{2}^{\bar{p}}\left( \rho_{t_{k}}^{Y, N},\delta_{0}\right) \bigg] \Delta t \notag \\
    &~+ C \mathbb{E} \bigg[\mathds{1}_{\Omega_{\mathcal R ,k}^{i}}\left|Y_{t_{k}}^{i,N}\right|^{\bar p - 1+ r_{2}(\gamma +1)}\bigg]\Delta t^{1+r_{1}} +C\mathbb{E}\bigg[\mathds{1}_{\Omega_{\mathcal R ,k}^{i}}  \mathbb{W}_{2}^{\bar{p}-1+r_{2}}\left( \rho_{t_{k}}^{Y, N},\delta_{0}\right) \bigg] \Delta t^{1+r_{1}} \notag \\
    &~+C\mathbb{E}\bigg[\mathds{1}_{\Omega_{\mathcal R ,k}^{i}}  \mathbb{W}_{2}^{\bar{p}+\frac{3}{2}\gamma +2}\left( \rho_{t_{k}}^{Y, N},\delta_{0}\right) \bigg] \Delta t^{2} +C \sum_{l=2}^{\bar{p}}\mathbb{E} \bigg[\mathds{1}_{\Omega_{\mathcal R ,k}^{i}} \frac{1}{N} \sum_{k_{1}=1}^{N} \left|Y_{t_{k}}^{k_{1},N}\right|^{\bar p +\gamma l+l}\bigg]\Delta t^{l} \notag \\
    &~+ C\sum_{l=2}^{\bar{p}}\mathbb{E}\bigg[\mathds{1}_{\Omega_{\mathcal R ,k}^{i}}\left|Y_{t_{k}}^{i,N}\right|^{\bar{p}+\gamma l+l}\bigg]\Delta t^{l} + C\sum_{l=3}^{\bar{p}}\mathbb{E}\bigg[\mathds{1}_{\Omega_{\mathcal R ,k}^{i}}\mathbb{W}_{2}^{\bar{p}+\frac{\gamma}{2}l+l}\left( \rho_{t_{k}}^{Y, N},\delta_{0}\right)\bigg]\Delta t^{l}\notag \\
    &~ + C\sum_{l=3}^{\bar{p}}\mathbb{E}\bigg[\mathds{1}_{\Omega_{\mathcal R ,k}^{i}}\left|Y_{t_{k}}^{i,N}\right|^{\bar{p}+\frac{\gamma}{2}l}\bigg]\Delta t^{\frac{l}{2}} + C\sum_{l=3}^{\bar{p}}\mathbb{E}\bigg[\mathds{1}_{\Omega_{\mathcal R ,k}^{i}}\mathbb{W}_{2}^{\bar{p}}\left( \rho_{t_{k}}^{Y, N},\delta_{0}\right)\bigg]\Delta t^{\frac{l}{2}} . 
\end{align}
It follows from \cite[Lemma 2.3]{goncalo2019freidlin}, together with an application of Jensen's inequality, that for any $q \geq 2$,
\begin{align} \label{eq:W_2^p}
\mathbb{W}_2^{q}\left(\rho_{t_k}^{Y, N}, \delta_0\right)=\mathbb{W}_2^{q}\bigg(\frac{1}{N} \sum_{i=1}^N \delta_{Y_{t_{k}}^{i, N},} \delta_0\bigg) \le \frac{1}{N} \sum_{i=1}^N\left|Y_{t_{k}}^{i, N}\right|^{q}.
\end{align}
Upon successive substitution of \eqref{eq:W_2^p} into \eqref{Inter-proce:I1+I2} and then into \eqref{Inter-proce:I}, we conclude that
\begin{align}
    &\mathbb{E} \bigg[\mathds{1}_{\Omega_{\mathcal{R}, k+1}^{i}} \left\vert Y_{t_{k+1}}^{i, N} \right\vert^{\bar p} \bigg] \le \mathbb{E} \left[\mathds{1}_{\Omega_{\mathcal R, k}^{i}} \left|Y_{t_{k+1}}^{i, N}\right|^{\bar p}\right]  \notag \\ 
    \le &~C\Delta t + \left(1+C\Delta t\right) \sup_{i\in\mathcal{I}_{N}}\mathbb{E} \bigg[\mathds{1}_{\Omega_{\mathcal R ,k}^{i}}\left|Y_{t_{k}}^{i,N}\right|^{\bar p} \bigg] + C \sup_{i\in \mathcal{I}_{N}} \mathbb{E} \bigg[\mathds{1}_{\Omega_{\mathcal R ,k}^{i}}\left|Y_{t_{k}}^{i,N}\right|^{\bar p - 1+ r_{2}(\gamma +1)} \bigg]\Delta t^{1+r_{1}}\notag \\
    &~ +C \sum_{l=3}^{\bar{p}} \sup_{i\in\mathcal{I}_{N}}\mathbb{E} \bigg[\mathds{1}_{\Omega_{\mathcal R, k}^{i}} \left|Y_{t_{k}}^{i,N}\right|^{\bar{p}+\frac{\gamma}{2} l}\bigg] \Delta t^{\frac{l}{2}}   + C \sum_{l = 2}^{\bar p} \sup_{i\in\mathcal{I}_{N}} \mathbb{E} \left[\mathds{1}_{\Omega_{\mathcal R, k}^{i}} \left|Y_{t_k}^{i, N}\right|^{\bar p +\gamma l+ l} \right] \Delta t^{l}.
\end{align}
With the choice $\mathcal{R}=\mathcal {R}(\Delta t) = \Delta t^{-1/\Theta(\gamma,r_{1},r_{2})}$, where $\Theta(\gamma,r_{1},r_{2})=\frac{r_{2}(\gamma+1)-1}{r_{1}}\vee 3\gamma$, we deduce that
\begin{align*}
    &\mathds{1}_{\Omega_{\mathcal R ,k}^{i}}\left|Y_{t_{k}}^{i,N}\right|^{\bar p - 1+ r_{2}(\gamma +1)} \Delta t^{1+r_{1}} = \mathds{1}_{\Omega_{\mathcal R ,k}^{i}}\left|Y_{t_{k}}^{i,N}\right|^{\bar p}\Delta t \Big(\mathds{1}_{\Omega_{\mathcal R ,k}^{i}}\left|Y_{t_{k}}^{i,N}\right|^{r_{2}(\gamma+1)-1}\Delta t ^{r_{1}}\Big) \le C \mathds{1}_{\Omega_{\mathcal R ,k}^{i}}\left|Y_{t_{k}}^{i,N}\right|^{\bar p}\Delta t,  \\
    &\mathds{1}_{\Omega_{\mathcal R, k}^{i}} \left|Y_{t_{k}}^{i,N}\right|^{\bar{p}+\frac{\gamma}{2} l} \Delta t^{\frac{l}{2}} = \mathds{1}_{\Omega_{\mathcal R, k}^{i}} \left|Y_{t_{k}}^{i,N}\right|^{\bar{p}}\Delta t \Big(\mathds{1}_{\Omega_{\mathcal R, k}^{i}} \left|Y_{t_{k}}^{i,N}\right|^{\frac{\gamma}{2}l}\Delta t^{\frac{l}{2}-1}\Big) \le C \mathds{1}_{\Omega_{\mathcal R ,k}^{i}}\left|Y_{t_{k}}^{i,N}\right|^{\bar p}\Delta t\;,l=3,\cdots,\bar{p}, \\
    &\mathds{1}_{\Omega_{\mathcal R, k}^{i}} \left|Y_{t_k}^{i, N}\right|^{\bar p +\gamma l+ l} \Delta t^{l} = \mathds{1}_{\Omega_{\mathcal R, k}^{i}} \left|Y_{t_k}^{i, N}\right|^{\bar p} \Delta t \Big(\mathds{1}_{\Omega_{\mathcal R, k}^{i}} \left|Y_{t_k}^{i, N}\right|^{l(\gamma+1)}\Delta t^{l-1}\Big) \le  C \mathds{1}_{\Omega_{\mathcal R ,k}^{i}}\left|Y_{t_{k}}^{i,N}\right|^{\bar p}\Delta t\;,l=2,\cdots,\bar{p}.
\end{align*}
From the above estimations, it follows that
\begin{align*}
    \sup_{i\in \mathcal{I}_{N}} \mathbb{E}\left[\mathds{1}_{\Omega^{i}_{\mathcal{R},k+1}}|Y_{t_{k+1}}^{i,N}|^{\bar{p}}\right] \le ~\sup_{i\in \mathcal{I}_{N}} \mathbb{E}\left[\mathds{1}_{\Omega^{i}_{\mathcal{R},k}}|Y_{t_{k+1}}^{i,N}|^{\bar{p}}\right] \le ~  C \Delta t +  (1+ C \Delta t) \sup_{i\in \mathcal{I}_{N}} \mathbb{E} \left[\mathds{1}_{\Omega^{i}_{\mathcal{R},k}}|Y_{t_{k}}^{i,N}|^{\bar{p}}\right].
\end{align*}
By recursion, we obtain
\begin{align} \label{ineq:Moments-bounded-of-sets}
    \sup_{i\in \mathcal{I}_{N}} \mathbb{E}\left[\mathds{1}_{\Omega^{i}_{\mathcal{R},k+1}}|Y_{t_{k+1}}^{i,N}|^{\bar{p}}\right] \le &~  C \left(1+\mathbb{E}\left[|Y_{0}|^{\bar{p}}\right]\right).
\end{align}

In what follows, we concentrate on estimating $\mathbb{E}\left[\mathds{1}_{\Omega_{\mathcal{R},k}^{i,c}}\left|Y_{t_{k}}^{i,N}\right|^{p}\right]$. To this end, we return to the Milstein schemes \eqref{eq:modified_Milstein_method}, from which we deduce that
\begin{align*} 
   \bigg|Y_{t_{k+1}}^{i, N}\bigg| \le & \bigg|Y_{0}^{i,N}\bigg| + \sum_{l=0}^{k}\bigg|\Gamma_{1}\left(f \left(Y_{t_{l}}^{i, N}, \rho_{t_{l}}^{Y, N} \right),\Delta t \right) \Delta t\bigg|  + \sum_{l=0}^{k}\bigg|\sum_{j_{1}=1}^{m} \Gamma_{2}\left(g_{j_{1}} \left(Y_{t_{l}}^{i, N}, \rho_{t_{l}}^{Y, N} \right), \Delta t\right) \Delta W^{i,j_{1}}_{t_{l}}\bigg|    \notag \\
  &~+\sum_{l=0}^{k} \bigg|\sum_{j_{1},j_{2}=1}^{m} \Gamma_{3}\left(\mathcal{L}_{y}^{j_{2}}~g_{j_{1}}\left(Y_{t_{l}}^{i, N}, \rho_{t_{l}}^{Y, N} \right),\Delta t\right) \int_{t_{l}}^{t_{l+1}} \int_{t_{l}}^{s} \,\mathrm{d} W^{i,j_{2}}_{r} \,\mathrm{d} W^{i,j_{1}}_{s}  \bigg| \notag \\
  &~+\sum_{l=0}^{k}\bigg|\frac{1}{N}\sum_{j_{1},j_{2}=1}^{m} \sum_{k_{1}=1}^{N} \Gamma_{4}\left( \mathcal{L}_{\rho}^{j_{2}}~g_{j_{1}} \Big(Y_{t_{l}}^{i,N},\rho_{t_{l}}^{Y,N},Y_{t_{l}}^{k_{1},N}\right) ,\Delta t\Big)\int_{t_{l}}^{t_{l+1}}  \int_{t_{l}}^{s} \,\mathrm{d} W^{k_{1},j_{2}}_{r}\,\mathrm{d} W^{i,j_{1}}_{s} \bigg|.
\end{align*}
Applying the elementary inequality $\Big(\sum\limits_{l=0}^{k}a_{l}\Big)^{\bar{p}}\le (k+1)^{\bar{p}-1}\sum\limits_{l=0}^{k}a_{l}^{\bar{p}},~a_{l}>0$, properties of conditional expectation, together with Assumption \ref{ass:Gamma-control-conditions} and moment inequalities, we have
\begin{align*} 
\mathbb{E}\Big[\Big|Y_{t_{k+1}}^{i,N}\Big|^{\bar{p}} \Big]
  \le &~C\mathbb{E}\Big[\Big|Y_{0}^{i,N}\Big|^{\bar{p}}\Big] + C\Delta t^{-\bar{p}\zeta_{1}} + C \Delta t^{-\bar{p}\left(\zeta_{2}+\frac{1}{2}\right)}+C\Delta t^{-\bar{p}\zeta_{3}} + C\Delta t^{-\bar{p}\zeta_{4}}   \notag \\
  \le &~C\mathbb{E}\Big[\Big|Y_{0}^{i,N}\Big|^{\bar{p}}\Big]+C\Delta t^{-\bar{p}\bar{\zeta}},
\end{align*}
where $\bar{\zeta} = \zeta_{1} \vee \left(\zeta_{2}+\frac{1}{2}\right) \vee \zeta_{3} \vee \zeta_{4}$. Therefore, we can conclude that
\begin{align} \label{ineq:high order moment of numerical solution}
  \mathbb{E}\Big[\Big|Y_{t_{k}}^{i,N}\Big|^{\bar{p}} \Big]
  \le ~C\mathbb{E}\Big[\Big|Y_{0}^{i,N}\Big|^{\bar{p}}\Big]+C\Delta t^{-\bar{p}\bar{\zeta}}.
\end{align}
Moreover, we observe that
\begin{align*}
    \mathds{1}_{\Omega_{\mathcal{R},k}^{i,c}}=&~1-\mathds{1}_{\Omega^{i}_{\mathcal{R},k}}=1-\mathds{1}_{\Omega^{i}_{\mathcal{R},k-1}}\mathds{1}_{\big|Y_{t_{k}}^{i,N}\big|\le \mathcal{R}} = \mathds{1}_{\Omega_{\mathcal{R},k-1}^{i,c}}+\mathds{1}_{\Omega^{i}_{\mathcal{R},k-1}}\mathds{1}_{\left|Y_{t_{k}}^{i,N}\right|> \mathcal{R}} 
    =\sum_{l=0}^{k}\mathds{1}_{\Omega^{i}_{\mathcal{R},l-1}}\mathds{1}_{\big|Y_{t_{l}}^{i,N}\big|> \mathcal{R}},
\end{align*}
where $\mathds{1}_{\Omega^{i}_{\mathcal{R},-1}}=1$. Based on the previous derivation, and applying $\rm H\ddot{o}lder$'s inequality (with $\frac{1}{p'}+\frac{1}{q'}=1$) together with Chebyshev's inequality, it follows that for $p \geq 2$, 
\begin{align}   \label{ineq:Moment-Bounded-Estimation-of-Set-Complements}
\mathbb{E}\Big[\mathds{1}_{\Omega_{\mathcal{R},k}^{i,c}}\left|Y_{t_{k}}^{i,N}\right|^{p}\Big]
    = & \sum_{l=0}^{k} \mathbb{E} \Big[\mathds{1}_{\Omega^{i}_{\mathcal{R},l-1}}\mathds{1}_{|Y_{t_{l}}^{i,N}|> \mathcal{R}} \left|Y_{t_{k}}^{i,N}\right|^{p}\Big] \notag \\
    \le &~\sum_{l=0}^{k} \Big(\mathbb{E}\Big[\left|Y_{t_{k}}^{i,N}\right|^{p\cdot p'}\Big]\Big)^{\frac{1}{p'}}\Big(\mathbb{E}\Big[\mathds{1}_{\Omega^{i}_{\mathcal{R},l-1}}\mathds{1}_{|Y_{t_{l}}^{i,N}|> \mathcal{R}}\Big]\Big)^{\frac{1}{q'}}  \notag \\
    \le &~\Big(\mathbb{E}\Big[\left|Y_{t_{k}}^{i,N}\right|^{p\cdot p'}\Big]\Big)^{\frac{1}{p'}} \sum_{l=0}^{k} \frac{\Big(\mathbb{E}\Big[\mathds{1}_{\Omega^{i}_{\mathcal{R},l-1}}\left|Y_{t_{l}}^{i,N}\right|^{\bar{p}}\Big]\Big)^{\frac{1}{q'}}}{(\mathcal{R})^{\bar{p}/q'}},
\end{align}
where $q'=\frac{\bar{p}}{ (\bar{\zeta}p+1)\Theta} >1$, as $p \leq \frac{\bar{p}-\Theta}{1+ \bar{\zeta}\Theta}$.

Since $p \leq \frac{\bar{p}-\Theta}{1+ \bar{\zeta}\Theta}$, we deduce that $pp'\le \bar{p}$. Applying $\mathrm{H\ddot{o}lder}$'s inequality and \eqref{ineq:high order moment of numerical solution}, we arrive at
\begin{align} \label{ineq:PP'-th moment of Y_{tn}}
\Big(\mathbb{E}\Big[\big|Y_{t_{k}}^{i,N}\big|^{p\cdot p'}\Big]\Big)^{\frac{1}{p'}} \le \Big(\mathbb{E}\Big[\big|Y_{t_{k}}^{i,N}\big|^{p\cdot p'\cdot \frac{\bar{p}}{pp'}}\Big]\Big)^{\frac{pp'}{\bar{p}}\frac{1}{p'}}
  \le C\Big(1+\mathbb{E}\Big[\big|Y_{0}\big|^{\bar{p}}\Big]\Big)^{\frac{p}{\bar{p}}} + C\Delta t^{-\bar{\zeta}p}.
\end{align}
With $\mathcal{R}(\Delta t) = \Delta t^{-1/\Theta(\gamma,r_{1},r_{2})}$, substituting of \eqref{ineq:PP'-th moment of Y_{tn}} and \eqref{ineq:Moments-bounded-of-sets} into \eqref{ineq:Moment-Bounded-Estimation-of-Set-Complements}  yields
\begin{align}  \label{ineq:Bounded-moment-estimates-of-complements}
    \mathbb{E}\Big[\mathds{1}_{\Omega_{\mathcal{R},k}^{i,c}}\big|Y_{t_{k}}^{i,N}\big|^{p}\Big] 
    \le &~C(k+1)\Delta t^{\bar{\zeta} p+1}\Big(C\Delta t^{-\bar{\zeta} p}+C\Big(1+\mathbb{E}\Big[\big|Y_{0}\big|^{\bar{p}}\Big]\Big)^{\frac{p}{\bar{p}}}\Big)\cdot \Big[C\Big(1+\mathbb{E}\Big[\big|Y_{0}\big|^{\bar{p}}\Big]\Big)\Big]^{\frac{1}{q'}} \notag \\
    \le&~ C\Big(1+\mathbb{E}\Big[\big|Y_{0}\big|^{\bar{p}}\Big]\Big)^{\frac{p}{\bar{p}}+\frac{1}{q'}}.
\end{align}
Combining $\rm H\ddot{o}lder$'s inequality, \eqref{ineq:Moments-bounded-of-sets} and \eqref{ineq:Bounded-moment-estimates-of-complements}, we show
\begin{align*}
    \sup_{i\in\mathcal{I}_{N}}\mathbb{E}\Big[\big|Y_{t_{k}}^{i,N}\big|^{p}\Big] 
    =&~\sup_{i\in\mathcal{I}_{N}}\mathbb{E}\Big[\mathds{1}_{\Omega_{\mathcal{R},k}}\big|Y_{t_{k}}^{i,N}\big|^{p}\Big]+\sup_{i\in\mathcal{I}_{N}}\mathbb{E}\Big[\mathds{1}_{\Omega_{\mathcal{R},k}^{c}}\big|Y_{t_{k}}^{i,N}\big|^{p}\Big]  \\
    \le &~\Big(\sup_{i\in\mathcal{I}_{N}}\mathbb{E}\Big[\mathds{1}_{\Omega_{\mathcal{R},k}}\big|Y_{t_{k}}^{i,N}\big|^{\bar{p}}\,\Big]\Big)^{\frac{p}{\bar{p}}}+\sup_{i\in\mathcal{I}_{N}}\mathbb{E}\Big[\mathds{1}_{\Omega_{\mathcal{R},k}^{c}}\big|Y_{t_{k}}^{i,N}\big|^{p}\Big]   \\
    \le &~C\Big(1+\Big(\mathbb{E}\Big[\big|Y_{0}\big|^{\bar{p} }\Big]\Big)^{\beta}\,\Big),
\end{align*}
where $\beta >0 $. Then, using $\rm H\ddot{o}lder$'s inequality, the estimate \eqref{esti:moment_boundness} can be extended to all $p\in [2, \frac{\bar{p}-\Theta}{1+ \bar{\zeta}\Theta}]$, including odd and non-integer values.
\end{proof}


\vspace{-1.5em}
\section{A summary of key parameters}
\label{appendix:summary-of-parameters}
\vspace{-1.2em}
\begin{table}[H]
\centering
\renewcommand{\arraystretch}{1.3}
\scalebox{0.9}{
\begin{tabular}{|c|c|c|p{8cm}|}
\hline
\textbf{Parameter} & \textbf{Type / Range} & \textbf{Places} & \textbf{Comments} \\
\hline
$\bar{p}$ & enough large   & Assumption \ref{ass:assumptions_for_MV_coefficients}  & This parameter facilitates high-order moment estimation and determines the initial moment exponent;  \\
\hline
$p^{*}$ & $>1$ & Assumption \ref{ass:assumptions_for_MV_coefficients}  & A parameter introduced in the coupled monotonicity condition and used in the convergence analysis; \\
\hline
$\gamma$ & $\ge 2$ & Assumption \ref{ass:polynomial-growth-of-f} &The polynomial growth order of the drift coefficient; \\
\hline  
$\zeta_{l},l=1,2,3,4$ &  $>0$ & Assumption \ref{ass:Gamma-control-conditions}  &Negative powers of $\Delta t$ used to control the mappings in the numerical scheme;  \\
\hline
$r_{1},r_{2}$ & $>0$  & Assumption \ref{ass:Coefficient-comparison-conditions-of-f} & Exponents of $\Delta t$ used to control the difference between $\Gamma_1$ and $f$ ;  \\
\hline
$\Theta$ & $\frac{r_{2}(\gamma+1)-1}{r_{1}}\vee 3\gamma$ & Lemma \ref{lem:moment-bounds-of-numerical-solution} & A parameter used to control the boundedness of the numerical solution on shrinking subsets;  \\
\hline 
$\beta$ & $>0$ & Lemma \ref{lem:moment-bounds-of-numerical-solution} & The exponent $\beta$ of the $\bar{p}$-th moment of the initial value; \\
\hline
$\bar{\zeta}$ & $\zeta_{1}\vee(\zeta_{2}+\frac{1}{2})\vee\zeta_{3}\vee\zeta_{4}$ &  Lemma \ref{lem:moment-bounds-of-numerical-solution} & The choice of larger values for the parameters $\zeta_{i}$; \\
\hline
$\delta_{l},\gamma_{l},~l=1,2,3,4$ & $\delta_{1},\delta_{2},\gamma_{l}\ge 1,~\delta_{3},\delta_{4}\ge \frac{1}{2}$  & Assumption \ref{ass:Coefficient-comparison-conditions-of-Gamma1-Gamma4}  & Parameters controlling the difference between a function and its mapping;  \\
\hline
$q$ & $\ge2$ & Lemma \ref{lem:The-difference-between-two-numerical-solutions}, \ref{lem:V_estimate}, \ref{lem:mathcal_V_estimate} & A free parameter greater than 2 ; \\
\hline
$\hat{p},~\tilde{p}$ & $\ge2$ &Lemma \ref{lem:g_and_Gamma_g_difference} & Admissible upper bounds for the moments of the numerical solution. \\
\hline
\end{tabular}
}
\caption{Summary of parameters}
\label{tab:parameter-overview}
\end{table}
\end{document}